\documentclass[psamsfonts]{amsart}

\usepackage{amssymb,amsfonts}
\usepackage[all,arc]{xy}
\usepackage{enumerate}
\usepackage{mathrsfs}
\usepackage{mathtools}
\usepackage{amsxtra}
\usepackage{hyperref}

\newcommand{\QQ}{\mathbb{Q}}

\newcommand{\CC}{\mathbb{C}}

\newcommand{\Hom}{\mathrm{Hom}}
\newcommand{\OO}{\mathcal{O}}

\newcommand{\ZZ}{\mathbb{Z}}

\newcommand{\Gal}{\mathrm{Gal}}

\newcommand{\Spec}{\mathrm{Spec}}

\newcommand\mapsfrom{\mathrel{\reflectbox{\ensuremath{\mapsto}}}}

\newtheorem{thm}{Theorem}[section]

\newtheorem{cor}[thm]{Corollary}
\newtheorem{prop}[thm]{Proposition}
\newtheorem{lem}[thm]{Lemma}
\newtheorem{conj}[thm]{Conjecture}

\theoremstyle{definition}
\newtheorem{defn}[thm]{Definition}

\newtheorem{exmp}[thm]{Example}
\newtheorem{exmps}[thm]{Examples}

\theoremstyle{remark}
\newtheorem{rem}[thm]{Remark}

\makeatletter
\let\c@equation\c@thm
\makeatother
\numberwithin{equation}{section}

\title{A New Northcott Property for Faltings Height}

\author{Lucia Mocz}

\date{}

\begin{document}

\maketitle

\begin{abstract}
In this work we prove a new Northcott property for the Faltings height. Namely we show, assuming the Colmez Conjecture and the Artin Conjecture, that there are finitely many CM abelian varieties over $\CC$ of a fixed dimension which have bounded Faltings height. The technique developed uses new tools from integral p-adic Hodge theory to study the variation of Faltings height within an isogeny class of CM abelian varieties. In special cases, we are moreover able to use the technique to develop new Colmez-type formulas for the Faltings height.
\end{abstract}

\tableofcontents

\section{Introduction}

In this paper we develop an explicit new technique to study the change in Faltings 
height within an isogeny class of CM abelian varieties over $\CC$. Most importantly, we 
obtain a new formula for computing this change for certain 
``fundamental'' isogenies between CM abelian varieties. 
Assuming the Colmez conjecture and the Artin conjecture, we are able to deduce 
from these formulas a new Northcott property 
for the Faltings heights of CM abelian varieties, namely that there are finitely 
many isomorphism classes of 
CM abelian varieties over $\CC$ of a fixed dimension $g$ which have bounded Faltings 
height. In certain special small dimension cases, 
we use these computations to write a Colmez-type formula 
for the Faltings height.

The results in this paper are new, although the question of computing the Faltings 
height of CM elliptic curves with non-maximal CM order was considered in a paper by 
Nakkajima and Taguchi \cite{japanese}, and the 
variation of the Faltings height for 
non-CM elliptic curves was studied by 
Szpiro and Ullmo in \cite{Szpiro}. Our techniques, which are also 
new to questions 
of this type, are more powerful than those considered in \cite{japanese} 
and \cite{Szpiro}, both of which are not 
easily seen to be generalizable to abelian varieties of 
dimension greater than $1$. The 
main input for our formulas comes from integral $p$-adic Hodge theory, namely to 
use Kisin modules to analyze the ramification behavior of isogenies between CM 
abelian varieties. These new tools may be introduced to reinterpret some classical 
results of Faltings \cite{Faltings} and Raynaud \cite{Raynaud}, 
but we remark that this would 
not yield a strengthening of their results. The technique generalizes 
to study isogenies between CM motives with an appropriate formulation of 
motivic Faltings height. Work on motivic Faltings height has been done 
by Kato \cite{Kato} and Koshikawa \cite{Koshikawa}, and a recent 
preprint by Johannes Ansch\"utz \cite{Anschutz} 
shows how to define CM objects in 
any rigid Tannakian category, which gives the appropriate 
setting for posing the question on CM motives.

To define the Faltings height of an abelian variety, we let $K$ be 
a number field and $A$ to denote an abelian 
variety of dimension $g\ge 1$ defined over the field $K$ 
and we assume it to have semi-stable reduction everywhere. Let 
$\mathcal{A}/\mathcal{O}_K$ denote the corresponding N\'eron model with 
identity section $s:\Spec(\OO_K)\to \mathcal{A}$ extending 
the identity section on $A$. Denote by $\omega_{\mathcal{A}} := s^*\bigwedge^g\Omega^1_{\mathcal{A}/\OO_K}$ 
the Hodge bundle on $\mathcal{A}$, and $\overline{\omega}_{\mathcal{A}}$ 
the line bundle $\omega_{\mathcal{A}}$ endowed with Hermitian metrics at 
all infinite places 
(defined in Section \ref{height}). 
The Hermitian structure allows us to define an Arakelov-type intersection 
theory, and more precisely an Arakelov-degree of line bundles. The 
Arakelov degree of $\overline{\omega}_{\mathcal{A}}$ is a useful 
invariant for studying questions of diophantine approximations 
(e.g., the classical work of Faltings \cite{Faltings} or the more recent 
work of Tsimerman \cite{Tsimerman}), 
and goes under the name of \emph{Faltings height}. 
The Faltings height may also be written by the explicit formula
$$[K:\QQ]h_{\mathrm{Fal}}(A):=\widehat{\mathrm{deg}}(\overline{\omega}_{\mathcal{A}}) = \log\left(\#s^*(\omega_{\mathcal{A}})/\OO_K\cdot \alpha\right) - \sum_{\sigma:K\hookrightarrow\CC}\frac{1}{2}\log \int_{A^\sigma(\CC)}^{\mathrm{norm}}\alpha^\sigma\wedge \overline{\alpha^\sigma}$$
where $\alpha \in \omega_{\mathcal{A}}$ and the integral has 
an appropriate normalization constant (see \cite{Pazuki} for a discussion on 
various normalizations).  
We note that this explicit formula is base-change invariant and 
independent of the choice of $\alpha\in \omega_{\mathcal{A}}$ by the 
adelic product formula.

The original Northcott property for Faltings height shown 
by Faltings in \cite{Faltings} states that 
for a fixed positive integer $d$ and fixed positive real number $C$, there 
exist finitely many 
abelian varieties $A/K$ where $[K:\QQ]\le d$ 
and $h_{\mathrm{Fal}}(A)<C$. This property is one of the key propositions in 
Faltings' proof of the Shafarevich and Mordell conjectures \cite{Faltings}.

We assume now that all our abelian varieties have complex multiplication 
(see Section \ref{CM} for a definition). \emph{A priori}, 
a CM abelian variety is a complex point on a Siegel moduli space, but the 
theorem of Shimura--Taniyama asserts that every CM abelian variety descends to a 
number field and moreover has a model over a number field over which 
it has good reduction 
everywhere. Both of these assertions are key to transforming the 
an abstract geometric question of counting certain complex points 
on the Siegel moduli space exhibiting extra symmetric behavior 
to a question in arithmetic.

The first main theorem we prove is the following.

\begin{thm}[CM Northcott Property for Isogeny Classes]
\label{CMIsog}
Let $C>0$ be a constant and $g>0$ be a fixed 
integer. Then the number of isomorphism classes of 
CM abelian varieties $A$ of dimension $g$ 
with $h_{\mathrm{Fal}}(A)<C$ is finite within isogeny classes.
\end{thm}

This theorem is proved quantitatively: the change in Faltings height 
is numerically computed based on the combinatorics of prime 
splitting in the CM field, or more generally a CM algebra, 
which geometrically corresponds to 
a canonical decomposition of the isogeny. At each prime-power-degree 
isogeny, the key computation is that of the relative Hodge bundle, 
which is done locally using Kisin modules. 

Since our computations are precise, a 
closed formula for the change in Faltings height 
between any two isogenous CM abelian varieties may 
be obtained in any case where the splitting behavior of all primes 
dividing the degree of the isogeny is entirely 
understood. We compute one example of this formula, which 
is new to the literature, namely the Faltings height 
variation for simple CM abelian surfaces 
whose CM endomorphism ring is Galois. Namely, 
we let $A$ be a simple CM abelian surface over a number 
field $K$ with CM by $(\OO_E,\Phi)$ 
where $E/\QQ$ is a cyclic, Galois CM field of degree $4$ 
and $F\subseteq E$ its totally real subfield. Let 
$A_{p^n}$ be a simple CM abelian surface defined 
over a number field $K$ with CM by $(\OO,\Phi)$ 
where $\OO\subseteq \OO_E$ is a non-maximal order with conductor 
divisible by $p$, and an isogeny 
$\phi:A\to A_{p^n}$ has minimal degree $p^n$. 

\begin{thm}
\label{MaxForm}
\begin{enumerate}
\item If $p = \mathfrak{p}^2$ such that $p$ ramifies in $F/\QQ$ 
and $\mathfrak{p}$ remains inert in $E/F$,
$$h_{\mathrm{Fal}}(A_{p^n}) = h_{\mathrm{Fal}}(A) + \left[\frac{n}{2} - \frac{2}{p+1}\left(\frac{1-p^{-n}}{1-p^{-1}}\right)\right]\log p.$$
\item If $p = \mathfrak{p}$ is inert, let 
$(\lambda_1,\lambda_2,\lambda_3)$ be a non-increasing tuple 
characterizing $\mathcal{G}$ satisfying
$\lambda_1 + \lambda_2 + \lambda_3 = n$, 
$\mathcal{G}$ is $p^{\lambda_1}$-torsion, and the $p$-height 
$\iota\in \{1,2,3\}$ of $\mathcal{G}$ is the largest integer 
such that $\lambda_{\iota}\neq 0$. 
Then
$$h_{\mathrm{Fal}}(A_{p^n}) = h_{\mathrm{Fal}}(A) + \left[\frac{n}{2} - \frac{p-1}{p^4-1}\left(\sum_{i=\lambda_2+1}^{\lambda_1}\frac{p + 1}{p^i}\right)\left(\sum_{j=\lambda_3 +1}^{\lambda_2}\frac{(p+1)^2}{p^j}\right)\left(\sum_{k = 1}^{\lambda_3}\frac{(p^2 + 2)(p+1)}{p^k}\right)\right]\log p$$
where the sums are evaluated if and only if the difference in bounds is at 
least zero.
\end{enumerate}
\end{thm}

Such formulas for $h_{\mathrm{Fal}}(A_{p^n})$ for all 9 possibilities of the splitting behavior for p in E are given in Theorem \ref{LASTTHING}. 
This theorem answers a question posed by Habegger and Pazuki in \cite{HabPaz}. 
In the paper, they ask whether there are finitely many curves 
$C$ of genus $2$ defined 
over $\overline{\QQ}$ with good reduction at all but a given 
finite set of places and for which 
$\mathrm{Jac}(C)$ has complex multiplication. They show this to be true when the 
curve has good reduction everywhere and a simple CM Jacobian with CM by a 
maximal order. To extend their result to a wider class of CM Jacobians, the 
problem is reduced to the stated Northcott property in the theorem above. 
Note that while the Northcott property is general, the result by 
Habegger and Pazuki and 
stated application is specific to genus $2$ curves, and highlights a 
unique property that may be shared only by curves of low genus.

We establish conditionally the CM Northcott property 
for isomorphism classes of CM abelian 
varieties by invoking a famous conjecture 
due to Colmez \cite{Colmez}.

\begin{conj}[Colmez Conjecture, \cite{Colmez}]
\label{colm}
Let $A$ be an abelian variety over $K$ of dimension $g$, and suppose 
that $A$ has CM by the ring of integers $\mathcal{O}_E\subseteq E$, 
where $E$ is a CM field with CM type $\Phi$ and $[E:\mathbb{Q}] = 2g$. 
Then:
$$h_{\mathrm{Fal}}(A) = -\sum_i c_i\left(\frac{L'(\chi_i,0)}{L(\chi_i,0)}+ \frac{1}{2}\log f_{\chi_i}\right) + \frac{g}{2}\log 2\pi$$
where the $\chi_i$ are a finite collection of 
irreducible Artin characters of $\mathrm{Gal}(K/\mathbb{Q})$ with 
Artin conductor $f_{\chi_i}$, and the $c_i$ are positive constants 
(see Section \ref{ColmezConjSect}).
\end{conj}

In \cite{Colmez2}, Colmez develops a lower bound on the Faltings 
height by invoking the Artin conjecture to bound the (logarithmic 
derivatives of) Artin $L$-functions in the formula. This lower bound is 
expressed in terms of Artin conductors of Galois representations, 
and together with Theorem \ref{CMIsog} is 
sufficient to establish the following theorem.

\begin{thm}[CM Northcott Property]
\label{CMNorth}
Let $C>0$ be a fixed constant and $g>0$ be a fixed integer. Then assuming 
Conjecture \ref{colm} and the Artin conjecture, the number of 
$\overline{\QQ}$-isomorphism classes of CM abelian varieties $A$ 
of dimension $g$ and with $h_{\mathrm{Fal}}(A)<C$ is finite.
\end{thm}

Cases where the Colmez conjecture are known to be true include abelian 
CM fields, which was shown up to a factor of $\log 2$ by Colmez in \cite{Colmez} 
and later cleared up by Obus in \cite{Obus}, and many classes of 
abelian surfaces shown by Yang \cite{TYang}. An average version has also been 
shown independently by Yuan--Zhang \cite{YuanZhang} and 
Andreatta--Goren--Howard--Madhapusi-Pera \cite{Howard}. We note that 
the average formula is not strong enough to establish our results unless 
all abelian varieties in computing the average have the same Faltings height.

The quantitative versions of Theorem \ref{CMIsog} and 
Theorem \ref{CMNorth} have applications to future questions 
requiring an understanding of the distribution of 
(CM) points on various moduli spaces of 
abelian varieties. For instance, it recovers the average 
growth of the Faltings height of the Hecke orbit of a CM abelian 
variety, a formula computed for some low-dimensional 
abelian varieties by Autissier in \cite{Autissier}. 
This Autissier growth formula was the key input to 
\cite{Charles} and \cite{ShankarTang} to construct thin sets of 
primes of certain reduction types on elliptic curves and abelian 
surfaces. In \cite{ShankarTang}, the authors show this 
growth formula computed on a CM abelian 
surface provides the growth formula for all abelian surfaces. 
This lemma is likely true in general, and these formulas 
then extend the Autissier growth formula to all dimensions, 
providing means to extend \cite{Charles} and \cite{ShankarTang} 
even further.

The organization of this paper is as follows. In Section 2, we introduce 
the Faltings height 
and cite the key theorems and conjectures: the Faltings Isogeny Lemma, 
the Colmez Conjecture, and the lower bound on Faltings height 
from the Colmez Conjecture and the Artin Conjecture. In Section 3, 
we review the theory of complex multiplication. The theory 
we require goes well 
beyond Shimura and Taniyama's introduction of the subject, since 
we require the CM theory of N\'eron 
models of abelian varieties and $p$-divisible groups, and is 
treated in vast amount of detail by 
Chai--Conrad--Oort \cite{Conrad}. Section 4 is 
devoted to the relevant recent developments 
in (integral) $p$-adic Hodge theory we need to 
compute the (relative) Hodge 
bundle in the change in Faltings height, most namely 
the development revolving around the Kisin modules originally introduced 
by Kisin in \cite{Kisin}. 
The brief Section 5 reviews the arithmetic and group theoretic 
properties encoded by Kisin modules, such as a characterization 
of finite flat subgroup schemes and the theory of their 
Harder-Narasimhan (HN) filtration.

Sections 6-8 are the heart of the paper and 
devoted to the proof of the CM Northcott Property and 
the application to Colmez-type formulae. 
In Section 6, 
we introduce the notion of an $\OO_E$-linear CM Kisin module 
and deduce decomposition theorems of their structure, 
including a general theorem on a weak characterization of 
unstable submodules of $\OO_E$-linear CM Kisin torsion modules. 
The work in this section inspired the recent development of 
CM Breuil-Kisin-Fargues (BKF) 
modules introduced by Johannes Ansch\"utz in \cite{Anschutz}, although 
he works with a certain ``isogeny'' category of 
BKF modules and does not recover (as we do) information on individual 
elements in the ``isogeny classes''. There is current work in progress 
by the present author to extend the present computations to his setting. 
In Section 7, we 
compute precisely the HN slopes of unstable submodules of 
$\OO_E$-linear torsion CM Kisin modules using 
explicit CM and Lubin-Tate theory. We emphasize that one 
may deduce from these 
computations the 
\emph{exact} change in Faltings height between any given two isogeneous 
CM abelian varieties. This is not yet possible 
to do without the CM hypothesis with current technology, 
but we show how to use some elementary bounds from 
transcendence theory to obtain bounds on this variation 
in a forthcoming paper. Finally, Section 8 is devoted to 
proving Theorems \ref{CMIsog}, 
\ref{MaxForm}, and \ref{CMNorth} from the developments in 
the rest of the paper.

\subsection*{Acknowledgements}
It is a pleasure for the author to thank her PhD advisor, Shouwu Zhang, 
for introducing her to this problem. She 
also thanks Johannes Ansch\"utz, Ana Caraiani, Brian Conrad, 
Ziyang Gao, Quentin Guignard, Peter Sarnak, Yunqing Tang, Salim Tayou, 
and Rafael Van K\"anel for discussions and/or comments. 
Especial thanks goes further to Johannes Ansch\"utz, 
Quentin Guignard, and Yunqing Tang for their thorough review of 
the mathematical content and helpful corrections to earlier versions of this paper. 
This work was prepared 
at Princeton University and IHES as part of the author's PhD dissertation, 
and the author is indebted to the hospitality offered at each location.

\section{Faltings height}

\subsection{Semi-stable Faltings height}
\label{height}

Let $K$ be a number field and $\OO_K\subseteq K$ be its ring of integers. We 
consider abelian varieties $A$ of dimension $g\ge 1$ defined over 
the field $K$ with semi-stable 
reduction everywhere and write $\mathcal{A}/\OO_K$ for their N\'eron model.  
Denote by $s:\Spec(\mathcal{O}_K)\to \mathcal{A}$ the extension 
of the zero 
section, and by $\omega_\mathcal{A}:=s^*\Omega^g_{\mathcal{A}/\OO_K}$ 
the canonical line bundle pulled back along the zero section to a line bundle 
on $\Spec(\OO_K)$.

Recall the following definition from Arakelov theory.

\begin{defn}
A \textbf{Hermitian line bundle} on $\mathrm{Spec}(\OO_K)$ is a pair 
$(\mathcal{L},\{||\cdot||_\sigma\}_{\sigma:K\hookrightarrow \CC})$ consisting of 
an invertible sheaf $\mathcal{L}$ on $\Spec(\OO_K)$ 
and a collection of continuous Hermitian metrics 
$||\cdot||_\sigma$ on $\mathcal{L}^{\text{an}}\otimes_\sigma \CC$ 
ranging over the complex embeddings $\sigma:K\hookrightarrow \CC$ 
such that  $|| s^c ||_{\sigma^c} = || s ||_{\sigma}$, where $c$ denotes 
complex conjugation.
\end{defn}

The line bundle $\omega_{\mathcal{A}}$ can be 
equipped with a collection of hermitian metrics to 
give it the structure of a hermitian line bundle. Define the 
hermitian norm $||\cdot ||_\sigma$ for each 
embedding $\sigma:K\hookrightarrow \mathbb{C}$ 
by the formula
$$||\alpha||_\sigma^2 = \frac{1}{(2\pi)^g}\left|\int_{\mathcal{A}_\sigma(\mathbf{C})}\alpha\wedge \overline{\alpha}\right|$$
where $\mathcal{A}_\sigma$ denotes the ``completion'' of the 
N\'eron model at the archimedean fiber, i.e., a complex manifold 
which descends to an algebraic model $A/K$. 
Here we consider $\alpha \in H^0(\mathcal{A}_\sigma,\Omega^g_{\mathcal{A}_\sigma})$ 
via the isomorphism $\omega_{\mathcal{A}}\otimes_\sigma \mathbf{C}\simeq H^0(\mathcal{A}_\sigma,\Omega^g_{\mathcal{A}_\sigma})$. 
Then the hermitian line bundle 
$\overline{\omega}_\mathcal{A}$ will denote the pair 
$(\omega_\mathcal{A}, \left\{||\cdot||_\sigma\right\}_{\sigma:K\hookrightarrow\mathbb{C}})$.

Recall as well the following definition from Arakelov theory.

\begin{defn}
The \textbf{Arakelov degree} of a Hermitian line bundle $(\mathcal{L}, \left\{||\cdot||_\sigma\right\}_{\sigma:K\hookrightarrow\mathbb{C}})$ on 
$\mathrm{Spec}(\OO_K)$ is the value
$$\widehat{\mathrm{deg}}(\overline{\mathcal{L}}) = \log\#(\mathcal{L}/\alpha) - \sum_{\sigma:K\hookrightarrow \mathbb{C}} \log||\alpha||_\sigma$$
where $\alpha$ is a non-zero global section of $\mathcal{L}$. 
\end{defn}

\begin{rem}
By the product formula, this is independent of the choice of $\alpha$, 
so the Arakelov degree is a well-defined invariant.
\end{rem}

Using the Arakelov degree, we define the Faltings height 
as follows. Note that the Faltings height is an arithmetic 
invariant on the isomorphism classes of abelian varieties.

\begin{defn}
The \textbf{(stable) Faltings height} of $A$ is
$$h_{\mathrm{Fal}}(A):=\frac{1}{[K:\mathbf{Q}]}\widehat{\mathrm{deg}}(\overline{\omega}_\mathcal{A}).$$
This is computed explicitly by 
the formula
$$\widehat{\mathrm{deg}}(\overline{\omega}_\mathcal{A}) = \log\#(\omega_{\mathcal{A}}/\mathcal{O}_K\alpha) - \sum_{\sigma:K\hookrightarrow \mathbf{C}}\log ||\alpha||_\sigma.$$
\end{defn}

The invariant receives its name as a height due to the fact it 
satisfies the following celebrated Northcott property demonstrated by 
Faltings in \cite{Faltings}. This Northcott property was a key insight in 
Faltings' proof of the Mordell conjecture.

\begin{thm}[Faltings' Northcott Property, \cite{Faltings}]
Let $C$ and $d$ be fixed positive constants and $g\ge 1$ be an integer. 
Then the set of isomorphism classes of abelian varieties 
$A/K$ of dimension $g$ with $[K:\QQ]< d$ and $h_{\mathrm{Fal}}(A)<C$ is finite.
\end{thm}

The Faltings height is not 
an isogeny-invariant. This very property of variation of the 
Faltings height in isogeny classes 
was exploited heavily by Faltings in \cite{Faltings}, 
and is measured by the following formula. 

\begin{lem}[Faltings Isogeny Lemma, \cite{Faltings}]
\label{FaltingsIsog}
Let $\phi: A_1\to A_2$ be an isogeny over $K$, and let $\mathcal{G}$ be the finite flat group 
scheme kernel of the extension $\phi:\mathcal{A}_1\to \mathcal{A}_2$ 
over $\OO_K$, 
where $\mathcal{A}_i$ are the N\'eron models corresponding to $A_i$ for 
$i\in \{1,2\}$. Then
$$h_{\mathrm{Fal}}(A_2) = h_{\mathrm{Fal}}(A_1) + \frac{1}{2}\log(\deg(\phi)) - \frac{1}{[K:\mathbb{Q}]}\log(\#s^*\Omega_{\mathcal{G}/\OO_K}^1).$$
\end{lem}

The techniques to be developed in this paper aim 
to compute the ramification term on the right side 
of this formula and show its contribution is, in all 
relevant cases, 
negligible compared to the term provided by the degree 
of the isogeny.

\subsection{Colmez Conjecture}
\label{ColmezConjSect}
We state here the Colmez Conjecture for the Faltings height of CM abelian varieties. 
The theory of complex multiplication for abelian varieties and $p$-divisible 
groups will be reviewed in Section \ref{CM}, so we invite the reader to 
briefly review the definitions if she wants first to 
familiarize herself with the subject. 
Apart from the basic definitions, however, 
familiarity with the theory discussed in 
the section will not be relevant to state the 
conjecture.

Let $(E,\Phi)$ be a CM pair where $E$ is a CM field 
such that $[E:\QQ]=2g$ 
and $\Phi = \{\phi_1,\ldots,\phi_g\}$ is a CM type of $E$. For 
$\phi:E\to \overline{\QQ}$, define the function $a_{E,\phi,\Phi}$ on 
$\mathrm{Gal}(\overline{\QQ}/\QQ)$ by
$$a_{E,\phi,\Phi}(g) = \left\{\begin{array}{ll}
1&\text{if }g\phi\in \Phi,\\
0&\text{otherwise.}\end{array}\right.$$
We construct the 
class function $A_{E,\Phi} := \sum_{\phi\in \Phi}a_{E,\phi,\Phi}$. 
Then, letting $E^*$ denote the reflex field of 
$(E,\Phi)$, this class function satisfies
$A_{E,\Phi}(g) = A_{E,\Phi}(1)$ if and 
only if $g \in \mathrm{Gal}(\overline{\QQ} / E^*).$
Since
irreducible Artin characters form a basis for the space of locally 
constant class functions on $\Gal(\overline{\QQ}/\QQ)$, 
$A_{E,\Phi}$ has a representation
$$A_{E,\Phi} = \sum_i c_i \chi_i$$ 
where the sum runs through all irreducible 
characters on $\Gal(\overline{\QQ}/\QQ)$ and $c_i$ denotes the 
multiplicity by which they occur. As demonstrated in 
\cite{Colmez2}, $c_i = 0$ for all but finitely 
many characters, and otherwise $c_i>0$. Moreover, $c_i>0$ only 
when $\chi_i$ is odd, and therefore  
$L(\chi_i,0)\neq 0$ for the characters that appear with 
a positive constant.

\begin{conj}[Colmez Conjecture, \cite{Colmez}]
\label{colmez}
Let $A$ be an abelian variety of dimension $g$ 
over a number field $K$, and suppose 
that $A$ has CM by the ring of integers $\mathcal{O}_E\subseteq E$, 
where $E$ is a CM field such that 
$[E:\QQ] = 2g$ and $\Phi$ is a CM type of $E$. 
Let $A_{E,\Phi} = \sum_i c_i \chi_i$ be as defined above, 
where the sum is over all irreducible 
characters $\chi_i\in \mathrm{Hom}(\Gal(\overline{\QQ}/\mathbb{Q}),\CC)$, 
each with corresponding (unique) Artin conductor $f_{\chi_i}$. Then:
$$h_{\mathrm{Fal}}(A) = -\sum_i c_i\left(\frac{L'(\chi_i,0)}{L(\chi_i,0)} + \frac{1}{2}\log f_{\chi_i}\right) + \frac{g}{2}\log 2\pi.$$
\end{conj}

In the original paper \cite{Colmez}, 
Colmez proves this conjecture for abelian CM fields satisfying a 
certain ramification 
condition above the prime $2$. Obus \cite{Obus} was later able to drop the 
ramification condition and generalize the result to all abelian CM fields. 
Yang \cite{TYang} through independent methods verified the conjecture for 
abelian surfaces whose CM field is not Galois over $\QQ$. In two different 
approaches towards this conjecture, Yuan--Zhang \cite{YuanZhang} and 
Andreatta--Goren--Howard--Madapusi-Pera\cite{Howard} proved an 
average version (where the average is over the CM types) of this is true.

Our main interest in the Colmez conjecture is to bound the Faltings height. 
In a separate work by Colmez \cite{Colmez2} the following bound was obtained.

\begin{thm}[Lower Bound on Faltings height, \cite{Colmez2}]
\label{colmezbound}
Assuming Conjecture \ref{colmez} and the Artin Conjecture, there 
exists an effective constant $C>0$ such that
$$h_{\mathrm{Fal}}(A) \ge \frac{1}{2}\left(-\log 2\pi - \log \pi + \frac{\Gamma'\left(\frac{1}{2}\right)}{\Gamma\left(\frac{1}{2}\right)}\right) + C\mu_{\mathrm{Art}}(A_{E,\Phi})$$
where $\mu_{\mathrm{Art}}(A_{E,\Phi}) = \sum_{\chi}\langle A_{E,\Phi},\chi\rangle \log f_{\chi}$ 
and $f_{\chi}$ is the Artin conductor corresponding to the character $\chi$.
\end{thm}

The lower bound immediately gives the following (primitive) CM Northcott property.

\begin{cor}
\label{primitive}
Let $C$ be a fixed positive constant and $g\ge 1$ an integer. Then assuming 
Conjecture \ref{colmez} and the Artin Conjecture, the set of isomorphism classes of 
simple CM abelian varieties of dimension $g$ with CM by a maximal order $\OO_E\subseteq E$ 
with $[E:\QQ] = g$ and $h_{\mathrm{Fal}}(A)<C$ is finite.
\end{cor}

\section{Complex Multiplication}
\label{CM}
\subsection{CM Abelian Varieties}
\subsubsection{Definitions}
We give here a brief survey of the relevant theory of complex multiplication 
for abelian varieties. 
This theory is well-known, so the familiar reader can skip this section as it 
is entirely self-contained. For a comprehensive treatment of the subject, 
see Chapter 1 in \cite{Conrad}. We note that we need more than just the 
theory introduced in \cite{Shimura}, which does not develop the theory 
on integral models of CM abelian varieties and CM types valued 
in arbitrary algebraically closed fields.

\begin{defn}
A \textbf{CM pair} $(E,\Phi_L)$ consists of:
\begin{itemize}
\item A number field $E$, called the \textbf{CM field}, having a non-trivial involution 
$\tau\in \mathrm{Aut}(E)$ such that every embedding $i:E\hookrightarrow\CC$ 
satisfies $i(\tau(x)) = i(x)^c$ for all $x\in E$. Here $c$ denotes complex conjugation 
in $\CC$.
\item A subset $\Phi_L\subseteq \mathrm{Hom}_{\QQ\text{-alg}}(E,L)$, called a \textbf{CM type} valued in an algebraically-closed field $L$, such that 
$\mathrm{Hom}_{\QQ\text{-alg}}(E,L) = \Phi_L\sqcup\Phi_L\circ \tau$.
\end{itemize}

A \textbf{CM algebra} $P$ is a finite product of CM fields 
$\prod_i E_i$. We define on it a 
CM type by $\Phi_L = \coprod_i \Phi_{L,i}$ where the 
$\Phi_{L,i}$ is a CM type on $E_i$.
\end{defn}

\begin{rem}
When $L = \CC$ in the definition above, or the algebraically-closed 
field $L$ is otherwise evident, we denote the CM type 
by $\Phi$.
\end{rem}

The element $\tau$ in the definition is necessarily unique 
and its fixed field is a totally real subfield $F\subseteq E$ 
which satisfies $[E:F] = 2$. We call $F$ the \emph{maximal totally real subfield} 
of $E$.

\begin{defn}
An abelian variety $A$ of dimension $g$ defined over $\CC$ 
\textbf{has CM by $(P,\Phi)$} if:
\begin{itemize}
\item There is an injective $\QQ$-algebra homomorphism 
$$i:P\hookrightarrow\mathrm{End}^0(A):=\QQ\otimes_\ZZ \mathrm{End}(A)$$
of a CM algebra $P$ such that $[P:\QQ] = 2g$.
\item The $P\otimes_{\QQ} \mathbb{C}$-module structure 
of $\mathrm{Lie}(A)$ is given by the CM type $\Phi$ valued in $\CC$.
\end{itemize}
\end{defn}

By the Poincar\'e reducibility theorem (see, e.g., Theorem 1.2.1.3 in 
\cite{Conrad}), every abelian variety is isogeneous to a product of 
its distinct isotypic pieces. Hence, if an abelian variety $A$ has CM by 
$(P,\Phi)$, then each of the isotypic factors $A_i$ 
have CM by a sub-algebra $P_i\subseteq P$ 
such that $P = \prod_i P_i$. Moreover, each $P_i$ is a simple algebra 
of finite dimension over $\QQ$, and hence a CM field. We will not 
lose any content for establishing the remainder of our theory if we 
assume that $P = E$ is a CM field and even if we assume that 
$A$ is simple, so we will use the term \emph{CM abelian variety} in 
the remainder of this section to refer to this case.

\begin{exmp}
When $A$ has complex multiplication by $(E,\Phi)$, where $E$ is a CM field, 
then necessarily $\mathrm{End}(A) = \OO$ for some (not necessarily 
maximal) order $\OO\subseteq E$. Its dual abelian variety has complex 
multiplication by $(E,\Phi^c\circ \tau)$, where the $E$-action 
on $A$ is induced by the composition of the dual action with 
complex conjugation.
\end{exmp}

\begin{rem}
When we wish to specify the endomorphism order of a CM abelian variety, 
we say it has CM by $(\OO,\Phi)$. Note that when $\OO = \OO_E$, 
$\Phi$ gives a type on the integral structure $\OO_E$.
\end{rem}

The starting point for questions of arithmetic interest involving 
CM theory is the following theorem.

\begin{thm}[Shimura--Taniyama, \cite{Conrad}, \cite{Shimura}]
\label{shimtan}
Let $A/\CC$ be a CM abelian variety. Then there exists a number field $K$ 
such that $A$ descends to $K$. Moreover, there exists a finite extension 
$K'/K$ such that $A_{K'} := A\otimes_K K'$ has good reduction at all 
the finite places of $K'$.
\end{thm}

By this theorem, one may construct appropriate integral models of CM 
abelian varieties. Moreover, since CM abelian varieties have 
good reduction everywhere over some 
finite extension $K'/K$, their abelian group structure extends to the 
entire integral model (see \cite{Neron}). 
The CM structure is also preserved on the integral model, i.e., complex 
multiplication is a theory on the generic fiber. This last fact is shown in 
the following theorem.

\begin{prop}
Let $A/K$ be an abelian variety with CM by $(\OO,\Phi)$ and let 
$\mathcal{A}/\OO_K$ denote its N\'eron model. Then 
$\mathrm{End}(\mathcal{A}) = \mathrm{End}(A)$ and $\mathrm{Lie}(\mathcal{A})\otimes_{\OO_K}\CC$ 
has the structure of an $\OO\otimes_{\OO_K}\CC$-module under $\Phi$.
\end{prop}
\begin{proof}
The map $\mathrm{End}(\mathcal{A})\to \mathrm{End}(A)$ 
always exists by 
restricting endomorphisms on the integral model to the generic fiber, 
and is therefore also always injective. 
Surjectivity follows from the universal mapping 
property of N\'eron models (see \cite{Neron}). By the isomorphisms
$$\mathrm{Lie}(\mathcal{A})\otimes_{\OO_K}\CC\simeq 
(\mathrm{Lie}(\mathcal{A})\otimes_{\OO_K}K)\otimes_K \CC \simeq \mathrm{Lie}(A)\otimes_K\CC,$$
$\mathrm{Lie}(\mathcal{A})\otimes_{\OO_K}\CC$ has an 
$\OO\otimes_{\OO_K}\CC$-module structure under $\Phi$ inherited from 
the $\OO\otimes_{\OO_K}\CC$-module structure of $\mathrm{Lie}(A)\otimes_K\CC$.
\end{proof}

\subsubsection{Theory of the Reflex Norm}
Theorem \ref{shimtan} has a stronger form which provides more 
precise information on the field 
of definition $K$ to which a CM abelian variety descends. This is 
contingent on the \emph{theory of the reflex norm}, the 
celebrated player of CM theory. We introduce this theory starting 
from the following definition.

\begin{defn}
Let 
$$h_{\Phi, E} = \prod_{i\in \mathrm{Hom}(E,\CC)}h_{i, E}:\mathbb{G}_m\to \mathrm{Res}_{E/\QQ}(\mathbb{G}_m)_\CC = \prod_{i\in \mathrm{Hom}(E,\CC)} \mathbb{G}_{m,i}$$
be a cocharacter where each $h_{i,E}$ is defined such that
$$h_{i,E}:t\mapsto \left\{\begin{array}{ll}
t&\text{ if }i\in \Phi,\\
1&\text{ if }i\notin\Phi.\end{array}\right.$$
The minimal field of definition for $h_{\Phi,E}$ is the \textbf{reflex field} $E^*$ of 
$E$.
\end{defn}

The reader can check that the minimal field of definition of $h_{\Phi,E}$ over $\QQ$ 
is stabilized by those elements of $\Gal(\overline{\QQ}/\QQ)$ which 
fix the CM type $\Phi$ (this, in fact, was the classical definition of 
the reflex field in \cite{Shimura}). Note that the reflex field is 
always itself a CM field\footnote{Assume 
to the contrary that $E^*$ is not a CM field. As it is 
a subfield of a CM field (say, the algebraic closure of $E$) it 
must therefore be totally real. But if $E^*$ were totally real, complex conjugation 
would lie in $\Gal(\overline{\QQ}/E^*)$, which 
contradicts the characterization of the Galois group by the CM 
type of $E$. Hence $E^*$ must be a CM field.}.

\begin{defn}
\label{reflexcochar}
Let $(E,\Phi)$ be a CM field together with a $\overline{\QQ}$-valued CM type 
and let $E^*$ be its reflex field. For any finite extension $K/E^*$, 
the \textbf{reflex norm} $N_{\Phi,K}$ is the composite
$$\mathrm{Res}_{K/\QQ}(\mathbb{G}_m)\xrightarrow{\mathrm{Res}_{K/\QQ}(h_{\Phi,E})}\mathrm{Res}_{K/\QQ}(\mathrm{Res}_{E/\QQ}(\mathbb{G}_m)_K)\xrightarrow{\mathrm{Nm}_{K/\QQ}}\mathrm{Res}_{E/\QQ}(\mathbb{G}_m).$$
\end{defn}

\begin{rem}
For any $\QQ$-algebra $R$, the reflex norm induces a map
\begin{align*}
N_{\Phi,K\otimes R}:(K\otimes R)^\times&\to (E\otimes R)^\times\\
x&\mapsto \mathrm{det}_E(x|_{V_{\Phi,K}\otimes_\QQ R})\end{align*}
where $V_{\Phi,K}$ is any $E\otimes_\QQ K$-module satisfying 
$V_{\Phi,K}\otimes_K\overline{\QQ}\simeq \prod_{i\in \Phi}\overline{\QQ}$. 
\end{rem}

The stronger form of Theorem \ref{shimtan} can now be stated as follows. 
We will let below $r_K^{\mathrm{Art}}:\mathbb{A}_K^\times/K^\times\to\mathrm{Gal}(K^{\mathrm{ab}}/K)$ be the global Artin reciprocity map 
where $K^\mathrm{ab}$ is the standard notation for the maximal 
abelian extension of $K$. We let 
$K_v$ denote the localization of $K$ at a place $v$ with 
ring of integral elements 
$\OO_{K_v}$ and a specified uniformizer $\pi_v$, and 
$k_v = \OO_{K_v}/\pi_v\OO_{K_v}$ be the residue field at $v$ 
of size $q_v$.

\begin{thm}[Main Theorem of Complex Multiplication, \cite{Conrad}]
Let $A$ be an abelian variety with CM by $(E,\Phi)$ 
defined over a number field $K$. Then 
the following hold true:
\begin{enumerate}
\item The reflex field $E^*$ is contained in $K$.
\item There exists a unique algebraic Hecke character 
$$\epsilon:\mathbb{A}_K^\times\to E^\times$$
such that at each prime $\ell$, the continuous homomorphism
\begin{align*}
\phi_\ell:\Gal(K^\mathrm{ab}/K)&\to E_\ell^\times\\
r_K^{\mathrm{Art}}(a)&\mapsto \epsilon(a)\mathrm{N}_{\Phi,K}(a_\ell^{-1})
\end{align*}
is equal to the $\ell$-adic representation of $\Gal(K^{\mathrm{ab}}/K)$ 
on the $\ell$-adic Tate module $V_\ell(A)$.
\item $A$ has good reduction at a finite place $v$ of $K$ if and 
only if $\epsilon|_{\OO_{K_v}^\times} = 1$. In this case, for 
any choice of uniformizer $\pi_v\in \OO_{K_v}$, 
$\epsilon(\pi_v) = \mathrm{Fr}_{v}^{q_v}$, where $\mathrm{Fr}_{v}\in E$ 
denotes the endomorphism lifting the relative Frobenius on the 
reduction $A_v$ of $A$ defined over the finite field $k_v$.
\end{enumerate}
\end{thm}

\subsection{CM $p$-divisible Groups}
\subsubsection{Definitions}
The theory of complex multiplication for $p$-divisible groups arises naturally 
from the integral theory of complex multiplication for abelian varieties. 
Throughout this section, $K$ will be a local field of characteristic $0$ 
with finite residue field $k$ of characteristic $p>0$ and $\OO_K\subseteq K$ 
will be its ring of integers. 
Then let $A/K$ be an abelian variety with CM 
and $\mathcal{A}/\OO_K$ be a corresponding 
local N\'eron model and $\mathcal{A}[p^\infty]/\OO_{K_p}$ 
a $p$-divisible group. Since 
$\mathrm{End}(A)\otimes_\ZZ\ZZ_p\hookrightarrow\mathrm{End}(\mathcal{A}[p^\infty])$, 
one expects a theory of complex multiplication for $p$-divisible 
groups parallel to the theory for abelian varieties. 
We refer the reader to Chapter 3 of \cite{Conrad} for a 
comprehensive treatment of the theory; below we review 
relevant statements.

Recall first that for any $p$-divisible group $\mathcal{G}$ defined over 
$\OO_K$, the endomorphism 
ring $\mathrm{End}(\mathcal{G})$ is a finite 
free $\mathbb{Z}_p$-module which is 
necessarily a $\mathbb{Z}_p$-subalgebra of 
$\mathrm{End}(T_p\mathcal{G})$. Each $p$-divisible group 
also carries with it an invariant 
$h$ called the \emph{height} which is an integer always 
bounded below (and sometimes equal to) the 
dimension $d$ of the $p$-divisible group.

\begin{defn}
A $p$-divisible group $\mathcal{G}/\OO_K$ of dimension $d$ and height $h$ 
\textbf{has ($p$-adic) CM by $(E,\Phi)$} if:
\begin{itemize}
\item There is an injective $\QQ_p$-algebra homomorphism 
$$i:E\hookrightarrow \mathrm{End}^0(\mathcal{G}):=\QQ_p\otimes_{\ZZ_p}\mathrm{End}(\mathcal{G})$$
where $E/\QQ_p$, called the \textbf{($p$-adic) CM algebra}, is a commutative, semisimple $\QQ_p$-algebra such that $[E:\QQ_p] = h$.
\item The $E\otimes_{\QQ_p}K$-module structure of 
$\mathrm{Lie}(\mathcal{G})\otimes_{\ZZ_p}\overline{\QQ}_p$ is given by 
a subset $\Phi\subseteq \mathrm{Hom}_{\QQ_p\text{-}\mathrm{alg}}(E, \overline{\QQ}_p)$, 
called the \textbf{($p$-adic) CM type}, such that $\#\Phi = d$.
\end{itemize}
\end{defn}

\begin{exmp}
\label{abpdiv}
Let $\mathcal{A}/\OO_K$ be a local integral model of an abelian variety 
$A/K$ and suppose $A$ has CM by $(E,\Phi)$. 
Then $\mathcal{A}[p^\infty]$ has CM by $(E_p,\Phi_p)$ where 
$E_p:=\QQ_p\otimes_\QQ E$ and 
$\Phi_p := \Phi_{\overline{\QQ}}\cap \mathrm{Hom}_{\QQ_p\text{-}\mathrm{alg}}(E_p,\overline{\QQ_p})$. 
In fact, if $\mathrm{End}(A) = \OO$, then 
$\mathrm{End}\left(\mathcal{A}[p^\infty]\right)=\OO\otimes_{\ZZ}\ZZ_p$.
\end{exmp}

A good class of examples of CM $p$-divisible groups are those that arise 
from CM abelian varieties, as demonstrated above. We note that although 
a CM abelian variety might be simple, its $p$-divisible group 
$\mathcal{A}[p^\infty]$ may not be so. Analogous to 
CM abelian varieties, each CM $p$-divisible 
group is isogenous to one that isotypically 
decomposes along the simple subalgebras of 
its CM algebra. We will assume for the remainder of 
this section that all CM $p$-divisible groups that we 
consider are simple.

\begin{exmp}
Let $\mathcal{G}/\OO_K$ have ($p$-adic) CM by $(E,\Phi)$. Then its Cartier 
dual $\mathcal{G}^*$ has a complementary structure $\Phi^c$ such 
that $\mathrm{Hom}_{\QQ_p\text{-}\mathrm{alg}}(E,\overline{\QQ}_p) = \Phi\coprod \Phi^c$. 
Moreover, when $\mathcal{G} = \mathcal{A}[p^\infty]$ for a CM 
abelian variety $\mathcal{A}/\OO_K$, then 
$\mathcal{A}[p^\infty]^* = \mathcal{A}^\vee[p^\infty]$ and 
$\Phi^c$ agrees with the induced type from $\mathcal{A}^\vee$ as 
in Example \ref{abpdiv}.
\end{exmp}

\subsubsection{Theory of the Reflex Norm}
CM $p$-divisible groups carry an analogous theory of the reflex norm to 
abelian varieties. We begin with the definition of a reflex field for a CM 
$p$-divisible group and recall that we assume $E$ to be a ($p$-adic) CM 
field.

\begin{defn}
\label{cochar}
Let $\{\xi_i\}_{i\in \mathrm{Hom}_{\QQ_p\text{-}\mathrm{alg}}(E,\overline{\QQ}_p)}$ 
be a $\ZZ_p$-basis of the character group of $\mathrm{Res}_{E/\QQ_p}\mathbb{G}_m$. 
Let 
$$\mu_{\Phi,E}:\mathbb{G}_m\to \mathrm{Res}_{E/\QQ_p}\mathbb{G}_m$$ 
be defined such that
$$\langle \xi_i,\mu_{\Phi,E}\rangle = \left\{\begin{array}{ll}
1&\text{if }i\in \Phi,\\
0&\text{if }i\notin \Phi.\end{array}\right.$$
The minimal field of definition for $\mu_{\Phi,E}$ is the \textbf{($p$-adic) reflex field} 
$E^*$ of $E$.
\end{defn}

\begin{exmp}
Let $\mathcal{A}/\OO_K$ be a local integral model of an abelian variety 
and suppose it has CM by the global pair $(E,\Phi)$ and let $E^*$ be the corresponding 
global reflex 
field. Then the ($p$-adic) reflex field of $\mathcal{A}[p^\infty]$ is 
$E^*_p:=E^*\otimes_{\QQ}\QQ_p$. This follows from the compatibility of 
$h_{\Phi}$ with the local cocharacter $\mu_{\Phi_p}$.
\end{exmp}

\begin{defn}
\label{above}
Let $K/E^*$ be any finite field extension of the 
($p$-adic) reflex field $E^*$ of the ($p$-adic) CM field $(E,\Phi)$, and let 
$\mu:\mathbb{G}_m\to (\mathrm{Res}_{E/\QQ_p}\mathbb{G}_m)_{K}$ be the 
$K$-descent of $\mu_\Phi$, as in Definition \ref{cochar}. Then the 
\textbf{($p$-adic) reflex norm} 
$$N\mu_{\Phi}:\mathrm{Res}_{K/\QQ_p}\mathbb{G}_m\to \mathrm{Res}_{E/\QQ_p}\mathbb{G}_m$$
is the composition
$$\mathrm{Res}_{K/\QQ_p}\mathbb{G}_m\xrightarrow{\mathrm{Res}_{K/\QQ_p}(\mu)}\mathrm{Res}_{K/\QQ_p}((\mathrm{Res}_{E/\QQ_p}\mathbb{G}_m)_{K})\xrightarrow{\mathrm{Nm}_{K/\QQ_p}}\mathrm{Res}_{E/\QQ_p}(\mathbb{G}_m).$$
\end{defn}

\begin{rem}
For any $\QQ_p$-algebra $R$, the ($p$-adic) reflex norm induces a map
\begin{align*}
N_{\Phi,K\otimes R}:(K\otimes R)^\times&\to (E\otimes R)^\times\\
x&\mapsto\mathrm{det}_{E\otimes R}(x|_{V_{\Phi,K}\otimes_{\QQ_p}R})\end{align*}
where $V_{\Phi,K}$ is any $E\otimes_{\QQ_p} K$-module satisfying $V_{\Phi,K}\otimes_{K}\overline{\QQ}_p\simeq \prod_{i\in \Phi}\overline{\QQ}_{p,i}$. 
\end{rem}

We will assume in what follows that the fields $E$, $E^*$, and $K$ are 
as in Definition \ref{above}. Let moreover 
$r_F^{\mathrm{loc}}:F^\times\to \Gal(F^{\mathrm{ab}}/F)$ denote the local Artin 
reciprocity map from local class field theory defined on the 
units of a local field $F$, where 
$F^{\mathrm{ab}}$ is the maximal 
abelian extension of $F$. We also mention the facts, which can be found in 
\cite{Dieudonne} and \cite{Conrad}, that every Galois representation attached to a 
$p$-divisible group is crystalline and moreover all $1$-dimensional 
crystalline representations are locally algebraic.

\begin{lem}
\label{localMT}
Let $\rho:\Gal(K^{\mathrm{ab}}/K)\to E^\times$ 
be a crystalline (i.e., locally algebraic) representation on a $1$-dimensional 
$E$-vector space, and let 
$\chi:\mathrm{Res}_{K/\QQ_p}\mathbb{G}_m\to \mathrm{Res}_{E/\QQ_p}\mathbb{G}_m$
be a $\QQ_p$-homomorphism. Suppose that $\rho\circ r_E^{\mathrm{loc}}$ and $\chi$ agree on a 
restriction to an open neighborhood of $\OO_K^\times\to \OO_E^\times$. 
Then the two maps agree on the entirety of $\OO_K^\times$.
\end{lem}

\begin{proof}
The assertion is a special case of a general result in $p$-adic Hodge theory 
that if an element in the representation space of a crystalline $p$-adic representation 
is fixed by an open subgroup of the inertia subgroup, then it is fixed by the 
whole inertia subgroup (see, e.g., Theorem 3.1.1 in \cite{BBM}). 
The open neighborhoods of $\OO_K^\times$ correspond 
to open subgroups of the inertia 
group by local class field theory.
\end{proof}

\begin{cor}[\cite{Conrad} Corollary 3.4.9]
\label{cristallinerep}
Every crystalline (i.e., locally algebraic) representation on a $1$-dimensional $E$-vector space 
arising from a CM $p$-divisible group is essentially given by the reflex 
norm. Precisely, for $L = K.W(\overline{k})$, a representation
$$\rho:\Gal(\overline{L}/L)\to E^\times$$
attached to a CM $p$-divisible group $\mathcal{G}$ equals the composition
$$\Gal(L^\mathrm{ab}/L)\to I_{E^*}^{\mathrm{ab}}\xleftarrow[\sim]{r_{E^*}^{\mathrm{loc}}}\OO_{E^*}^\times\xrightarrow{1/N\mu_{\Phi}}E^\times,$$
where $I_{E^*}^{\mathrm{ab}}$ is the abelian part of the inertia group of $E^*$.
\end{cor}
\begin{proof}
Here we lay out the ideas; for details we refer to 
Proposition 3.4.3, Proposition 3.4.4, and Corollary 3.4.9 in \cite{Conrad} 
for the various parts. 
The strategy is to construct a CM abelian variety which locally at $p$ 
is isogenous to the given CM $p$-divisible group. Then the main theorem 
of complex multiplication for abelian varieties 
demonstrates that the restriction of $\rho$ to the inertia subgroup 
$I_L^{\mathrm{ab}}\subseteq \Gal(L^\mathrm{ab}/L)$ equals the reciprocal 
of the composition 
$$I_L^\mathrm{ab}\xrightarrow{\simeq}I_{E^*}^{\mathrm{ab}}\xleftarrow[\simeq]{r_{E^*}^{\mathrm{loc}}}\OO_{E^*}^\times\xrightarrow{N\mu_{\Phi}}E^\times.$$
We conclude with an isogeny between the given $p$-divisible group and one with 
an unramified twist so the image of the representation is given 
by inertia. Note that $L$ in the statement may be replaced by a 
finite extension of $\QQ_p$ 
given by the compositum $K.W(k')$ for a large enough field 
extension $k'\supseteq k$, so long as the appropriate unramified 
twist of $\mathcal{G}$ is defined over $K.W(k')$.
\end{proof}

We conclude by remarking 
that Lemma \ref{localMT} and Corollary \ref{cristallinerep} 
are the local ($p$-divisible group) analogues of 
the Main Theorem of Complex Multiplication for abelian varieties.

\section{Integral $p$-adic Hodge Theory}

\subsection{Deformation Theory of $p$-divisible Groups}
\label{BrSEC}
The original approach to classify $p$-divisible groups by Messing in 
\cite{Messing} was to use their deformation theory. This 
recovers useful invariants pertaining to a $p$-divisible group, 
such as the Hodge bundle $\omega_{\mathcal{G}}$ 
and the Lie algebra $\mathrm{Lie}(\mathcal{G})$. In subsequent work, 
Breuil \cite{Breuil} developed an algebraic 
theory of what are now called \emph{Breuil modules} to capture Messing's 
abstract theory in computations. 
Here we summarize some of these developments by \cite{Messing} and \cite{Breuil}, 
and show how to write the Hodge bundle $\omega_{\mathcal{G}}$ 
using Breuil modules.

Let $k$ be a perfect field of characteristic $p>0$ and $W = W(k)$ be its 
ring of Witt vectors. Let $T_0$ be a $W$-scheme on which $p$ is nilpotent and 
$T$ be a $W$-scheme on which $p$ is locally-nilpotent such that $T_0\hookrightarrow T$ 
is a closed $W$-subscheme. Then define 
$\mathcal{G}_0$ to be a $p$-divisible group on $T_0$ and $\mathcal{G}$ to be a lift 
of $\mathcal{G}_0$ to $T$, i.e., 
$\mathcal{G}$ is a \emph{$T$-deformation of $\mathcal{G}_0$}. 
We recall that such a $\mathcal{G}$ has a natural \emph{relative Frobenius} 
morphism acting on it, which is the pullback of the Frobenius acting 
on the base. 

In \cite{Messing}, Messing constructs a contravariant functor 
from $p$-divisible groups $\mathcal{G}$ defined over $T$ to 
\emph{deformation pairs}: 
$$\mathcal{G}\mapsto (\mathbb{D}_T(\mathcal{G}),\mathrm{Fil}^1(\mathbb{D}_T(\mathcal{G}))),$$
where the deformation pair consists of a Frobenius crystal $\mathbb{D}_T(\mathcal{G})$ on the crystalline site of $T$ together with a 
filtered piece $\mathrm{Fil}^1(\mathbb{D}_T(\mathcal{G}))\subseteq \mathbb{D}_T(\mathcal{G})$. 
For the definition of Frobenius crystals we encourage the reader to see 
\cite{Messing}. We summarize the properties we 
need about crystals from \emph{loc. cit.} in the following lemma.

\begin{lem}[\cite{Messing}]
\label{properties}
Let $T$ be a $W$-scheme on which $p$ is locally-nilpotent, $T_0\hookrightarrow T$ be a closed 
$W$-subscheme on which $p$ is nilpotent, and $T'\to T$ be any crystalline homomorphism of $W$-schemes.
\begin{enumerate}
\item (Crystalline Base Change). $\mathbb{D}_{T'}\simeq \mathbb{D}_T\otimes_T T'$, i.e., the Frobenius crystal 
is compatible with arbitrary crystalline base change.
\item (Frobenius Linearization). When $p$ is nilpotent on $T$, the relative Frobenius $\phi$ on $\mathcal{G}$ induces a map
$$\phi^*(\mathbb{D}_T(\mathcal{G}))\xrightarrow{\sim}\mathbb{D}_T(\phi^*(\mathcal{G}))\to \mathbb{D}_T(\mathcal{G}).$$
\item (Crystal Growth). If $\mathcal{G}/T$ is a lift of $\mathcal{G}_0/T_0$ then
$$\mathbb{D}_{T_0}(\mathcal{G}_0)(T)\xrightarrow{\sim}\mathbb{D}_{T}(\mathcal{G})(T).$$
\item (Isomorphism Invariance). If $\mathcal{G}$ and $\mathcal{G}'$ are non-isomorphic $p$-divisible 
groups over $T$, then a given isomorphism of crystals 
$\mathbb{D}_T(\mathcal{G}) \xrightarrow{\sim} \mathbb{D}_T(\mathcal{G}')$ does 
not map $\mathrm{Fil}^1(\mathbb{D}_T(\mathcal{G}))$ onto 
$\mathrm{Fil}^1(\mathbb{D}_T(\mathcal{G}'))$.
\item (Compatibility with Duality). $\mathbb{D}_T$ admits a dual $\mathbb{D}_T^*$ satisfying 
$\mathbb{D}_T^*(\mathcal{G}) = \mathbb{D}_T(\mathcal{G}^*)$, 
where $\mathcal{G}^*$ denotes the Cartier dual of $\mathcal{G}$.
\end{enumerate}
\end{lem}

The deformation exact sequence of a $p$-divisible group $\mathcal{G}$ 
defined over $T$ is given by 
$$0\to \mathrm{Lie}(\mathcal{G})^\vee\to \mathbb{D}_T(\mathcal{G})(T)\to \mathrm{Lie}(\mathcal{G}^*)\to 0$$
where $\mathrm{Lie}(\mathcal{G})^\vee$ 
denotes the $\mathcal{O}_T$-linear dual of the module 
$\mathrm{Lie}(\mathcal{G})$ and $\mathcal{G}^*$ denotes the Cartier dual of the group $\mathcal{G}$. 
By Lemma \ref{properties}(3), this sequence corresponds to a 
deformation pair 
$(\mathbb{D}_{T_0}(\mathcal{G}_0)(T),\mathrm{Lie}(\mathcal{G})^\vee)$, and by Lemma \ref{properties}(1) and (4), 
each $T$-deformation of $\mathcal{G}_0/T_0$ uniquely determines such a pair up to crystalline 
base change. It turns out that this completely classifies all $T$-deformations of $\mathcal{G}_0/T_0$.

\begin{thm}[Messing's Thesis \cite{Messing}]
Let $T$ be a $W$-scheme on which $p$ is locally-nilpotent and $T_0\hookrightarrow T$ be a 
closed $W$-subscheme on which $p$ is nilpotent. Then there is a one-to-one correspondence
$$\{T\text{-deformations of }\mathcal{G}_0/T_0\}\leftrightarrow \{\text{pairs }(\mathbb{D}_{T_0}(\mathcal{G}_0)(T),L)\}$$
where $L$ is an $\OO_T$-submodule such that the quotient $\mathbb{D}_{T_0}(\mathcal{G}_0)(T)/L$ is $\OO_T$-free.
In particular, a given deformation $\mathcal{G}/T$ corresponds to the pair 
$(\mathbb{D}_T(\mathcal{G})(T),\mathrm{Lie}(\mathcal{G})^\vee)$.
\end{thm}

Breuil's theory encodes the abstract data of the pairs 
$(\mathbb{D}^*_{T_0}(\mathcal{G})(T),\mathrm{Lie}(\mathcal{G})^\vee)$ to 
computable 
algebraic structures now known as 
Breuil modules. We retain that $k$ is a perfect field of characteristic $p>0$ 
and $W = W(k)$ its ring of Witt vectors with fraction field denoted by 
$K_0$. Fix $K/K_0$ to be a finite, totally ramified extension with characteristic 
polynomial given by $\mathsf{Eis}(u) = u^e + a_{e-1}u^{e-1} + \ldots + a_0$, 
where we note that $p|a_i$ for all $0\le i\le e-1$ and $p^2\nmid a_0$.  
We define $S$ to be the $p$-adic completion of 
$W[u,\mathsf{Eis}(u)^n/n!]_{n\ge 1}$. 
Then $S$ has a natural Frobenius endomorphism $\phi$ extending the 
Frobenius on $W$ by mapping $u\mapsto u^p$. The ideal 
$\mathrm{Fil}^1(S) = \mathrm{ker}(S\xrightarrow{u\mapsto \pi} \OO_K)$ 
is also naturally equipped with divided powers so that $S$ is an object 
in the crystalline site of $\OO_K$. We also define the ring 
$S_\infty = S\otimes_W K_0/W$.

\begin{defn}
Let $\mathbf{BT}_{/S}^\phi$ be the category of pairs 
$(\mathcal{M},\mathrm{Fil}^1(\mathcal{M}))$, where $\mathcal{M}$ is a finite free 
$S$-module and $\mathrm{Fil}^1(\mathcal{M})\subseteq \mathcal{M}$ 
is an $S$-submodule, satisfying:
\begin{enumerate}
\item $\mathrm{Fil}^1(S)\cdot \mathcal{M}\subseteq \mathrm{Fil}^1(\mathcal{M})$,
\item $\mathcal{M}/\mathrm{Fil}^1(\mathcal{M})$ is a free $\OO_K$-module, and
\item there exists a $\phi$ semi-linear map 
$$\phi_\mathcal{M}:\mathrm{Fil}^1(\mathcal{M})\to \mathcal{M}$$
such that its induced map $\phi^*(\mathrm{Fil}^1(\mathcal{M}))\to \mathcal{M}$ is 
surjective.
\end{enumerate}
An object in this category is called a \textbf{Breuil module} and the 
map $\phi_{\mathcal{M}}$ its \textbf{Frobenius}. The morphisms 
in this category are natural morphisms of modules which preserve the 
Frobenius $\phi_{\mathcal{M}}$ and filtrations.
\end{defn}

We characterize the following 
objects in the category.

\begin{itemize}
\item An object $(\mathcal{M},\mathrm{Fil}^1\mathcal{M})\in\mathbf{BT}_{/S}^\phi$ is \emph{connected} if
$$m\mapsto \phi_{\mathcal{M}}(\mathsf{Eis}(u)m)$$
is topologically nilpotent on $\mathcal{M}/u\mathcal{M}$ for the $p$-adic 
topology.

\item The \emph{dual object} to a pair $(\mathcal{M},\mathrm{Fil}^1\mathcal{M})\in\mathbf{BT}_{/S}^\phi$ is defined as
$$(\mathcal{M}^*,\,\mathrm{Fil}^1\mathcal{M}^*):=\left(\mathrm{Hom}_S(\mathcal{M},S), \,\{f\in \mathcal{M}^*:f(\mathrm{Fil}^1\mathcal{M})\subseteq \mathrm{Fil}^1S_\infty\}\right)$$
and for all $f\in \mathrm{Fil}^1\mathcal{M}^*$, the Frobenius 
$\phi_{\mathcal{M}}^*$ is the unique morphism making the following 
diagram commute

\bigskip
\centerline{
\begin{xy}
(0,15)*+{\mathrm{Fil}^1\mathcal{M}}="a";
(20,15)*+{\mathcal{M}}="b";
(0,0)*+{\mathrm{Fil}^1 S}="c";
(20,0)*+{S}="d";
{\ar^{\phi_{\mathcal{M}}} "a";"b"};{\ar^{\phi_S} "c";"d"};
{\ar_f "a";"c"};{\ar^{\phi_\mathcal{M}^*(f)} "b";"d"};
\end{xy}.}
\bigskip

\noindent For a proof of uniqueness, we refer to Caruso \cite{Caruso}.

\end{itemize}

In what follows, we drop the subscript on the functor $\mathbb{D}$ for 
clarity of notation.

\begin{thm}[Breuil's Classification, \cite{Breuil}]
\label{Breuil}
Let $\mathcal{G}/\OO_K$ be a $p$-divisible group. Then there is 
a functor
\begin{align*}
\mathcal{M}:p\mathbf{-div}/\OO_K&\to \mathbf{BT}_{/S}^\phi\\
\mathcal{G}&\mapsto \left(\mathcal{M}(\mathcal{G}) = \mathbb{D}(\mathcal{G})(S), \;\mathrm{Fil}^1(\mathcal{M}(\mathcal{G})) = \mathrm{Fil}^1(\mathbb{D}(\mathcal{G}))(S)\right)\end{align*}
which induces an equivalence of categories when $p>2$. When $p = 2$, 
this is an equivalence on connected objects.
\end{thm}

Note that $\mathcal{M}$ identifies connected objects in each category and 
is compatible with duality, i.e., 
$\mathcal{M}(\mathcal{G}^*) = \mathcal{M}(\mathcal{G})^*$. These 
assertions are shown in \cite{Caruso}.

\begin{exmps}
Although we will not work with Breuil modules directly, it is helpful for the reader to 
see the following examples (see also \cite{CPadic}).
\begin{enumerate}
\item Let $\mathcal{G} = \varinjlim p^{-n}\mathbb{Z}/\mathbb{Z} = \QQ_p/\ZZ_p$. Then
$$\mathcal{M}(\mathcal{G}) = Se,\quad \mathrm{Fil}^1(\mathcal{M}(\mathcal{G})) = Se,\quad \phi_{\mathcal{M}(\mathcal{G})} = \phi,$$
where $e$ denotes a generator of the $S$-modules $\mathcal{M}(\mathcal{G})$.
\item Let $\mathcal{G} = \varinjlim \mu_{p^n} = \mu_{p^{\infty}}$. Then
$$\mathcal{M}(\mathcal{G}) = Se,\quad \mathrm{Fil}^1(\mathcal{M}(\mathcal{G})) = \mathrm{Fil}^1(S)e,\quad \phi_{\mathcal{M}(\mathcal{G})} = \frac{1}{p}\phi$$
where $e$ denotes a generator of the $S$-modules $\mathcal{M}(\mathcal{G})$.
\end{enumerate}
\end{exmps}

\begin{cor}
\label{deformation}
Let $\mathcal{M} = \mathcal{M}(\mathcal{G})$ and 
$\mathcal{M}^* =\mathcal{M}(\mathcal{G})^*$. Then
\begin{align*}\omega_{\mathcal{G}}&\simeq \mathrm{Fil}^1\mathcal{M}/(\mathrm{Fil}^1S)\mathrm{Fil}^1\mathcal{M},\\
\mathrm{Lie}(\mathcal{G})&\simeq \mathcal{M}^*/\mathrm{Fil}^1\mathcal{M}^*.\end{align*}
\end{cor}
\begin{proof}
By Theorem \ref{Breuil} we obtain the isomorphism of exact sequences

\bigskip
\centerline{
\begin{xy}
(0,15)*+{0}="a";
(20,15)*+{\mathrm{Lie}(\mathcal{G})^\vee}="b";
(50,15)*+{\mathbb{D}(\mathcal{G})(S)}="c";
(80,15)*+{\mathrm{Lie}(\mathcal{G}^*)}="d";
(100,15)*+{0}="e";
( 0,0)*+{0}="f";
(20,0)*+{\mathrm{Fil}^1\mathcal{M}}="g";
(50,0)*+{\mathcal{M}}="h";
(80,0)*+{\mathcal{M}/\mathrm{Fil}^1\mathcal{M}}="i";
(100,0)*+{0}="j";
{\ar "a";"b"};{\ar "b";"c"};{\ar "c";"d"};{\ar "d";"e"};
{\ar "f";"g"};{\ar "g";"h"};{\ar "h";"i"};{\ar "i";"j"};
{\ar_{\simeq} "b";"g"};{\ar_{\simeq} "c";"h"};{\ar_{\simeq} "d";"i"};
\end{xy}}
\bigskip

\noindent where we recall that we use the convention 
that $\mathrm{Lie}(\mathcal{G})^\vee$ denotes the 
$\OO_S$-linear dual of the first filtered piece of 
$\mathbb{D}(\mathcal{G})$ (see the paragraph 
following Lemma \ref{properties}). 
By Lemma \ref{properties}(1), since $S$ is in the crystalline 
site of $\OO_K$, we also obtain the exact diagram

\bigskip
\centerline{
\begin{xy}
(5,15)*+{0}="a";
(30,15)*+{\omega_{\mathcal{G}}}="b";
(65,15)*+{\mathbb{D}(\mathcal{G})(\OO_K)}="c";
(95,15)*+{\mathrm{Lie}(\mathcal{G}^*)}="d";
(110,15)*+{0}="e";
( 5,0)*+{0}="f";
(30,0)*+{\mathrm{Fil}^1\mathcal{M}/(\mathrm{Fil}^1S)\mathrm{Fil}^1\mathcal{M}}="g";
(65,0)*+{\mathcal{M}/(\mathrm{Fil}^1S)\mathcal{M}}="h";
(95,0)*+{\mathcal{M}/\mathrm{Fil}^1\mathcal{M}}="i";
(110,0)*+{0}="j";
{\ar "a";"b"};{\ar "b";"c"};{\ar "c";"d"};{\ar "d";"e"};
{\ar "f";"g"};{\ar "g";"h"};{\ar "h";"i"};{\ar "i";"j"};
{\ar_{\simeq} "b";"g"};{\ar_{\simeq} "c";"h"};{\ar_{\simeq} "d";"i"};
\end{xy}.}
\bigskip

\noindent The first vertical arrow is the isomorphism 
$\omega_{\mathcal{G}}\simeq \mathrm{Fil}^1\mathcal{M}/(\mathrm{Fil}^1S)\mathrm{Fil}^1\mathcal{M}$, and, repeating the construction with $\mathcal{G}^*$ 
in place of $\mathcal{G}$, the last vertical arrow is the isomorphism 
$\mathrm{Lie}(\mathcal{G})\simeq \mathcal{M}^*/\mathrm{Fil}^1\mathcal{M}^*$.
\end{proof}

\subsection{Kisin Modules}
A modern approach to classifying $p$-divisible groups is via their 
\emph{Kisin modules}. This improves on the classification by Messing and 
Breuil reviewed in the previous section in that it gives also a 
complete classification of all integral and torsion crystalline representations, not just 
those arising from $p$-divisible groups. However, to deduce the Hodge theory of 
$p$-divisible groups, one needs the correspondence between 
Kisin modules and Breuil modules (we do this in Section \ref{hodgeBK}).

Let $k$ be a perfect field of characteristic $p>0$ and 
$W=W(k)$ be its ring of Witt vectors. We construct a ring 
$\mathfrak{S} = W[[u]]$ which has a natural Frobenius endomorphism $\phi$ 
extending the Frobenius on $W$ by sending $u\mapsto u^p$. 
We also fix $K$ to be a finite, totally ramified extension of $K_0:=W[1/p]$ 
with ring of integers $\OO_K$ and we fix a choice of uniformizer 
$\pi_K$. Then let $\mathsf{Eis}(u)\in W[u]$ denote the 
Eisenstein polynomial corresponding to the minimal polynomial of the element $\pi_K$. 
In particular, $\mathsf{Eis}(u) = u^e + a_{e-1}u^{e-1} + \cdots + a_1u + a_0$ 
where $e$ is the ramification index of $K/K_0$, $p|a_i$ for all $i$, and 
$a_0 = cp$ for some $c\in W^\times$. We finally let 
$\OO_{C^\flat} = \varprojlim \OO_{C}/p$, 
where $C$ is an algebraic closure of $K$ and the 
limit is over maps $\alpha\mapsto \alpha^p$. We note that 
$\mathfrak{S}\hookrightarrow W(\OO_{C^\flat})$ with compatible Frobenius structures. 

\begin{defn}
Let $\mathbf{Mod}_{/\mathfrak{S}}^\phi$ be the category of finite 
$\mathfrak{S}$-modules $\mathfrak{M}$ equipped with a 
$\phi$ semi-linear isomorphism
$$1\otimes \phi_{\mathfrak{M}}:\phi^*(\mathfrak{M})\left[\frac{1}{\mathsf{Eis}(u)}\right]\xrightarrow{\sim}\mathfrak{M}\left[\frac{1}{\mathsf{Eis}(u)}\right]$$
where $\phi^*(\mathfrak{M}) := \mathfrak{S}\otimes_{\phi}\mathfrak{M}$. 
An object in this category is called a \textbf{Kisin module} 
and its endomorphism $\phi_\mathfrak{M}$ its \textbf{Frobenius}. Morphisms 
in this category are the naturally-defined morphisms of modules which commute 
with the semi-linear Frobenius isomorphism.
\end{defn}

We also introduce a category of modules which will behave as the 
``generic fiber'' of the Kisin modules defined above. We let 
$\OO_{\mathcal{E}}$ denote the $p$-adic completion of 
$\mathfrak{S}[1/u]$, which inherits its Frobenius map $\phi$ from 
$\mathfrak{S}$, and $\mathcal{E}$ be its fraction field. Note that 
the residue field of $\OO_{\mathcal{E}}$ is $k((u))$.

\begin{defn}
Let $\mathbf{Mod}_{/\OO_{\mathcal{E}}}^\phi$ be the category of 
finite $\mathcal{O}_{\mathcal{E}}$-modules $\mathit{M}$ equipped 
with a Frobenius semi-linear isomorphism
$$1\otimes \phi_{\mathit{M}}:\phi^*\mathit{M}\xrightarrow{\sim}\mathit{M}$$
where 
$\phi^*\mathit{M}:=\OO_{\mathcal{E}}\otimes_{\phi}\mathit{M}$. 
An object in this category is an \textbf{\'etale $\OO_{\mathcal{E}}$-module} 
and its endomorphism $\phi_M$ its \textbf{Frobenius}. 
Morphisms are morphisms of modules which commute with $\phi_{\mathit{M}}$.
\end{defn}

\begin{rem}
We will sometimes drop the subscript $\mathfrak{M}$ or 
$\mathit{M}$ of $\phi_{\mathfrak{M}}$ or $\phi_{\mathit{M}}$ in 
the definitions when the Kisin module or \'etale $\OO_{\mathcal{E}}$-module 
is specified or where its 
retention clutters notation. The distinction of $\phi$ from $\phi_{\mathfrak{M}}$ 
or $\phi_\mathit{M}$ will be clear in those situations.
\end{rem}

\begin{thm}[\cite{Kisin}, Proposition 2.1.12]
\label{GF}
The functor
\begin{align*}
\mathrm{GF}:\mathbf{Mod}_{/\mathfrak{S}}^\phi&\to \mathbf{Mod}_{/\mathcal{O}_{\mathcal{E}}}^\phi\\
\mathfrak{M}&\mapsto\mathfrak{M}\otimes_{\mathfrak{S}}\mathcal{O}_{\mathcal{E}}\end{align*}
is fully faithful.
\end{thm}

In the category $\textbf{Mod}_{/\OO_{\mathcal{E}}}^\phi$, we can 
define the dual object to an object 
$\mathit{M}\in\textbf{Mod}_{/\OO_{\mathcal{E}}}^\phi$ as the module
$$\mathit{•}{M}^\vee:=\mathrm{Hom}_{\mathcal{E},\phi}(\mathit{M},\mathcal{O}_{\mathcal{E}})$$
with Frobenius endomorphism induced by
$$1\otimes \phi_{\mathit{M}^\vee}:\phi^*\mathit{M}^\vee =\left(\phi^*\mathit{M}\right)^\vee \xrightarrow{((1\otimes\phi_{\mathit{M}})^{\vee})^{-1}}\mathit{M}^\vee.$$
Note that duality induces an involution on all of 
$\mathbf{Mod}_{/\mathcal{O}_\mathcal{E}}^\phi$.

One cannot, however, specify an involution by duality on all of 
$\mathbf{Mod}_{/\mathfrak{S}}^\phi$, so we restrict to the subcategory 
$\mathbf{Mod}_{/\mathfrak{S}}^{\phi,\mathrm{fr}}\subseteq \mathbf{Mod}_{/\mathfrak{S}}^\phi$ 
of $\mathfrak{S}$-free modules to make this 
definition. Then given $\mathfrak{M}\in \mathbf{Mod}_{/\mathfrak{S}}^{\phi,\mathrm{fr}}$, its dual is the module
$$\mathfrak{M}^\vee:=\mathrm{Hom}_{\mathfrak{S}}(\mathfrak{M},\mathfrak{S})$$
with Frobenius
\begin{align*}
1\otimes \phi_{\mathfrak{M}^\vee}&:\phi^*\mathfrak{M}^\vee\left[\frac{1}{\mathsf{Eis}(u)}\right] =\left(\phi^*\mathfrak{M}\left[\frac{1}{\mathsf{Eis}(u)}\right]\right)^\vee \\
&\quad\quad\quad\quad\quad\quad\quad\quad\quad\quad\quad\quad\xrightarrow{((1\otimes\phi_{\mathfrak{M}})^{\vee})^{-1}}\left(\mathfrak{M}\left[\frac{1}{\mathsf{Eis}(u)}\right]\right)^\vee = \mathfrak{M}^\vee\left[\frac{1}{\mathsf{Eis}(u)}\right].\end{align*}
The reader can check that duality is 
preserved under $\mathrm{GF}$.

Let now $\mathbf{Rep}_{G_K}^{\mathrm{cris}\circ}$ denote the category of 
crystalline $G_K$-stable $\ZZ_p$-lattices spanning a crystalline 
$G_K$-representation.

\begin{thm}[\cite{Kisin},\cite{Liu}]
\label{Mfunc}
There is a fully faithful tensor functor
$$\mathfrak{M}:\mathbf{Rep}_{G_K}^{\mathrm{cris}\circ}\to \mathbf{Mod}_{/\mathfrak{S}}^{\phi}$$
which has the following properties. For any $L\in\mathbf{Rep}_{G_K}^{\mathrm{cris}\circ}$:

\begin{enumerate}
\item (Preservation of Rank). $\mathfrak{M}(L)$ is finite free over $\mathfrak{S}$, of rank equal to the 
rank of $L$.
\item(Compatibility with Tensor Powers).
There are canonical $\phi$-equivariant isomorphisms
\begin{align*}
\mathrm{Sym}^n(\mathfrak{M}(L))&\simeq \mathfrak{M}(\mathrm{Sym}^n(L)),\\
\bigwedge^n(\mathfrak{M}(L))&\simeq \mathfrak{M}(\bigwedge^n(L)).\end{align*}
\item(Compatibility with Unramified Base Change). 
If $k'/k$ is an algebraic extension of fields, then there exists a canonical $\phi$-equivariant isomorphism
$$\mathfrak{M}(L|_{G_{K'}})\xrightarrow{\sim}\mathfrak{M}(L)\otimes_{\mathfrak{S}}\mathfrak{S}'$$
where $\mathfrak{S}' = W(k')[[u]]$ and $G_{K'} = \Gal(\overline{K}\cdot W(k')[1/p]/K\cdot W(k')[1/p])$. In particular,
$$\mathfrak{M}(L|_{I_K})\xrightarrow{\sim}\mathfrak{M}(L)\otimes_{\mathfrak{S}}\mathfrak{S}^{\mathrm{ur}}$$
where $\mathfrak{S}^{\mathrm{ur}} := W(\overline{k})[[u]]$.
\item(Compatibility with Formation of Duals).
Let $L^\vee$ denote the dual representation to $L$. Then there exists a canonical 
$\phi$-equivariant isomorphism
$$\mathfrak{M}(L^\vee)\simeq \mathfrak{M}(L)^\vee.$$
\end{enumerate}
\end{thm}

The functor $\mathfrak{M}$ in the theorem easily extends to the category  
$\mathbf{Rep}^{\mathrm{cris},\mathrm{tors}}_{G_K}$ of 
torsion semi-stable $G_K$-modules with non-negative Hodge-Tate 
weights. We note that 
objects in $\mathbf{Rep}^{\mathrm{cris},\mathrm{tors}}_{G_K}$ have a 
2-term resolution by objects in 
$\mathbf{Rep}_{G_K}^{\mathrm{cris}\circ}$. Therefore, the main content 
of the following theorem is that $\mathfrak{M}$ extends to a functor 
on the category $\mathbf{Rep}^{\mathrm{cris},\mathrm{tors}}_{G_K}$ so that 
its image in $\textbf{Mod}_{/\mathfrak{S}}^\phi$ is 
independent of this choice of resolution.

\begin{thm}[\cite{Kisin},\cite{Liu}]
\label{extend}
The functor $\mathfrak{M}$ in Theorem \ref{Mfunc} extends to a 
fully faithful functor
$$\mathfrak{M}:\mathbf{Rep}_{G_K}^{\mathrm{cris},\mathrm{tors}}\to \mathbf{Mod}_{/\mathfrak{S}}^\phi$$
which is compatible with the formation of symmetric and tensor powers, 
unramified base change, and with the formation of duals.
\end{thm}

The theory of duality on torsion Kisin in this theorem 
is due to Caruso \cite{Caruso} and Liu \cite{Liutors}. 
Objects in $\textbf{Rep}_{G_K}^{\mathrm{cris},\mathrm{tors}}$ have image in 
the subcategory $\textbf{Mod}_{/\mathfrak{S}}^{\phi,\mathrm{tors}}\subseteq \textbf{Mod}_{/\mathfrak{S}}^{\phi}$, consisting of $p$-power torsion objects which 
have a $2$-term resolution by objects in 
$\textbf{Mod}_{/\mathfrak{S}}^{\phi,\mathrm{fr}}$. 
Then given $L_1,L_2\in \textbf{Rep}_{G_K}^{\mathrm{cris}\circ}$ and 
$L\in \textbf{Rep}_{G_K}^{\mathrm{cris,tors}}$ related by the 
short exact sequence
$$0\to L\to L_1\to L_2\to 0,$$
duality is the unique involution making the 
following diagram commute

\bigskip
\centerline{
\begin{xy}
(15,15)*+{0}="a";
(30,15)*+{\mathfrak{M}(L_1^\vee)}="b";
(50,15)*+{\mathfrak{M}(L_2^\vee)}="c";
(70,15)*+{\mathfrak{M}(L^\vee)}="d";
(85,15)*+{0}="e";
(15,0)*+{0}="f";
(30,0)*+{\mathfrak{M}(L_1)^\vee}="g";
(50,0)*+{\mathfrak{M}(L_2)^\vee}="h";
(70,0)*+{\mathfrak{M}(L)^\vee}="i";
(85,0)*+{0}="j";
{\ar "a";"b"};{\ar "b";"c"};{\ar "c";"d"};{\ar "d";"e"};
{\ar "f";"g"};{\ar "g";"h"};{\ar "h";"i"};{\ar "i";"j"};
{\ar_{\simeq} "b";"g"};{\ar_{\simeq} "c";"h"};{\ar_{\simeq} "d";"i"};
\end{xy}.}
\bigskip

\noindent Liu provides an intrinsic definition of duality on the category 
$\mathbf{Mod}_{/\mathfrak{S}}^{\phi,\mathrm{tors}}$ in 
\cite{Liutors}, which is shown to be compatible by the induced 
duality above. \cite{Liutors} also 
shows that $\mathrm{GF}$ restricted to this subcategory 
has image in  
$\mathbf{Mod}_{\OO_{\mathcal{E}}}^{\phi,\mathrm{tors}}$ 
consisting of $p$-power torson \'etale $\OO_{\mathcal{E}}$-modules. 
The above definition of duality commutes as well with 
duality on 
$\mathbf{Mod}_{\OO_{\mathcal{E}}}^{\phi,\mathrm{tors}}$.

Up to this point, the definition of a Kisin module relied on a choice of uniformizer 
and, therefore, 
using compatible choices of uniformizers to define duality and base change. It 
is a technical theorem of Liu \cite{Liu3} that defining a Kisin module is independent 
of this choice, and therefore compatible choices do not have to 
be made. We state the theorem below for reference.

\begin{thm}[Theorem 1.0.1, \cite{Liu3}]
\label{unifindep}
Let $\mathfrak{M}$ and $\mathfrak{M}'$ both be Kisin modules 
associated to the same object $L\in \mathrm{Rep}_{G_K}^{\mathrm{cris,}\circ}$ 
but constructed on different choices of 
uniformizers $\pi_K$ and $\pi'_K$ of $K$. Then 
$$W(\OO_{C^\flat})\otimes_{\phi,\pi_K}\mathfrak{M}\simeq W(\OO_{C^\flat})\otimes_{\phi,\pi'_{K}}\mathfrak{M}'$$
where the tensor product on each side is with respect to the embedding
$\mathfrak{S}\hookrightarrow W(\OO_{C^\flat})$ 
given by  
$u\mapsto [\pi]$ along the uniformizers $\pi = \pi_K$ 
and $\pi = \pi'_{K}$. 
In particular, different choices of 
uniformizers define isomorphic Kisin modules.
\end{thm}

\subsection{Classification of Barsotti-Tate Groups by Kisin Modules}
\label{KisSEC}
We define the subcategories of Kisin modules which were considered 
in \cite{Kisin} to classify Barsotti-Tate groups. We retain the notation from 
the previous two sections.

\begin{defn}
\label{BTDef}
Let $\mathbf{BT}^{\phi,r}_{/\mathfrak{S}}$ 
be the full subcategory of $\mathbf{Mod}_{/\mathfrak{S}}^{\phi}$ 
consisting of $\mathfrak{S}$-free modules $\mathfrak{M}$ such that the 
cokernel of $1\otimes \phi_{\mathfrak{M}}: \phi^*\mathfrak{M}\to \mathfrak{M}$ 
is killed by $\mathsf{Eis}^r(u)$. Such Kisin modules are said to have \textbf{height} 
at most $r$.
\end{defn}

A Kisin module of height $r$ comes naturally equipped with another map 
$\psi_{\mathfrak{M}}:\phi^*\mathfrak{M}\to \mathfrak{M}$, called the 
\textbf{Verschiebung}, defined so that it satisfies 
$$(1\otimes \phi_{\mathfrak{M}})\circ \psi_{\mathfrak{M}}=\mathsf{Eis}^r(u).$$
The reader can check this uniquely defines $\psi_{\mathfrak{M}}$.

Let $p\text{-}\mathbf{div}/\OO_K$ denote the category of $p$-divisible groups over $\OO_K$. 
Given an object $\mathcal{G}/\OO_K$ in this category, we let 
$\mathcal{G}^*/\OO_K$ denote its Cartier dual and $T_p\mathcal{G} := \varprojlim_n \mathcal{G}[p^n]$ denote its Tate module.

\begin{thm}[\cite{Kisin}, \cite{Liu}]
\label{GroupEquiv}
The functor $\mathfrak{M}$ of Theorem \ref{Mfunc} induces a fully faithful contravariant equivalence
\begin{align*}
\mathfrak{M}:(p\text{-}\mathbf{div}/\OO_K)&\to\mathbf{BT}^{\phi,1}_{/\mathfrak{S}}\\
\mathcal{G}&\mapsto \mathfrak{M}(\mathcal{G}):=\mathfrak{M}(T_p\mathcal{G}^*).\end{align*}
The functor $\mathfrak{M}$ moreover admits an 
inverse functor $\mathcal{G}$, which gives the isomorphism of $G_{K_\infty}$-modules
$$\epsilon_{\mathfrak{M}}:\mathcal{G}(\mathfrak{M})|_{G_{K_\infty}}\to T^*_{\mathfrak{S}}(\mathfrak{M}) := \mathrm{Hom}_{\mathfrak{S},\phi}(\mathfrak{M},W(\OO_{C^\flat}))$$
where we recall that $\OO_{C^\flat} = \varprojlim_n\OO_{C}/p^n$ and $W(\OO_{C^\flat})$ is equipped with the Frobenius morphism 
coming from the Witt construction.
\end{thm}

\begin{exmps}
The following examples can be computed by elementary means.
\begin{enumerate}
\item Let $\mathcal{G} = \varinjlim p^{-n}\mathbb{Z}/\mathbb{Z} = \QQ_p/\ZZ_p$. Then 
$$\mathfrak{M}(\mathcal{G})\simeq \mathfrak{S}e,\quad \phi_{\mathfrak{M}(\mathcal{G})}(e) = \frac{1}{c}\mathsf{Eis}(u)e,\quad \psi_{\mathfrak{M}(\mathcal{G})}(e) = c\otimes e$$ 
where $e$ denotes a generator of the $\mathfrak{S}$-module $\mathfrak{M}$. 
We denote this Kisin module by $\mathfrak{S}(1)$.
\item Let $\mathcal{G} = \varinjlim \mu_{p^n} = \mu_{p^\infty}$. Then 
$$\mathfrak{M}(\mathcal{G})\simeq \mathfrak{S}e,\quad \phi_{\mathfrak{M}(\mathcal{G})}(e) = \frac{1}{c}e,\quad \psi_{\mathfrak{M}(\mathcal{G})}(e) = c\otimes \mathsf{Eis}(u)e$$ 
where $e$ denotes a generator of the $\mathfrak{S}$-module $\mathfrak{M}$. 
We denote this Kisin module by $\mathfrak{S}$.
\end{enumerate}
\end{exmps}

\begin{cor}
\label{thecor}
Let $\mathcal{G}$ be an object in $p\text{-}\mathbf{div}/\OO_K$ 
and $\mathfrak{M}(\mathcal{G})^*$ 
denote the module $\mathrm{Hom}_{\mathfrak{S}}(\mathfrak{M}(\mathcal{G}),\mathfrak{S}(1))$.
\begin{enumerate}
\item The rank of $\mathfrak{M}(\mathcal{G})$ encodes the height of the 
$p$-divisible group $\mathcal{G}$, i.e., $\mathrm{rk}_{\mathfrak{S}}\mathfrak{M}(\mathcal{G}) = \mathrm{ht}(\mathcal{G})$. Moreover, 
$\phi_{\mathfrak{M}(\mathcal{G})}$ and $\psi_{\mathfrak{M}(\mathcal{G})}$ encode the 
relative Frobenius and relative Verschiebung endomorphisms on $\mathcal{G}$.
\item $\mathfrak{M}$ commutes with $*$-duality, i.e., 
$\mathfrak{M}(\mathcal{G})^*\simeq\mathfrak{M}(\mathcal{G}^*)$. In particular, 
the Frobenius and Verschiebung endomorphisms on $\mathfrak{M}(\mathcal{G})^*$ are given by
$$\phi_{\mathfrak{M}(\mathcal{G})^*}(T):=\frac{1}{c}(1\otimes T)\circ\psi_{\mathfrak{M}(\mathcal{G})},$$
$$\psi_{\mathfrak{M}(\mathcal{G})^*}(T):=cT\circ (1\otimes \phi_{\mathfrak{M}(\mathcal{G})}).$$
\end{enumerate}
\end{cor}
\begin{proof}
\begin{enumerate}
\item Since $\mathfrak{M}$ preserves rank, 
and the rank of the Tate module of a $p$-divisible group equals the height of 
the group, the first part of the statement follows. 
The correspondence between the Frobenius and Verschiebung morphisms 
is implicit in Theorem \ref{GroupEquiv}.

\item By definition, there is a perfect pairing 
on the $\mathfrak{S}$-modules
$$\langle,\rangle_{\mathfrak{M}}:\mathfrak{M}(\mathcal{G})\times \mathfrak{M}(\mathcal{G})^*\to \mathfrak{S}(1).$$
The inverse functor $\mathcal{G}$ therefore induces 
a perfect pairing on the $G_{K_\infty}$-modules 
$$\langle,\rangle_{T^*_{\mathfrak{S}}(\mathfrak{M})}:T^*_{\mathfrak{S}}(\mathfrak{M})\times T^*_{\mathfrak{S}}(\mathfrak{M}^*)\to \mathfrak{S}(1).$$
Using the isomorphisms $\epsilon_{\mathfrak{M}}$ and 
$\epsilon_{\mathfrak{M}^*}$ from Theorem \ref{GroupEquiv}, 
we obtain the exact diagram

\bigskip
\centerline{
\begin{xy}
(0,15)*+{\mathcal{G}(\mathfrak{M})}="a";
(20,15)*+{\mathcal{G}(\mathfrak{M}^*)}="b";
(50,15)*+{\mathcal{G}(\mathfrak{S}(1))}="c";
( 0,0)*+{T_\mathfrak{S}^*(\mathfrak{M})}="d";
(20,0)*+{T_\mathfrak{S}^*(\mathfrak{M}^*)}="e";
(50,0)*+{T_\mathfrak{S}^*(\mathfrak{S}(1))}="f";
(10,15)*+{\times}="g";
(10,0)*+{\times}="h";
{\ar_{\epsilon_{\mathfrak{M}}}^{\simeq} "a";"d"};{\ar^{\epsilon_{\mathfrak{M}^*}}_{\simeq} "b";"e"};{\ar^{\simeq} "c";"f"};
{\ar "b";"c"};{\ar "e";"f"};
\end{xy}.}
\bigskip

By this diagram, to show that $*$-duality commutes 
with $\mathfrak{M}$ is equivalent to showing that $*$-duality 
commutes with $\mathcal{G}$. 
This follows from the fact that the pairing in the top row of the diagram 
is identified with pairing of the $p$-divisible group $\mathcal{G}$ with its 
Cartier dual $\mathcal{G}^*$.

The Frobenius and 
Verschiebung morphisms are computed from these diagrams.
\end{enumerate}
\end{proof}

\begin{rem}
The module $\mathfrak{M}^*$ from the above 
corollary relates to the Kisin module dual $\mathfrak{M}^\vee$ 
by a twist: $\mathfrak{M}^*\simeq \mathfrak{M}^\vee\otimes_{\mathfrak{S}}\mathfrak{S}(1)$. We prefer to introduce this notation to 
make it clear this module is identified with the Cartier dual 
of the underlying $p$-divisible group.
\end{rem}

This classification extends naturally to finite flat group schemes. 
It is well-known (see, e.g., \cite{BBM}) that any finite flat group 
scheme $\mathcal{G}$ defined over $\OO_K$ 
has a $2$-term resolution by $p$-divisible groups 
$\mathcal{G}_1$ and $\mathcal{G}_2$ each also defined over $\OO_K$.
Thus $\mathfrak{M}(\cdot)$ extends to the category of 
finite flat group schemes in such a way that on the exact sequence
$$0\to \mathcal{G}\to \mathcal{G}_1\to \mathcal{G}_2\to 0$$
it induces the commutative diagram

\bigskip
\centerline{
\begin{xy}
(15,15)*+{0}="a";
(30,15)*+{\mathfrak{M}(\mathcal{G}_2)}="b";
(50,15)*+{\mathfrak{M}(\mathcal{G}_1)}="c";
(70,15)*+{\mathfrak{M}(\mathcal{G})}="d";
(85,15)*+{0}="e";
(15,0)*+{0}="f";
(30,0)*+{\mathfrak{M}(\mathcal{G}_2)}="g";
(50,0)*+{\mathfrak{M}(\mathcal{G}_1)}="h";
(70,0)*+{\mathfrak{M}(\mathcal{G})}="i";
(85,0)*+{0}="j";
{\ar "a";"b"};{\ar "b";"c"};{\ar "c";"d"};{\ar "d";"e"};
{\ar "f";"g"};{\ar "g";"h"};{\ar "h";"i"};{\ar "i";"j"};
{\ar_{\phi_{\mathfrak{M}(\mathcal{G}_2)}} "b";"g"};{\ar_{\phi_{\mathfrak{M}(\mathcal{G}_1)}} "c";"h"};{\ar_{\phi_{\mathfrak{M}(\mathcal{G})}} "d";"i"};
\end{xy}}
\bigskip

\noindent and the construction of $\mathfrak{M}(\mathcal{G})$ is 
independent of the choice of $p$-divisible groups providing a 
resolution of $\mathcal{G}$. 
We have already seen that $\mathfrak{M}$ 
extends in this way in Theorem \ref{extend}. Here we identify the 
precise subcategory 
of $\mathbf{Mod}_{/\mathfrak{S}}^{\phi,\mathrm{tors}}$ 
which classifies finite flat group schemes.

\begin{defn}
\label{finitedef}
Let $\mathbf{Mod}^{\phi,r}_{/\mathfrak{S}}$, 
be the full subcategory of $\mathbf{Mod}_{/\mathfrak{S}}^\phi$ 
consisting of $\mathfrak{S}$-modules of projective dimension $1$ 
which are $p$-power torsion and and such that the 
cokernel of $1\otimes \phi_{\mathfrak{M}}: \phi^*\mathfrak{M}\to \mathfrak{M}$ 
is killed by $\mathsf{Eis}^r(u)$. This category is often called the category of \textbf{finite Kisin modules of height at most $r$}.
\end{defn}

Let $p\text{-}\mathbf{Gr}/\OO_K$ denote the category of finite flat group schemes over $\OO_K$ of $p$-power order.

\begin{thm}[\cite{Kisin},\cite{Liu}]
\label{equivfgroup}
Let $0\to \mathcal{G}\to \mathcal{G}_1\to \mathcal{G}_2\to 0$ be any 
resolution of the finite flat group scheme $\mathcal{G}/\OO_K$ by $p$-divisible 
groups $\mathcal{G}_1/\OO_K$ and $\mathcal{G}_2/\OO_K$. 
The functor $\mathfrak{M}$ of Theorem \ref{Mfunc} extends to induce a fully faithful functor
\begin{align*}
\mathfrak{M}:(p\text{-}\mathbf{Gr}/\OO_K)&\to\mathbf{Mod}^{\phi,1}_{/\mathfrak{S}}\\
\mathcal{G}&\mapsto \mathfrak{M}(\mathcal{G}):=\mathrm{coker}(\mathfrak{M}(\mathcal{G}_1)\to \mathfrak{M}(\mathcal{G}_2))\end{align*}
which is an equivalence when $p>2$. When $p = 2$, $\mathfrak{M}$ gives an equivalence 
on the full subcategories of connected objects.
\end{thm}

To study finite flat group schemes of $p^n$-torsion for a 
fixed positive integer $n$ we introduce further subcategories 
of $\mathbf{Mod}_{/\mathfrak{S}}^\phi$ and modify $\mathfrak{M}$ in 
the following way. Recall that 
$\mathfrak{S} = W[[u]]$. Then $\mathfrak{S}$ projects onto the rings 
$\mathfrak{S}_n = W_n[[u]]$ for each integer $n>0$, where $W_n$ 
is the $n$th truncation of the Witt ring $W$. 
We note that $\mathfrak{S}_n$ has a Frobenius extending the Frobenius 
on $W_n$ by $u\mapsto u^p$.

\begin{defn}
Let $\mathbf{Mod}^{\phi,r}_{/\mathfrak{S}_n}$ 
be the subcategory of $\mathbf{Mod}^{\phi,r}_{/\mathfrak{S}}$ 
consisting of $p^n$-torsion objects for a fixed positive integer $n$. These modules 
are defined over the ring $\mathfrak{S}_n$ and are called 
\textbf{finite $\mathfrak{S}_n$-modules of height $r$}.
\end{defn}

For each $n$ we construct the functor 
\begin{align*}
-\otimes\mathfrak{S}_n:\mathbf{Mod}^{\phi,1}_{/\mathfrak{S}}&\to \mathbf{Mod}^{\phi,1}_{/\mathfrak{S}_n}\\
\mathfrak{M}&\mapsto \mathfrak{M}\otimes_{\mathfrak{S}}\mathfrak{S}_n=:\mathfrak{M}_{\mathfrak{S}_n}.\end{align*}
It is clear this commutes with $\phi$.

\begin{prop}
\label{fingroup}
The functor $-\otimes\mathfrak{S}_n$ is full, and faithful on 
$p^n$-torsion objects in $\mathbf{Mod}^{\phi,1}_{/\mathfrak{S}}$. In particular, 
there is an equivalence between finite flat $p^n$-torsion 
group schemes and the category $\mathbf{Mod}^{\phi,1}_{/\mathfrak{S}_n}$.
\end{prop}
\begin{proof}
Each part of this statement is easily checked by the compatibility of 
$-\otimes \mathfrak{S}_n$ with the Frobenius morphisms. The equivalence 
then follows from Theorem \ref{GroupEquiv}.
\end{proof}

\begin{rem}
Every object in $\mathbf{Mod}_{/\mathfrak{S}_1}^{\phi,r}$ is 
$\mathfrak{S}_1$-free. This, however, is not true of objects 
in $\mathbf{Mod}_{/\mathfrak{S}_n}^{\phi,r}$ when $n>1$.
\end{rem}

\begin{exmps}
Let $\mathsf{Eis}_n(u)\in W_n[u]$ be the projection of 
$\mathsf{Eis}(u)\in W[u]$.
\begin{enumerate}
\item Let $\mathcal{G}_n = p^{-n}\mathbb{Z}/\mathbb{Z}$. Then 
$$\mathfrak{M}(\mathcal{G}_n)\simeq \mathfrak{S}_ne,\quad \phi_{\mathfrak{M}(\mathcal{G}_n)}(e) = \frac{1}{c}\mathsf{Eis}_n(u)e,\quad \psi_{\mathfrak{M}(\mathcal{G}_n)}(e) = c\otimes e$$ 
where $e$ denotes a generator of the $\mathfrak{S}_n$-module 
$\mathfrak{M}(\mathcal{G}_n)$. 
We denote this Kisin module by $\mathfrak{S}_n(1)$.
\item Let $\mathcal{G}_n = \mu_{p^n}$. Then 
$$\mathfrak{M}(\mathcal{G}_n)\simeq \mathfrak{S}_ne,\quad \phi_{\mathfrak{M}(\mathcal{G}_n)}(e) = \frac{1}{c}e,\quad \psi_{\mathfrak{M}(\mathcal{G}_n)}(e) = c\otimes \mathsf{Eis}_n(u)e$$ 
where $e$ denotes a generator of the $\mathfrak{S}_n$-module 
$\mathfrak{M}(\mathcal{G}_n)$. 
We denote this Kisin module by $\mathfrak{S}_n$.
\end{enumerate}

We note that in these examples $\mathcal{G} = \varinjlim_n\mathcal{G}_n$ is a 
$p$-divisible group, and 
$\mathfrak{M}(\mathcal{G}) = \varprojlim_n\mathfrak{M}(\mathcal{G}_n)$ since 
the $\mathfrak{M}(\mathcal{G}_n)$ have compatible Frobenius structures.
\end{exmps}

\begin{rem}
In light of Proposition \ref{fingroup}, we will write $\mathfrak{M}_{\mathfrak{S}_n}(\mathcal{G}):= \mathfrak{M}(\mathcal{G})\otimes_{\mathfrak{S}}\mathfrak{S}_n$. Then 
$\mathfrak{M}_{\mathfrak{S}_n}$ is a functor on any Barsotti-Tate group 
and its image is the Kisin module corresponding to its maximal $p^n$-torsion 
subgroup scheme.
\end{rem}

\subsection{Deformation Theory with Kisin Modules}
\label{hodgeBK}
We interpret the deformation theory of $p$-divisible groups in 
the language of Kisin modules. 
Recall that $S$ is the $p$-adic completion of $W[u,\mathsf{Eis}(u)^n/n!]_{n\ge 1}$ 
with a Frobenius extending the Frobenius on $W$ by $u\mapsto u^p$ and 
a filtration given by $\mathrm{Fil}^1(S) = \mathrm{ker}(S\xrightarrow{u\mapsto\pi}\OO_K)$. 
We also recall that $\mathfrak{S} = W[[u]]$ with a Frobenius extending the 
Frobenius on $W$ by $u\mapsto u^p$. The categories $\mathbf{BT}_{/S}^{\phi}$ and $\mathbf{BT}_{/\mathfrak{S}}^{\phi,1}$ denoted Breuil modules and Kisin modules of height $1$, respectively, and were defined in Sections \ref{BrSEC} and \ref{KisSEC}.

The ring $S$ has a natural structure as an $\mathfrak{S}$-algebra 
under $u\mapsto u$, which induces the functor
\begin{align*}
\mathcal{M}:\mathbf{BT}^{\phi,1}_{/\mathfrak{S}}&\to \mathbf{BT}^{\phi}_{/S}\\
\mathfrak{M}&\mapsto \mathcal{M}(\mathfrak{M}):=S\otimes_{\mathfrak{S}}\phi^*\mathfrak{M}
\end{align*}
where the $S$-module filtration 
$\mathrm{Fil}^1\mathcal{M}(\mathfrak{M})\subseteq \mathcal{M}(\mathfrak{M})$ 
is defined by the Cartesian diagram

\bigskip
\centerline{
\begin{xy}
(0,15)*+{\mathrm{Fil}^1\mathcal{M}(\mathfrak{M})}="a";
(35,15)*+{\mathcal{M}(\mathfrak{M}) = S\otimes_{\mathfrak{S}}\phi^*\mathfrak{M}}="b";
( 0,0)*+{\mathrm{Fil}^1S\otimes_{\mathfrak{S}}\mathfrak{M}}="c";
(35,0)*+{S\otimes_{\mathfrak{S}}\mathfrak{M}}="d";
{\ar@{^{(}->} "a";"b"};{\ar^{1\otimes\phi_{\mathfrak{M}}} "b";"d"};{\ar@{^{(}->} "c";"d"};{\ar "a";"c"};
\end{xy}.}
\bigskip

\noindent This filtered submodule is also written as
$$\mathrm{Fil}^1\mathcal{M}(\mathfrak{M}) = S\cdot \mathrm{Fil}^1\phi^*(\mathfrak{M}) + \mathrm{Fil}^1 S\cdot \mathcal{M}(\mathfrak{M})$$
where $\mathrm{Fil}^1\phi^*(\mathfrak{M})$ is defined by the Cartesian diagram

\bigskip
\centerline{
\begin{xy}
(0,15)*+{\mathrm{Fil}^1\phi^*(\mathfrak{M})}="a";
(35,15)*+{\phi^*\mathfrak{M} = \mathfrak{S}\otimes_{\phi,\mathfrak{S}}\mathfrak{M}}="b";
( 0,0)*+{\mathsf{Eis}(u)\cdot \mathfrak{M}}="c";
(35,0)*+{\mathfrak{M}}="d";
{\ar@{^{(}->} "a";"b"};{\ar^{1\otimes\phi_{\mathfrak{M}}} "b";"d"};{\ar@{^{(}->} "c";"d"};{\ar "a";"c"};
\end{xy}.}
\bigskip

\begin{thm}[\cite{Kisin}]
\label{GroupExt}
Let $\mathcal{G}/\OO_K$ be a $p$-divisible group. Then there 
exists an isomorphism
$$\mathbb{D}(\mathcal{G})(S)\xrightarrow{\sim}\mathcal{M}(\mathfrak{M}(\mathcal{G}))$$
which is compatible with $\phi$.
\end{thm}

\begin{cor}
There is a canonical $\phi$-equivariant isomorphism 
$$\phi^*(\mathfrak{M}(\mathcal{G})/u\mathfrak{M}(\mathcal{G}))\xrightarrow{\sim}\mathbb{D}(\mathcal{G}_0)(W)$$
where $\mathcal{G}_0 = \mathcal{G}\otimes_{\OO_K} k$. In particular, the linear Frobenius and 
Verschiebung endomorphisms on the Dieudonn\'e module are given by $1\otimes\phi_{\mathfrak{M}(\mathcal{G})}\mod u$ 
and $1\otimes\psi_{\mathfrak{M}(\mathcal{G})}\mod u$.
\end{cor}
\begin{proof}
The ring $W$ has the structure of an $S$-algebra by mapping $u\mapsto 0$. 
We compute
\begin{align*}
\mathbb{D}(\mathcal{G}_0)(W)\xrightarrow{\sim}\mathbb{D}(\mathcal{G})(S)\otimes_S W&\xrightarrow{\sim}\mathcal{M}(\mathfrak{M}(\mathcal{G}))\otimes_S W\\
&=\phi^*(\mathfrak{M}(\mathcal{G}))\otimes_{\mathfrak{S}} W\\
&=\phi^*(\mathfrak{M}(\mathcal{G}))/u\phi^*(\mathfrak{M}(\mathcal{G}))
\end{align*}
where the first isomorphism follows from the crystalline base change 
property of the Frobenius 
crystal in Lemma \ref{properties} (1) and the second isomorphism follows from Theorem \ref{GroupExt}. 
Since each step in the computation commutes with the Frobenius operator, 
we conclude that the Frobenius and Verschiebung operators on the Dieudonn\'e 
module are reductions of (the linearization of) the original modulo $u$.
\end{proof}

\begin{lem}
\label{diffmod}
\label{Liemod}
Let $\mathcal{G}/\OO_K$ be a $p$-divisible group. There are natural isomorphisms
\begin{align*}
\omega_{\mathcal{G}}&\simeq \mathfrak{M}(\mathcal{G})/\phi_{\mathfrak{M}(\mathcal{G})}(\mathfrak{M}(\mathcal{G}))\\
\mathrm{Lie}(\mathcal{G})&\simeq \phi^*\mathfrak{M}(\mathcal{G})^*/\mathrm{Fil}^1\phi^*\mathfrak{M}(\mathcal{G})^*.\end{align*}
\end{lem}
\begin{proof}
Let $\mathfrak{M} = \mathfrak{M}(\mathcal{G})$. By Theorem \ref{GroupExt} 
and Corollary \ref{deformation}, we have the exact diagram

\bigskip
\centerline{
\begin{xy}
(10,15)*+{0}="a";
(40,15)*+{\omega_{\mathcal{G}}}="b";
(87,15)*+{\mathbb{D}(\mathcal{G})(\OO_K)}="c";
(125,15)*+{\mathrm{Lie}(\mathcal{G}^*)}="d";
(147,15)*+{0}="e";
(10,0)*+{0}="f";
(40,0)*+{\mathrm{Fil}^1\mathcal{M}(\mathfrak{M})/(\mathrm{Fil}^1S)\mathrm{Fil}^1\mathcal{M}(\mathfrak{M})}="g";
(87,0)*+{\mathcal{M}(\mathfrak{M})/(\mathrm{Fil}^1S)\mathcal{M}(\mathfrak{M})}="h";
(125,0)*+{\mathcal{M}(\mathfrak{M})/\mathrm{Fil}^1\mathcal{M}(\mathfrak{M})}="i";
(147,0)*+{0}="j";
{\ar "a";"b"};{\ar "b";"c"};{\ar "c";"d"};{\ar "d";"e"};
{\ar "f";"g"};{\ar "g";"h"};{\ar "h";"i"};{\ar "i";"j"};
{\ar_{\simeq} "b";"g"};{\ar_{\simeq} "c";"h"};{\ar_{\simeq} "d";"i"};
\end{xy}.}
\bigskip

The first vertical arrow gives the isomorphisms
\begin{align*}
\omega_{\mathcal{G}}&\simeq \mathrm{Fil}^1\mathcal{M}(\mathfrak{M}(\mathcal{G}))/(\mathrm{Fil}^1S)\mathrm{Fil}^1\mathcal{M}(\mathfrak{M}(\mathcal{G}))\\
&=\left[S\cdot \mathrm{Fil}^1\phi^*(\mathfrak{M}) + \mathrm{Fil}^1 S\cdot \mathcal{M}(\mathfrak{M})\right]/(\mathrm{Fil}^1S)\left[S\cdot \mathrm{Fil}^1\phi^*(\mathfrak{M}) + \mathrm{Fil}^1 S\cdot \mathcal{M}(\mathfrak{M})\right]\\
&\simeq \mathfrak{M}(\mathcal{G})/\phi_{\mathfrak{M}(\mathcal{G})}(\mathfrak{M}(\mathcal{G})).
\end{align*}
Replacing the objects in the diagram with their dual objects, the last vertical arrow 
gives the isomorphisms
\begin{align*}
\mathrm{Lie}(\mathcal{G})&\simeq \mathcal{M}(\mathfrak{M}(\mathcal{G}^*))/\mathrm{Fil}^1\mathcal{M}(\mathfrak{M}(\mathcal{G}^*))\\
&= \left[S\otimes \phi^*\mathfrak{M}(\mathcal{G}^*)\right]/\left[S\cdot \mathrm{Fil}^1\phi^*(\mathfrak{M}(\mathcal{G}^*)) + \mathrm{Fil}^1S\cdot (S\otimes \phi^*\mathfrak{M}(\mathcal{G}^*))\right]\\
&=\phi^*\mathfrak{M}(\mathcal{G}^*)/\mathrm{Fil}^1\phi^*\mathfrak{M}(\mathcal{G}^*).
\end{align*}
The isomorphism $\mathfrak{M}(\mathcal{G}^*)\simeq \mathfrak{M}(\mathcal{G})^*$ from Corollary \ref{thecor} now gives the statement.
\end{proof}

Recall now that $\mathbf{Mod}^{\phi,1}_{/\mathfrak{S}_n}$ is equivalent to the category of 
$p^n$-torsion finite flat group schemes. For each $n$ we can also define the 
category $\mathbf{Mod}_{/S_n}^{\phi}$ in the natural way, where $S_n$ 
is the level-$n$ quotient of $S$ constructed in the obvious way on the 
$n$th truncation of $W$. Then we can construct a functor
\begin{align*}
-\otimes S_n: \mathbf{Mod}^{\phi}_{/S}&\to \mathbf{Mod}^{\phi}_{/S_n}\\
\mathcal{M}&\mapsto \mathcal{M}\otimes_{S}S_n.\end{align*}
It is clear this commutes with $\phi$ and with filtrations.

\begin{prop}
\label{blah}
The functor $\mathcal{M}$ induces a map
$$\mathcal{M}:\mathbf{Mod}^{\phi,1}_{/\mathfrak{S}_n}\to \mathbf{Mod}^{\phi}_{/S_n}.$$ 
When $n=1$, this is an equivalence when $p>2$ and an equivalence on connected objects 
when $p=2$.
\end{prop}
\begin{proof}
The functor $\mathcal{M}$ commutes with 
$-\otimes\mathfrak{S}_n$ and $-\otimes S_n$. We deduce the 
equivalence from Theorem \ref{GroupExt}. 
\end{proof}

\begin{prop}
\label{deg}
Suppose $k = \mathbb{F}_q$, $K_0 = W(k)[1/p]$, and $K/K_0$ is a finite, totally ramified 
extension. For any $p$-torsion finite flat group scheme $\mathcal{G}/\OO_K$, 
$$\#\omega_{\mathcal{G}} = q^{v_u\left(\det \phi_{\mathfrak{M}_{\mathfrak{S}_1}(\mathcal{G})}\right)}.$$
\end{prop}
\begin{proof}
By Proposition \ref{blah}, the isomorphisms of 
Lemma \ref{diffmod} extend to the 
category $\mathbf{Mod}_{/\mathfrak{S}_1}^{\phi,1}$, so
$$\omega_\mathcal{G}\simeq \mathfrak{M}_{\mathfrak{S}_1}(\mathcal{G})/\phi_{\mathfrak{M}_{\mathfrak{S}_1}(\mathcal{G})}\mathfrak{M}_{\mathfrak{S}_1}(\mathcal{G}).$$
Now choose a basis $e_1,\ldots,e_h$ of 
$\mathfrak{M}_{\mathfrak{S}_1}(\mathcal{G})$ and let 
$u^{s_1},\ldots,u^{s_h}$ denote the elementary divisors of the endomorphism 
$\phi_{\mathfrak{M}_{\mathfrak{S}_1}(\mathcal{G})}$. 
Then the degree of the characteristic ideal of $\omega_{\mathcal{G}}$ 
is $\sum_{i=1}^h s_i = v_u(\mathrm{det}\phi_{\mathfrak{M}_{\mathfrak{S}_1}(\mathcal{G})})$. The statement now follows as a comparison of 
modules over $\mathbb{F}_q$.
\end{proof}

\begin{rem}
Recall that the category of Kisin modules is a tensor category 
and that $\mathfrak{M}$ commutes with symmetric and 
exterior powers. Therefore, we may equivalently state 
Proposition \ref{deg} in the following way: 
let $\mathfrak{L}:=\bigwedge^h \mathfrak{M}(\mathcal{G})$ 
where $h$ denotes the height of the $p$-torsion finite flat group scheme 
$\mathcal{G}/\OO_K$, then
$$\#\omega_{\mathcal{G}} = q^{v_u(\phi_{\mathfrak{L}})}.$$
\end{rem}

\section{Finite Group Theory with Kisin Modules}
\subsection{Classifying Finite Flat Subgroup Schemes}
\label{classification}
We restrict our attention in this section to the category 
$\mathbf{Mod}_{/\mathfrak{S}_n}^{\phi,r}$ introduced in Section \ref{KisSEC}. 
We recall this is the category of $p^n$-torsion Kisin modules for a 
fixed positive integer $n$. 
When $r=1$ the category is equivalent to the category of 
$p^n$-torsion finite flat group schemes by Proposition \ref{fingroup}. 

\begin{defn}
Let $\mathfrak{M}\in \mathbf{Mod}_{/\mathfrak{S}_n}^{\phi,r}$. A submodule $\mathfrak{N}\subseteq\mathfrak{M}$ 
is \textbf{saturated in $\mathfrak{M}$} (or just \textbf{saturated}) if 
$\mathfrak{N} = (\mathfrak{N}\otimes_\mathfrak{S}\mathfrak{S}[1/u])\cap\mathfrak{M}$.
\end{defn}

\begin{lem}[Proposition 2.3.2, \cite{Liutors}]
\label{projdimone}
Let $\mathfrak{M}\in \textbf{Mod}_{/\mathfrak{S}_n}^{\phi,r}$ and 
$\mathfrak{N}\subseteq \mathfrak{M}$ be a $\phi$-stable submodule. 
Then the following are equivalent:
\begin{enumerate}
\item $\mathfrak{N}$ is saturated in $\mathfrak{M}$,
\item $\mathfrak{N}$ and $\mathfrak{M}/\mathfrak{N}$ are both objects in 
$\textbf{Mod}_{/\mathfrak{S}_n}^{\phi,r}$, i.e., both have projective 
dimension $1$,
\item $\mathrm{GF}(\mathfrak{N}) = \mathfrak{N}\otimes_{\mathfrak{S}}\OO_{\mathcal{E}}$ is a 
subobject of $\mathrm{GF}(\mathfrak{M}) = \mathfrak{M}\otimes_{\mathfrak{S}}\OO_{\mathcal{E}}$.
\end{enumerate}
\end{lem}

The equivalence between (1), (2), and (3) 
relies on the following lemma.

\begin{lem}[Proposition 2.3.2, \cite{Liutors}]
\label{utorsfree}
Any object $\mathfrak{M}\in \mathbf{Mod}_{/\mathfrak{S}_n}^{\phi,r}$ is $u$-torsion free.
\end{lem}
\begin{proof}
Each object $\mathfrak{M}\in \mathbf{Mod}_{/\mathfrak{S}}^{\phi,r}$ 
has projective dimension $1$, hence has a resolution
\begin{equation}
\label{exact}
0\to\mathfrak{M}_2\xrightarrow{f}\mathfrak{M}_1\to \mathfrak{M}\to 0
\end{equation}
by finite, free $\mathfrak{S}$-modules $\mathfrak{M}_1$ and $\mathfrak{M}_2$.
Let $\{e_1,\ldots, e_d\}$, resp. $\{f_1,\ldots,f_d\}$, be a choice of 
$\mathfrak{S}$-basis for $\mathfrak{M}_1$, resp. $\mathfrak{M}_2$, such that 
$(e_1,\ldots,e_d) = (f_1,\ldots,f_d)A$, where $A$ is the transition matrix 
corresponding to $f$ in (\ref{exact}). Since $\mathfrak{M}$ is killed 
by $p^n$ for some positive integer $n$, there exists a matrix $B$ with $AB = p^nI$.

Assume for the sake of contradiction that there exists a non-zero element 
$\overline{a}\in \mathfrak{M}$ 
which is killed by $u^k$ for some $k\ge 1$. 
Let $a = \sum_{i=1}^d a_ie_i$ 
be a lift of $\overline{a}$ to $\mathfrak{M}_1$. Then we have 
$$u^k(a_1,\ldots,a_d) = (b_1,\ldots,b_d)A$$
for some $b_i\in \mathfrak{S}$. Since $AB = p^nI$, we further have
$$(b_1,\ldots,b_d) = u^k(p^{-n})(a_1,\ldots,a_d)B.$$
Let $(c_1,\ldots,c_d) = p^{-n}(a_1,\ldots,a_d)B$. Then $c_i\in \mathfrak{S}$ and 
$(a_1,\ldots,a_d) = (c_1,\ldots,c_d)A$. But by the exactness of (\ref{exact}), 
this implies $\overline{a} = 0$, a contradiction.

To conclude the statement for the category 
$\mathbf{Mod}_{/\mathfrak{S}_n}^{\phi,r}$, we apply the functor 
$-\otimes \mathfrak{S}_n$ from Proposition \ref{fingroup}. Then an object in $\mathbf{Mod}_{/\mathfrak{S}_n}^{\phi,r}$ is 
$u$-torsion free if its lift in $\mathbf{Mod}_{/\mathfrak{S}}^{\phi,r}$ is 
$u$-torsion free.
\end{proof}

\begin{proof}[Proof of Lemma \ref{projdimone}]
By Lemma \ref{utorsfree}, objects of $\mathbf{Mod}_{/\mathfrak{S}_n}^{\phi,r}$ are 
$u$-torsion free, and hence $\mathrm{GF}$ induces the bijection
\begin{align*}
\{\text{saturated submodules of }\mathfrak{M}\}&\xrightarrow{\sim}\{\text{subobjects of }\mathrm{GF}(\mathfrak{M})\}\\
\mathfrak{N}&\mapsto\mathfrak{N}\otimes_{\mathfrak{S}}\OO_{\mathcal{E}}\\
\mathcal{N}\cap\mathfrak{M}\otimes 1&\mapsfrom\mathcal{N}.\end{align*}
This shows the equivalence between (1) and (3).

For the equivalence between (1) and (2), by the proof of Lemma \ref{utorsfree}, 
an object $\mathfrak{M}$ in 
$\mathbf{Mod}_{/\mathfrak{S}}^{\phi,r}$ has projective dimension 
$1$ if and only if $u$ is regular for $\mathfrak{M}$. 
Therefore, any submodule $\mathfrak{N}\subseteq \mathfrak{M}$ 
necessarily has projective dimension $1$. To establish the statement, we must 
relate the condition that the quotient $\mathfrak{M}/\mathfrak{N}$ 
has projective dimension $1$ to the condition that $\mathfrak{N}$ 
is saturated in $\mathfrak{M}$. If $\mathfrak{N}$ is saturated in 
$\mathfrak{M}$, it is easy to see that $u$ is regular on 
$\mathfrak{M}/\mathfrak{N}$, and conversely, and so we 
apply Lemma \ref{utorsfree} again.
\end{proof}

Saturated submodules have a geometric interpretation on objects 
in $\textbf{Mod}_{\mathfrak{S}_n}^{\phi,1}$. In particular, they describe 
finite flat subgroup schemes under the equivalence in Theorem \ref{equivfgroup}. 
We prove this in the following proposition.

\begin{prop}
\label{saturation}
Let $\mathcal{G}/\OO_K$ be a finite flat group scheme of height $g$. 
There is a one-to-one correspondence between saturated 
submodules $\mathfrak{N}\subseteq \mathfrak{M}_{\mathfrak{S}_1}(\mathcal{G})$ and 
finite, flat subgroup schemes $\mathcal{H}\subseteq\mathcal{G}$; 
in particular, 
$$\mathfrak{M}_{\mathfrak{S}_1}(\mathcal{H}) = \mathfrak{M}_{\mathfrak{S}_1}(\mathcal{G})/\mathfrak{N}.$$ 
\end{prop}
\begin{proof}
Let $\mathcal{H}\subseteq \mathcal{G}$ be a finite, flat subgroup scheme. Then 
we necessarily have the short exact sequence of finite flat subgroup schemes
$$0\to \mathcal{H}\to \mathcal{G}\to \mathcal{G}/\mathcal{H}\to 0.$$
Applying the contravariant functor $\mathfrak{M}_{\mathfrak{S}_1}$, we obtain
$$0\to \mathfrak{M}_{\mathfrak{S}_1}(\mathcal{G}/\mathcal{H})\to \mathfrak{M}_{\mathfrak{S}_1}(\mathcal{G})\to \mathfrak{M}_{\mathfrak{S}_1}(\mathcal{H})\to 0.$$
Since $\mathfrak{M}_{\mathfrak{S}_1}$ is an exact functor and quotients exist in the category of 
Kisin modules, it follows that
$$\mathfrak{M}_{\mathfrak{S}_1}(\mathcal{H})\simeq \mathfrak{M}_{\mathfrak{S}_1}(\mathcal{G})/\mathfrak{M}_{\mathfrak{S}_1}(\mathcal{G}/\mathcal{H}).$$
By Lemma \ref{projdimone}, $\mathfrak{M}_{\mathfrak{S}_1}(\mathcal{G}/\mathcal{H})$ must 
be saturated in $\mathfrak{M}_{\mathfrak{S}_1}(\mathcal{G})$ for 
$\mathfrak{M}_{\mathfrak{S}_1}(\mathcal{H})$ to have projective dimension $1$. The converse 
follows from the equivalence in Theorem \ref{equivfgroup}.
\end{proof}

\subsection{Harder-Narasimhan Theory of Finite Kisin Modules}
\label{HNSect}
\label{submoduledecomp}
We now introduce the Harder-Narasimhan (HN) theory for 
the category $\mathbf{Mod}_{/\mathfrak{S}}^{\phi,r}$.  
HN theory was originally developed for finite flat group 
schemes by Fargues \cite{Fargues} and extended to Kisin 
modules more broadly by Levin and Wang-Erickson \cite{LevinWang}. 
The study of the 
Harder-Narasimhan filtration of finite flat group schemes, and more 
broadly finite Kisin modules, is intimately linked with 
understanding the variation of (semi-stable) Faltings height 
(see Remark \ref{fheighthn}).

We start with a warm-up lemma, which will be relevant to the 
computation of the HN slopes.

\begin{lem}
\label{dettrick}
Suppose $k = \mathbb{F}_q$, $K_0 = W(k)[1/p]$, and $K/K_0$ is a 
finite, totally ramified extension. 
Let $\mathcal{G}/\OO_K$ be a finite flat $p$-torsion group scheme of height $g$ and 
$\mathcal{H}\subseteq \mathcal{G}$ be a flat subgroup scheme of 
height $d$. Then there is a line $\mathfrak{L}\subseteq \bigwedge^{g-d}\mathfrak{M}(\mathcal{G})$ such that 
$$\#\omega_{\mathcal{H}} = q^{v_u\left(\mathrm{det}(\phi_{\mathfrak{M}(\mathcal{G})})\right) - v_u\left(\phi_{\mathfrak{L}}\right)}.$$
\end{lem}
\begin{proof}
As in Proposition \ref{saturation}, we have a short exact sequence of 
$\mathfrak{S}$-modules
$$0\to \mathfrak{M}_{\mathfrak{S}_1}(\mathcal{G}/\mathcal{H})\to \mathfrak{M}_{\mathfrak{S}_1}(\mathcal{G})\to \mathfrak{M}_{\mathfrak{S}_1}(\mathcal{H})\to 0$$
where the terms in the sequence have $\mathfrak{S}$-rank $g-d$, $g$, 
and $d$, respectively, corresponding to the height of the groups. Since 
$\mathbf{Mod}_{/\mathfrak{S}}^\phi$ is a tensor category, determinants are 
additive on the short exact sequence. Moreover, the $\mathfrak{S}$-linear 
determinant of $\mathfrak{M}(\mathcal{G}/\mathcal{H})$ will be a 
line $\mathfrak{L}\subseteq \bigwedge^{g-d}\mathfrak{M}(\mathcal{G})$. The statement now follows 
from Proposition \ref{deg}.
\end{proof}

We are now ready to introduce the Harder-Narasimhan filtration on 
$\textbf{Mod}_{\mathfrak{S}_n}^{\phi,r}$. 

\begin{defn}
\label{slope}
The \textbf{Harder-Narasimhan slope} (or \textbf{HN slope}) of an 
object $\mathfrak{M}\in \textbf{Mod}_{\mathfrak{S}_n}^{\phi,r}$ is
$$\mu(\mathfrak{M}):=\frac{\mathrm{deg}(\mathfrak{M})}{\mathrm{rk}(\mathfrak{M})}$$
where
\begin{enumerate}
\item $\mathrm{deg}(\mathfrak{M})$ is the \textbf{degree} of $\mathfrak{M}$, 
defined as
$$\mathrm{deg}(\mathfrak{M}):=\frac{1}{[K:\QQ_p]}\ell_{\ZZ/p}(\mathrm{coker}(\phi_{\mathfrak{M}})),$$
\item $\mathrm{rk}(\mathfrak{M})$ is the \textbf{rank} of $\mathfrak{M}$, 
defined as
$$\mathrm{rk}(\mathfrak{M}) = \ell_{\OO_\mathcal{E}/p}(\mathrm{GF}(\mathfrak{M})).$$
\end{enumerate}
\end{defn}

\begin{rem}
\label{fheighthn}
When $\mathfrak{M}\in \textbf{Mod}_{\mathfrak{S}_n}^{\phi,1}$, it can 
be shown that the degree and rank in Definition \ref{slope} 
equal to the degree and rank defined for 
finite flat group schemes in \cite{Fargues} for objects in 
$\mathbf{Mod}_{\mathfrak{S}_n}^{\phi,1}$. Thus, if we let 
$\mathcal{A}_1/\OO_K$ and $\mathcal{A}_2/\OO_K$ 
denote N\'eron models of abelian varieties with potential good 
reduction and $f:\mathcal{A}_1\to \mathcal{A}_2$ be an isogeny 
of $p$-power degree,  
then $\mathcal{G}/\OO_K = \mathrm{ker}(f)$ is a finite flat $p$-group 
scheme and $\mu(\mathfrak{M}(\mathcal{G}^*))$ is the ratio of 
the terms in $\log(\mathrm{deg}(f))$ and 
$\frac{1}{[K:\QQ]}\log(\#s^*\Omega^1_{\mathcal{G}/\OO_K})$ 
appearing in Lemma \ref{FaltingsIsog}. Thus, the Harder-Narasimhan 
theory of finite flat group schemes relates in a natural way 
to understanding their Faltings height, an observation 
that has already been made in \cite{Fargues}.
\end{rem}

\begin{defn}
An object $\mathfrak{M}\in \mathbf{Mod}_{/\mathfrak{S}_n}^{\phi,r}$ is 
\textbf{semistable} if for all saturated submodules 
$\mathfrak{N}\subseteq \mathfrak{M}$, their slopes satisfy the inequality 
$\mu(\mathfrak{N}) \ge \mu(\mathfrak{M})$. 
If an object is not semistable, then it is \textbf{unstable}.
\end{defn}

\begin{exmp}
All objects $\mathfrak{M}\in \mathrm{Mod}_{\mathfrak{S}_1}^{\phi,1}$ of 
rank $1$ are semistable.
\end{exmp}

\begin{lem}[Corollary 2.3.12, \cite{LevinWang}]
Let $\mathfrak{M}\in \mathbf{Mod}_{/\mathfrak{S}_n}^{\phi,r}$ be semistable. 
Then a saturated submodule (resp., quotient by a saturated submodule) 
$\mathfrak{N}\subseteq \mathfrak{M}$ (resp., $\mathfrak{M}\twoheadrightarrow \mathfrak{N}$) such that $\mu(\mathfrak{N}) = \mu(\mathfrak{M})$ is 
again semistable.
\end{lem}

\begin{thm}[Harder-Narasimhan Filtration, \cite{Fargues}, \cite{LevinWang}]
\label{HNthm}
Let $\mathfrak{M}\in \mathbf{Mod}_{\mathfrak{S}_n}^{\phi,r}$. Then 
there exists a unique filtration
$$0 = \mathfrak{M}_0\subseteq \mathfrak{M}_1\subseteq \cdots\subseteq \mathfrak{M}_k = \mathfrak{M}$$
by saturated subobjects such that for all $i$, $\mathfrak{M}_{i+1}/\mathfrak{M}_i$ 
is semi-stable and $\mu(\mathfrak{M}_{i+1}/\mathfrak{M}_i)>\mu(\mathfrak{M}_i/\mathfrak{M}_{i-1})$. Moreover, this filtration is preserved under endomorphisms.
\end{thm}

\begin{defn}
The \textbf{Harder-Narasimhan polygon} (or \textbf{HN polygon}) of an 
object $\mathfrak{M}\in \mathbf{Mod}_{/\mathfrak{S}}^{\phi,r}$ is the 
convex polygon whose segments have slope $\mu(\mathfrak{M}_{i+1}/\mathfrak{M}_i)$ 
and length $\mathrm{rk}(\mathfrak{M}_{i+1}/\mathfrak{M}_i)$, where 
the modules $\mathfrak{M}_i$ are those occuring in the Harder-Narasimhan filtration. 
We represent the HN-polygon as a piecewise linear function
$$\mathrm{HN}(\mathfrak{M}):[0,\mathrm{rk}(\mathfrak{M})]\to [0,\mathrm{deg}(\mathfrak{M})]$$
such that $\mathrm{HN}(\mathfrak{M})(0) = 0$ and $\mathrm{HN}(\mathfrak{M})(\mathrm{rk}(\mathfrak{M})) = \mathrm{deg}(\mathfrak{M})$.
\end{defn}

\begin{exmp}
The Harder-Narasimhan polygon of a semistable finite Kisin module 
is a single line of slope $\mu(\mathfrak{M})$.
\end{exmp}

\begin{rem}
When $\mathfrak{M}\in \mathbf{BT}_{\mathfrak{S}}^{\phi,r}$ is 
\emph{flat} over $\mathfrak{S}$, we can extend the Harder Narasimhan 
theory on its projections $\mathfrak{M}\otimes_{\mathfrak{S}}\mathfrak{S}_n$ 
to $\mathfrak{M}$ as follows. We define
$$\mathrm{deg}_{\mathfrak{S}}(\mathfrak{M}) := \mathrm{deg}(\mathfrak{M}/p^n)/n$$
$$\mu_{\mathfrak{S}}(\mathfrak{M}) := \mu(\mathfrak{M}/p^n).$$
That this is independent of $n$ is the content of Proposition 6.3.5 in \cite{LevinWang}. 
Alternatively, if we let $\mathrm{HN}(\mathfrak{M}/p^n)$ denote the HN 
polygon of the object $\mathfrak{M}/p^n\in \mathbf{Mod}_{/\mathfrak{S}}^{\phi,r}$, 
the HN polygon (and hence the HN filtration) 
of $\mathfrak{M}$ is obtained as the limit of the maps
$$x\mapsto \frac{1}{n}\mathrm{HN}(\mathfrak{M}/p^n)(nx)$$
which converges uniformly increasingly as $n\to \infty$ (Theorem 6.2.3, \cite{LevinWang}).
It is shown in \cite{LevinWang} that these two definitions are equivalent 
and moreover the isogeneous flat Kisin modules have the same HN 
filtration. We leave it to the reader to work check for herself the details or consult 
\cite{LevinWang} for reference.
\end{rem}

\section{CM Kisin Modules}

\subsection{$\OO_E$-Linear CM Kisin Modules}
\label{OProjConstr}
Fix $E$ to be a ($p$-adic) CM field such that $[E:\QQ_p]=h$ 
and let $E^*$ be its ($p$-adic) reflex field. 
Denote by $\OO_E\subseteq E$ 
the ring of integral elements. Let $K/\QQ_p$ be a finite field extension with 
ring of integers $\OO_K$, a choice of uniformizer 
$\pi_K$, and residue field $k = \OO_K/\pi_K\OO_K$, and 
assume that $K\supseteq E^*$. Let $W = W(k)$ and 
$K_0 = W[1/p]\subseteq K$ be the maximal unramified 
subextension of $K$ over $\QQ_p$. Then the characteristic 
polynomial of the extension $K/K_0$ is an 
Eisenstein polynomial $\mathsf{Eis}(u)\in W[u]$, i.e., 
$\mathsf{Eis}(u) = u^e + a_{e-1} + \cdots + a_1 u + a_0$, where $e$ is the ramification 
index of $K/K_0$, $p|a_i$ for all $i$, and $a_0 = cp$ for some $c\in W(k)^\times$.

Now let $\mathcal{G}/\OO_K$ be a CM $p$-divisible group 
with ($p$-adic) CM by $(\OO_E,\Phi)$, and 
we assume that $\mathcal{G}$ has an $\OO_E$-linear structure, 
i.e., $\mathrm{Lie}(\mathcal{G})\simeq \OO_K\otimes \OO_E$. 
We note that $\mathcal{G}$ has 
height $h$ and denote its dimension by $d$.

\begin{defn}
\label{OEDef}
An \textbf{$\OO_E$-linear CM Kisin Module} of type $(\OO_E,\Phi)$ is an object 
of $\mathbf{BT}_{/\mathfrak{S}}^{\phi,1}$ 
isomorphic to $\mathfrak{M}(\mathcal{G})$ for a CM $p$-divisible group $\mathcal{G}$ 
with ($p$-adic) CM by $(\OO_E,\Phi)$ and 
an $\OO_E$-linear structure.
\end{defn}

Recall the following definition of the reflex norm by evaluating 
on $\QQ$-points the map from Definition \ref{reflexcochar}:
\begin{align*}
N_{\Phi,K}:K^\times&\to(K\otimes E)^\times\to E^\times\\
x&\mapsto\mathrm{det}_E(x|_{V_{\Phi,K}}).\end{align*}
Here, $V_{\Phi,K}$ is a finitely-generated $K\otimes_{\QQ_p}E$-module 
and $x\in K$ induces a $K\otimes_{\QQ_p}E$-endomorphism on 
$V_{\Phi,K}$. The subscript $E$ indicates we take 
the $E$-linear determinant of the endomorphism induced by 
$x$ on $V_{\Phi,K}$ considered as an $E$-vector space.

For any $p$-adic local subfield $L\subseteq K$ of finite index, 
we can define a \emph{relative reflex norm map} to be
\begin{align*}
N_{\Phi,K,L\otimes E}:K^\times&\to(K\otimes E)^\times\to (L\otimes E)^\times = \prod_i E_i^\times\\
x&\mapsto\mathrm{det}_{L\otimes E}(x|_{V_{\Phi,K,L\otimes E}}) := \prod_i\mathrm{det}_{E_i}(x|_{V_{\Phi,K,E_i}}),\end{align*}
where $V_{\Phi,K,L\otimes E} \simeq V_{\Phi,K}$ as an $L\otimes E$-module, 
$V_{\Phi,K,E_i}$ is a finitely-generated $K\otimes E_i$-module, and 
$V_{\Phi,K,L\otimes E} \simeq \bigoplus_i V_{\Phi,K,E_i}$ is induced from 
the decomposition of $L\otimes E \simeq \prod_i E_i$ into a product of fields. 
Note that the $E_i$ do not all have to be isomorphic fields, so as $\QQ_p$-vector 
spaces the $V_{\Phi,K,E_i}$ could all have different dimension. 
This determinant is well-defined up to permutation of the fields 
$E_i$. By the following lemma, it also has the more intrinsic definition.

\begin{lem}
Let $A$ be a ring and $M$ a finite projective $A$-module with 
an endomorphism $x$. Let $F$ be a finite free $A$-module such 
that $F\simeq M\oplus M'$ and $I_x$ the extension by 
the identity of $x$ to $F$. Then setting 
$\mathrm{det}_A(x|_M) = \mathrm{det}_A(I_x|_F)$ is 
independent of the choice of $F$.
\end{lem}
\begin{proof}
Let $F_1$ and $F_2$ be two choices of free $A$-modules with 
respective decompositions $M \oplus M_1$ and $M\oplus M_2$ 
into projective $A$-modules. 
Then for $i\in \{1,2\}$ let $x$ extend to endomorphisms 
$I_{x,i}$ on $F_i$ which acts on $M$ by $x$ and on $M_i$ by $1$. 
Construct the module $F_1\oplus F_2 = M\oplus M_1 \oplus M\oplus M_2$, 
and consider the automorphism $u$ which switches the two factors 
of $M$ and behaves as identity on the other two factors. This 
automorphism satisfies
$$u(I_{x,1}\oplus \mathrm{id})u^{-1} = \mathrm{id}\oplus I_{x,2}.$$
Taking determinants on each side of this equality shows that 
$\mathrm{det}_A(I_{x,1}|_{F_1}) = \mathrm{det}_A(I_{x,2}|_{F_2})$.
\end{proof}

Apply the lemma with $A = L\otimes E$ and 
$M = V_{\Phi,K,L\otimes E}$. There is a natural 
choice of a free $L\otimes E$-module $F = K\otimes E$ 
which has $L\otimes E$-rank $[K:L]$ and splits as
$$K\otimes E\simeq V_{\Phi,K,L\otimes E}\oplus V_{\Phi^c,K,L\otimes E}.$$ 
For $x\in K^\times$, 
define $I_x$ as in the proof of the lemma to be the endomorphism of $F$ 
which acts on $V_{\Phi,K,L\otimes E}$ by $x$ 
and on $V_{\Phi^c,K,L\otimes E}$ by $1$. Then by the 
lemma 
$N_{\Phi,K,L\otimes E}(x) = \mathrm{det}_{L\otimes E}(I_x|_{K\otimes E})$, 
and is easily seen to coincide with the definition above, 
up to a permutation of the $E_i$-factors. 
Note that $N_{\Phi^c,K,L\otimes E}(x) = \mathrm{det}_{L\otimes E}(I_x^c|_{K\otimes E})$ 
where $I_x^c$ acts on $V_{\Phi,K,L\otimes E}$ by $1$ 
and on $V_{\Phi^c,K,L\otimes E}$ by $x$. By multiplicativity 
of the determinant, this moreover induces the decomposition
\begin{equation}
\label{multiplicativity}
\mathrm{Nm}_{K/L}(x) =  N_{\Phi,K,L \otimes E}(x) N_{\Phi^c,K,L \otimes E}(x).
\end{equation}
When $L = \QQ_p$, the definition easily 
recovers the classic definition of the reflex norm.

For each finite extension $K/E^*$, the relative reflex norm extends to 
a multiplicative map on polynomial rings after a base extension by 
$\QQ_p[u]$ defined by
\begin{align*}
R_{\Phi,K,L\otimes E}:K[u]&\to (L\otimes E)[u]\\
f(u)&\mapsto N_{\Phi,K,L\otimes E}(f(u)).\end{align*}
The multiplicativity of this map comes from the multiplicativity of the 
relative reflex norm. 
We let $P_{\Phi,x,L\otimes E}(u)$ denote the image of the polynomial $u-x$ 
for $x\in K$ under this map.

\begin{prop}
\label{CMMods}
Let $\mathfrak{M}(\mathcal{G})$ be an $\OO_E$-linear CM Kisin module of type $(\OO_E,\Phi)$. Then $\mathfrak{M}(\mathcal{G})$ is 
a rank $1$ projective $\mathfrak{S}\otimes\mathcal{O}_E$-module with 
generator $e$ with Frobenius and Verschiebung endomorphisms given by
$$\phi_{\mathfrak{M}(\mathcal{G})}(e) = \frac{1}{c}P_{\Phi^c,\pi_K,K_0\otimes E}(u)e,$$
$$\psi_{\mathfrak{M}(\mathcal{G})}(e) = \phi^*\left(cP_{\Phi,\pi_K,K_0\otimes E}(u)\right)e.$$
In particular, 
$$\mathsf{Eis}(u) = P_{\Phi,\pi_K,K_0\otimes E}\cdot P_{\Phi^c,\pi_K,K_0\otimes E}(u).$$
\end{prop}
\begin{proof}
Let $\mathfrak{M}$ be a rank $1$ projective $\mathfrak{S}\otimes\OO_E$-module 
with Frobenius and Verschiebung endomorphisms $\phi_{\mathfrak{M}}$ and 
$\psi_{\mathfrak{M}}$ defined as in the statement. By Theorem \ref{GroupEquiv}, 
$\mathfrak{M} \simeq \mathfrak{M}(\mathcal{G}')$ for some 
$p$-divisible group $\mathcal{G}'$. By Lemma \ref{Liemod}, we 
compute $\mathrm{Lie}(\mathcal{G}')$ to be
$$\phi^*\mathfrak{M}(\mathcal{G}'^*)/\mathrm{Fil}^1\phi^*\mathfrak{M}(\mathcal{G}'^*)\simeq\phi^*\mathfrak{M}^*/\mathrm{Fil}^1\phi^*\mathfrak{M}^*\simeq \mathfrak{S}\otimes_{\mathbb{Z}_p}\OO_E/(P_{\Phi,\pi_K,K_0\otimes E}(u)).$$
This induces the $E$-linear isomorphism
$$\mathrm{Lie}(\mathcal{G}')\otimes \overline{\QQ}_p\simeq \left[\mathfrak{S}\otimes_{\mathbb{Z}_p} \mathcal{O}_E / (P_{\Phi,\pi_K,K_0\otimes E}(u))\right]\otimes\overline{\QQ}_p\simeq \prod_{i\in \Phi}(\overline{\mathbb{Q}}_p)_i$$
and hence $\mathcal{G}'$ is an $\OO_E$-linear CM $p$-divisible group with CM by 
$(\OO_E,\Phi)$. By Corollary \ref{cristallinerep}, after 
passing to an appropriate unramified extension, any two $\OO_E$-linear CM Kisin 
modules of the same type are isomorphic, so $\mathcal{G}\simeq \mathcal{G}'$ 
and $\mathfrak{M}(\mathcal{G})\simeq \mathfrak{M}$.
The decomposition of $\mathsf{Eis}(u)$ is induced by the 
decomposition formula (\ref{multiplicativity}).
\end{proof}

\begin{cor}
\label{decompCMMod}
Let $\mathfrak{M}(\mathcal{G})$ be an $\OO_E$-linear CM Kisin module of 
type $(\OO_E,\Phi)$. Suppose that $K$ contains all 
$\overline{\QQ_p}$-embeddings of $E$, and let 
$E^{\mathrm{ur}}\subseteq E$ denote the 
the maximal unramified subfield. Then under the identification
$$\mathfrak{M}(\mathcal{G}):=(\mathfrak{S}\otimes_{\ZZ_p} \OO_E) e\xrightarrow{\simeq}\bigoplus_{\tau\in\mathrm{Hom}(E^{\mathrm{ur}},K_0)}(\mathfrak{S}\otimes_{\OO_{E^{\mathrm{ur}},\tau}}\OO_E) e_\tau$$
the Frobenius and Verschibung are explicitly
$$\phi_{\mathfrak{M}(\mathcal{G})} e_\tau = f_{\sigma(\tau)}(u)e_{\sigma(\tau)}$$
$$\psi_{\mathfrak{M}(\mathcal{G})} e_\tau = \phi^*(v_{\sigma^{-1}(\tau)}(u))e_{\sigma^{-1}(\tau)}$$
where
\begin{align*}
f_\tau(u) = \prod_{\nu\in \Phi_\tau}h_{\nu}(u),&\quad\quad \Phi_\tau = \{\nu\in \Phi:\nu|_{E^{\mathrm{ur}}} = \tau\}\\
v_\tau(u) = \prod_{\nu\in \Phi_\tau^c}h_{\nu}(u),&\quad\quad\Phi_\tau^c = \{\nu\in \Phi^c:\nu|_{E^{\mathrm{ur}}} = \tau\}\end{align*}
and the polynomials $h_\nu(u)\in \mathfrak{S}\otimes_{\OO_{E^{\mathrm{ur}}},\nu}\OO_E$ satisfy $\mathsf{Eis}(u) = \prod_{\nu\in \mathrm{Hom}_\tau(E,K)}h_{\nu}(u)$.
\end{cor}
\begin{proof}
This follows from Proposition \ref{CMMods} and the definition of 
$R_{\Phi,E^*,L\otimes E}$.
\end{proof}

\begin{exmp}
\label{inertcase}
Assume that $E/\QQ_p$ is unramified and that 
$K_0$ contains all embeddings of $E$ into $\overline{\QQ}_p$. 
Then the module $\mathfrak{M}$ 
from Proposition \ref{CMMods} has the explicit $\mathfrak{S}$-basis 
$\{e_1,\ldots,e_h\}$ so that its Frobenius $\phi_\mathfrak{M}$ 
has the presentation
$$\phi_{\mathfrak{M}}:e_i\mapsto\left\{\begin{array}{ll}\frac{1}{c}
\textsf{Eis}(u)e_{\sigma(i)}&\text{if }i\in \Phi\\
e_{\sigma(i)}&\text{if }i\notin \Phi\end{array}\right.$$
where $\sigma = \phi|_{W}$. Since $\sigma$ is cyclic, we can choose an 
ordering such that $\sigma(i) = i+\delta \mod h$ for $1\le i\le h$ and $1\le \delta\le h$. In fact, since $\sigma$ has order $h$ on the residue field, we can take $\delta = 1$. 
An element $a\in\OO_E$ acts on the vector 
$x = \sum_i x_i$ by $a(x) = \sum_i i(a)x_i$, where $i:\OO_E\hookrightarrow W$. 
It is clear therefore that 
$\phi_{\mathfrak{M}}$ 
commutes with the $\OO_E$-action, and that such a module 
must correspond to a CM $p$-divisible group with CM by $(\OO_E,\Phi)$. 
We let the reader construct the analogous 
explicit form of $\psi_{\mathfrak{M}}$ as an easy exercise.
\end{exmp}

\begin{exmp}
\label{ramex}
Assume that $E/\QQ_p$ is a totally ramified Galois extension. Then we 
can write the (reduced) module $\mathfrak{M}_{\mathfrak{S}_1}$ 
from Proposition \ref{CMMods} 
with an explicit basis $\{e_1,\ldots,e_h\}$ as follows. We let $\pi = \pi_E$ 
and note that 
$$\mathsf{Eis}(u) = u^e - ph(u) = \prod_{g\in \Gal(E/\QQ_p)}\left(u^{\frac{e}{h}}-g(\pi)h_g(u)\right) = \prod_{i=1}^h\left(u^{\frac{e}{h}}-c_i\pi h_i(u)\right)$$
where $g(\pi) = c_i\pi$ for some $c_i\in \OO_E^\times$ and 
$h_i(u)\in W_E[[u]]$, where $W_E = W\otimes_{W\cap E}E$ and 
the degree of $h_i(u)$ is strictly less than $e/h$. Then
$$P_{K,\pi_K,K_0\otimes E}(u)e = \prod_{i\in \Phi^c}\left(u^{\frac{e}{h}}-c_i\pi h_i(u)\right)e.$$
We let 
$H_n(u) = \displaystyle\sum_{\substack{S\subseteq \Phi^c \\ \#S = n}}\prod_{i\in S}c_ih_i(u) = \sum_{k\ge 0} \pi^k H_{n,k}(u)$, where 
$H_{n,k}(u)$ are defined to have unit coefficients in $W_E$ 
(note, in particular, that $H_{0,0}(u) = 1$ and 
$H_{0,k} = 0$ for $k>0$). 
Define 
$$G_{n,j}(u) = \displaystyle\sum_{\substack{k\ge 0 \\ k\equiv j\mod d}}p^{\lfloor\frac{k}{h}\rfloor}H_{n,k}(u),$$ 
and
$$P_k(u) = \sum_{\substack{j,n\ge 0\\j + n = k}} G_{n,j}(u)u^{\frac{e}{h}(h-d-n)},$$
so that
$$P_{K,\pi_K,K_0\otimes E}(u)e = \sum_{j=0}^{d-1} \pi^j\left(\sum_{n=0}^{h-d}\pi^{n}G_{n,j}(u) u^{\frac{e}{h}(h-d-n)}\right)e  = \sum_{k=0}^{h-1}\pi^kP_k(u)e.$$

Using the basis $e_i = \pi^{i-1}$ for $1\le i\le h$, we may write
$$\phi_{\mathfrak{M}}(e_{i}) = \sum_{j=0}^{h-1-i}\left[P_j(u)\right]e_{j+i} + p\sum_{j=h-i}^{h-1}\left[P_j(u)\right]e_{j-(h-i-1)}.$$
Then, reducing the coefficients modulo $p$, 
$G_{n,j}(u) = H_{n,j}(u)$ in the above matrix 
presentation and $\phi_{\mathfrak{M}}$ is represented 
by an upper-triangular matrix.
\end{exmp}

The Galois representation attached to $\mathcal{G}[p^s]$ 
is a torsion-crystalline representation on a one-dimensional $E$-vector space, 
hence abelian. 
By Corollary \ref{cristallinerep}, the image 
agrees with the map
$$N_{\Phi,K}:\mathcal{O}_K^\times\rightarrow (\mathcal{O}_E/p^s\mathcal{O}_E)^\times,$$
and by local class field theory, the representation factors through 
a totally ramified abelian field extension $K^s/K$ of degree $\#(\mathcal{O}_E/p^s\mathcal{O}_E)^\times$. 
Let $s\textsf{-Eis}(u) = u^{e_s} + \ldots + c_sp$ be the characteristic 
polynomial of the extension $K^s/K_0$.

\begin{defn}
\label{smodeldef}
The \textbf{level-$s$ model} of an $\OO_E$-linear CM Kisin module $\mathfrak{M}\simeq \mathfrak{M}(\mathcal{G})$ 
is a rank $1$ projective $\mathfrak{S}\otimes\mathcal{O}_E$-module $\mathfrak{M}^s$ with 
generator $e$ and its Frobenius and Verschiebung endomorphisms given by
$$\phi_{\mathfrak{M}^s}(e) = \frac{1}{c_s}P_{\Phi^c,\pi_{K^s},K_0\otimes E}(u)e,$$
$$\psi_{\mathfrak{M}^s}(e) = \phi^*\left(c_sP_{\Phi,\pi_{K^s},K_0\otimes E}(u)\right)e.$$
In particular, 
$$s\textsf{-Eis}(u) = P_{\Phi,\pi_{K^s},K_0\otimes E}\cdot P_{\Phi^c,\pi_{K^s},K_0\otimes E}(u).$$
\end{defn}

\begin{lem}
\label{smodel}
Let $\mathfrak{M}\simeq \mathfrak{M}(\mathcal{G})$ be an $\OO_E$-linear 
CM Kisin module and $\mathfrak{M}^s$ be 
its level $s$ module. Then $\mathfrak{M}^s\simeq \mathfrak{M}(\mathcal{G}_s)$ 
where $\mathcal{G}\otimes_{\mathcal{O}_K}\OO_{K^s}\simeq \mathcal{G}_s$. 
In particular, all $p^s$-torsion on $\mathcal{G}$ is defined over the field $K^s$.
\end{lem}
\begin{proof}
There are $\#(\OO_E/p^s\OO_E) = p^{s[E:\QQ_p]}$ different $p$-torsion points on 
$\mathcal{G}_s$. Since the height of $\mathcal{G}_s$ equals $[E:\QQ_p]$, 
all the $p^s$-torsion points are rational over $K^s$. The isomorphism 
$\mathfrak{M}^s \simeq \mathfrak{M}(\mathcal{G}_s)$ follows from 
Proposition \ref{CMMods} and Theorem \ref{unifindep}.
\end{proof}

\begin{rem}
\label{LT}
The abelian extensions of $K$ can be described by Lubin-Tate theory. 
Specifically, for the uniformizer $\pi = \pi_E$ of $E$ one can define a formal group 
law by
$$[\pi](u) = \pi u + u^{q}$$
where $q$ denotes the size of the residue field of $E$. Then the class 
extension $K^s/K$ is defined by the Eisenstein polynomial
$$h_{\pi,s}(u) = \frac{[\pi^s](u)}{[\pi^{s-1}](u)} = \pi + \left([\pi^{s-1}](u)\right)^{q-1}$$
where we define $[\pi^0](u) = u$. 
This further explicates Corollary \ref{decompCMMod} 
by setting 
$$f_\tau(u) = \prod_{g\in \Sigma_\tau}(g_*h_{\pi,s})(u)$$
$$v_\tau(u) = \prod_{g\in \Sigma_{\tau}^c}(g_*h_{\pi,s})(u),$$
where
\begin{align*}\Sigma_\tau = \{g\in \Gal(\overline{\QQ}_p/K_0)/\Gal(\overline{\QQ}_p/K):g^{-1}\circ \tau\in \Phi_\tau\},&\quad \Phi_\tau = \{i\in \Phi: i|_{E^{\mathrm{ur}}} = \tau\}\\
\Sigma_\tau^c = \{g\in \Gal(\overline{\QQ}_p/K_0)/\Gal(\overline{\QQ}_p/K):g^{-1}\circ \tau\in \Phi_\tau^c\},&\quad \Phi_\tau^c = \{i\in \Phi^c: i|_{E^{\mathrm{ur}}} = \tau\}.\end{align*}
\end{rem}

\subsection{Quasi-Kisin Modules}
We introduce the notion of a \emph{quasi-Kisin module} 
to simplify computations involving $\OO_E$-linear CM 
Kisin modules, which were introduced in Proposition \ref{CMMods} and  
Corollary \ref{decompCMMod}. The target computation is 
that of the relative Hodge bundle in Lemma \ref{FaltingsIsog}, 
related to Kisin modules by Lemma \ref{diffmod} and Proposition \ref{deg}.

Fix $k$ to be a finite field containing the subfield $\mathbb{F}_q$ 
where $q = p^f$. Define 
$\mathfrak{S}^{(q)}$ to be the module $W(k)[[u]]$ 
with a $q$-Frobenius, i.e., $\sigma^{(q)}|_{W(k)}:\alpha\mapsto \alpha^q$ 
and extends to $u\mapsto u^q$ on $W(k)[[u]]$.

\begin{defn}
A \textbf{quasi-Kisin module} (or \textbf{q-Kisin Module}) 
is a finite $\mathfrak{S}^{(q)}$-module $\mathfrak{M}^{(q)}$ equipped with the 
$q$-Frobenius semi-linear isomorphism
$$1\otimes\phi^{(q)}_{\mathfrak{M}}:\phi^{(q),*}(\mathfrak{M}^{(q)})[1/\mathsf{Eis}^{(q)}(u)]\xrightarrow{\simeq}\mathfrak{M}^{(q)}[1/\mathsf{Eis}^{(q)}(u)]$$
where $\phi^{(q),*}(\mathfrak{M}):=\mathfrak{S}^{(q)}\otimes_{\phi^{(q)}}\mathfrak{M}$ 
and $\mathsf{Eis}^{(q)}(u) = \phi^{f}(\mathsf{Eis}(u))$.
\end{defn}

There is no correspondence such as Theorem \ref{GroupEquiv} for 
quasi-Kisin modules, so they do not describe geometric objects in the 
way the category of Kisin modules does. 
However, one can 
show that any $\OO_E$-linear CM Kisin module is composed of quasi-Kisin pieces.

Fix $E$ to be a ($p$-adic) CM field such that $E/\QQ_p$ has ramification 
degree $r\ge 1$ and inertia degree $f$, and let $\OO_E\subseteq E$ be its ring of integers with 
choice of uniformizer $\pi_E$. Denote by $\mathcal{G}/\OO_K$ an $\OO_E$-linear CM 
$p$-divisible group with ($p$-adic) CM type $(\OO_E,\Phi)$ and let 
$h = [E:\QQ_p]$ be its height and $d$ be its dimension. 
Let $I:=\mathrm{Hom}(E^\mathrm{ur},K_0)$. Then recall by Lemma 
\ref{decompCMMod} that, after enlarging $K$ if necessary, we have the decomposition
$$\mathfrak{M}(\mathcal{G})=(\mathfrak{S}\otimes_{\ZZ_p} \OO_{E}) e\xrightarrow{\simeq}\bigoplus_{i\in I}(\mathfrak{S}\otimes_{\OO_{E^{\mathrm{ur}},i}}\OO_{E}) e_i$$
into projective $\mathfrak{S}\otimes_{\OO_{E^{\mathrm{ur}},i}}\OO_E$-modules 
of rank $1$. We will write
$$\mathfrak{M}(\mathcal{G})_i:=\mathfrak{S}\otimes_{\OO_{E^{\mathrm{ur}},i}}\OO_{E}$$
and use the subscript $i$ to denote the $i$-isotypic piece in this decomposition.

\begin{lem}
\label{Fontaine}
Let $E/\QQ_p$ be a ($p$-adic) CM field, and assume that 
$\mathcal{G}/\OO_K$ is a CM $p$-divisible group with ($p$-adic) CM by $(\OO_E,\Phi)$. 
If we let $\mathfrak{M}^{(q)}\simeq \mathfrak{M}(\mathcal{G})\otimes_{\mathfrak{S}}\mathfrak{S}^{(q)}$ as modules 
be equipped with a $\phi_{\mathfrak{M}}^{(q)}$-Frobenius $\phi_{\mathfrak{M}^{(q)}}^{(q)} = \phi_{\mathfrak{M}}^f$, then there is a decomposition into 
$\mathfrak{S}^{(q)}$-modules
$$\mathfrak{M}^{(q)}\simeq \bigoplus_{i\in I}\mathfrak{M}^{(q)}_i$$
where $\phi_{\mathfrak{M}^{(q)}}^{(q)} = \bigoplus_{i\in I}\phi_{\mathfrak{M}^{(q)}_i}^{(q)}$. If $E/\QQ_p$ is Galois, $\mathfrak{M}_i^{(q)}$ all have the same 
$\mathfrak{S}^{(q)}$-rank.
\end{lem}
\begin{proof}
This decomposition is the quasi-Kisin analogue of Corollary \ref{decompCMMod}. 
In particular, in \emph{loc. cit.}, the Frobenius acts cyclically of 
order $r$ on elements of $I$. The statement on the rank is a basic fact from 
Galois theory.
\end{proof}

We will refer to the decomposition in Lemma \ref{Fontaine} as the 
\textbf{quasi-Kisin decomposition} of the Kisin module $\mathfrak{M}$.

\begin{cor}
\label{Fontainecor}
If $\mathfrak{N}\subseteq \mathfrak{M}$ is a saturated 
$\mathfrak{S}$-submodule, then there exists a saturated 
$\mathfrak{S}^{(q)}$-submodule 
$\mathfrak{N}^{(q)}\subseteq\mathfrak{M}^{(q)}$ with a 
decomposition 
$$\mathfrak{N}^{(q)}\simeq\bigoplus_{i\in I}\mathfrak{N}_i^{(q)}$$
such that $\phi_{\mathfrak{N}^{(q)}}^{(q)} = \bigoplus_{i\in I}\phi_{\mathfrak{N}_i^{(q)}}^{(q)}$ and 
$$v_u(\det \phi_{\mathfrak{N}}) = \left(\frac{p-1}{p^f-1}\right)v_u\left(\det\phi_{\mathfrak{N}^{(q)}}\right) =  \left(\frac{p-1}{p^f-1}\right)\sum_{i\in I}v_u\left(\phi_{\mathfrak{N}^{(q)}_i}\right)$$
where $f$ denotes the inertia degree of $E/\QQ_p$. Conversely, 
a saturated $\mathfrak{S}^{(q)}$-module 
$\mathfrak{N}^{(q)}\subseteq\mathfrak{M}^{(q)}$ corresponds to 
a saturated $\mathfrak{S}$-module 
$\mathfrak{N}\subseteq \mathfrak{M}$ 
if its Frobenius respects the quasi-Kisin decomposition of 
$\mathfrak{M}$.
\end{cor}
\begin{proof}
We construct the $\mathfrak{S}^{(q)}$ module $\mathfrak{N}^{(q)}$ from 
$\mathfrak{N}$ as we constructed $\mathfrak{M}^{(q)}$ from $\mathfrak{M}$ 
in Lemma \ref{Fontaine}. Then $\mathfrak{N}^{(q)}$ is 
necessarily a submodule of $\mathfrak{M}^{(q)}$ as $\mathfrak{N}$ is a 
submodule of $\mathfrak{M}$, and saturation also follows. We moreover obtain 
the decomposition on $\mathfrak{N}^{(q)}$ from the decomposition on 
$\mathfrak{M}^{(q)}$, and the statement on the determinants 
follows by the definition $\phi_{\mathfrak{N}^{(q)}}^{(q)} = \phi_{\mathfrak{N}}^f$ 
and following its effect on the leading coefficient of the Eisenstein polynomial 
at each iteration. Finally, the converse part of the corollary is clear 
by reversing the steps in the above construction.
\end{proof}

We say that a ($p$-adic) CM field $E$ is ramified along a CM 
type $\Phi$ if it is ramified along the embeddings corresponding 
to the elements in $\Phi$.

\begin{cor}
\label{ram11}
If $E/\QQ_p$ is Galois and ramified along the CM type $\Phi$, then the 
$\mathfrak{M}_i^{(q)}$ are all isomorphic as $\mathfrak{S}^{(q)}$-modules.
\end{cor}
\begin{proof}
Since $E/\QQ_p$ is Galois, each of the $\mathfrak{M}_i^{(q)}$ have the same 
rank. Therefore, one needs to check their Frobenius endomorphisms 
$\phi_{\mathfrak{M}_i^{(q)}}^{(q)}$ induce isomorphic $\mathfrak{S}^{(q)}$-modules. 
By Corollary \ref{decompCMMod}, $\phi_{\mathfrak{M}_i^{(q)}}^{(q)} = f_i(u)^{\phi^f}$, 
and since $E/\QQ_p$ ramifies along the CM type, the $f_i(u)$ are all 
$E^{\mathrm{ur}}$-conjugate to each other.
\end{proof}

We use the quasi-Kisin decomposition to reduce the computation 
of the relative Hodge bundle to either the case where the 
($p$-adic) CM field $E/\QQ_p$ is unramified or 
where it is totally ramified and Galois. In each instance 
we need a form of the Serre tensor construction 
on CM $p$-divisible groups. A general introduction to the 
Serre tensor construction is given by \cite{Conrad}.

\begin{prop}
\label{SerreTens}
Let $E'/E$ be a totally ramified extension of degree $\rho$ between 
($p$-adic) CM fields, and choose a uniformizer $\pi_{E'}$ for $E'$ and 
a uniformizer $\pi_E$ for $E$ such that $\mathrm{Nm}_{E'/E}(\pi_{E'}) = \pi_{E}$. 
Let $\mathcal{G}$ be an $\OO_{E}$-linear CM $p$-divisible group and 
$\mathcal{G}' := \mathcal{G} \otimes_{\OO_{E}}\OO_{E'}$ be the 
Serre tensor construction. Denote by $\mathcal{G}_j$ the image of 
$\mathcal{G}$ under $\pi_{E'}^j$ for $j = 0,\ldots,\rho-1$. 
Then every $\OO_{E'}$-stable finite subgroup 
$\mathcal{H}'\subseteq \mathcal{G}'$ has a decomposition 
$\mathcal{H}' = \prod_{j=0}^{\rho -1}\mathcal{H}_j$ where 
$\mathcal{H}_j\subseteq \mathcal{G}_j$ is an $\OO_{E}$-stable 
finite subgroup. In particular,
$$v_u(\det\phi_{\mathfrak{M}_{\mathfrak{S}_1}(\mathcal{H}')}) = \sum_j v_u(\det\phi_{\mathfrak{M}_{\mathfrak{S}_1}(\mathcal{H}_j)}).$$
Moreover, when $\mathcal{H}'\cap \mathcal{G}'[\pi_{E'}^r]$ has 
the same height 
for all $1\le r<\rho$, 
$$v_u(\det\phi_{\mathfrak{M}_{\mathfrak{S}_1}(\mathcal{H}')}) =\rho v_u(\det\phi_{\mathfrak{M}_{\mathfrak{S}_1}(\mathcal{H}_0)}).$$
\end{prop}
\begin{proof}
Enlarge $K$ so that it contains all embeddings of $E'$. Define
\begin{align*}
\mathfrak{M}(\mathcal{G})_i&:=\mathfrak{S}\otimes_{\OO_{E^{\mathrm{ur}},i}}\OO_{E},\\
\mathfrak{M}(\mathcal{G}')_i&:= (\mathfrak{S}\otimes_{\OO_{E^{\mathrm{ur}},i}}\OO_{E})\otimes_{\OO_{E}}\OO_{E'}=\mathfrak{M}(\mathcal{G})_i\otimes_{\OO_{E^{\mathrm{ur}}}}\OO_{E'}.\end{align*}
By Lemma \ref{decompCMMod}, 
$$\mathfrak{M}(\mathcal{G}')=\left((\mathfrak{S}\otimes_{\ZZ_p} \OO_{E})\otimes_{\OO_{E}}\OO_{E'}\right) e\xrightarrow{\simeq}\bigoplus_{i\in I}\left((\mathfrak{S}\otimes_{\OO_{E^{\mathrm{ur}},i}}\OO_{E})\otimes_{\OO_{E}}\OO_{E'}\right) e_i = \bigoplus_i\mathfrak{M}(\mathcal{G}')_i.$$
Moreover,
$$\mathfrak{M}(\mathcal{G}_j)=(\mathfrak{S}\otimes_{\ZZ_p} \OO_{E}\pi^j_{E'})e\xrightarrow{\simeq}\bigoplus_{i\in I}(\mathfrak{S}\otimes_{\OO_{E^{\mathrm{ur}},i}}\OO_{E}\pi^j_{E'}) e_i.$$

Since $\mathcal{H}'$ is $\OO_{E'}$-stable by hypothesis, its Kisin module 
is of the form
$$\mathfrak{M}(\mathcal{H}') = \bigoplus_{i\in I}(\mathfrak{S}\otimes_{\OO_{E^\mathrm{ur}},i}\pi_{E'}^{-k_i}\OO_{E'}/\OO_{E'})e_i \simeq \bigoplus_{i\in I}\pi_{E'}^{-k_i}\mathfrak{M}(\mathcal{G}')_i/\mathfrak{M}(\mathcal{G}')_i$$
for some collection of non-negative integers $k_i$. This induces the 
module isomorphism
$$\mathfrak{M}(\mathcal{H}')_i := \pi_{E'}^{-k_i}\mathfrak{M}(\mathcal{G}')_i/\mathfrak{M}(\mathcal{G}')_i \simeq \bigoplus_{j=0}^{\rho-1}\pi_{E}^{-\lfloor \frac{k_i+j}{\rho}\rfloor}\mathfrak{M}(\mathcal{G}_j)_i/\mathfrak{M}(\mathcal{G}_j)_i.$$
We wish to show that 
$$\mathfrak{M}_j:=\bigoplus_{i\in I} \pi_{E}^{-\lfloor \frac{k_i+j}{\rho}\rfloor}\mathfrak{M}(\mathcal{G}_j)_i/\mathfrak{M}(\mathcal{G}_j)_i$$
is a finite $\OO_{E}$-stable Kisin module so that it corresponds to 
an $\OO_E$-stable finite subgroup $\mathcal{H}_j\subseteq \mathcal{G}_j$ 
under Theorem \ref{GroupExt}.

Since $\mathfrak{M}(\mathcal{G}_j)$ is $\OO_{E}$-linear, there exist 
non-negative integers $\ell_i$ such that
$$\phi_{\mathfrak{M}(\mathcal{G}_j)}:\mathfrak{M}(\mathcal{G}_j)_i\mapsto \pi_{E}^{\ell_i}\mathfrak{M}(\mathcal{G}_j)_{\sigma(i)}.$$
Then 
$$\phi_{\mathfrak{M}(\mathcal{G}')}:\mathfrak{M}(\mathcal{G}')_i\mapsto \pi_{E'}^{\rho\ell_i}\mathfrak{M}(\mathcal{G}')_{\sigma(i)}$$
and since $\mathfrak{M}(\mathcal{H}')$ is $\OO_{E'}$-stable,
$$k_i-\rho\ell_i\le k_{\sigma(i)}.$$
But this implies that
$$\lfloor \frac{k_i+j}{\rho}\rfloor - \ell_i \le \lfloor \frac{k_{\sigma(i)}+j}{\rho}\rfloor$$
so that
$$\phi_{\mathfrak{M}(\mathcal{G}_j)}:\pi_{E}^{-\lfloor \frac{k_i+j}{\rho}\rfloor}\mathfrak{M}(\mathcal{G}_j)_i/\mathfrak{M}(\mathcal{G}_j)_i\mapsto \pi_{E}^{-\lfloor \frac{k_i+j}{\ell}\rfloor + \ell_i}\mathfrak{M}(\mathcal{G}_j)_{\sigma(i)}/\mathfrak{M}(\mathcal{G}_j)_{\sigma(i)}$$
and the image is contained in $\pi_{E}^{-\lfloor \frac{k_{\sigma(i)}+j}{\rho}\rfloor}\mathfrak{M}(\mathcal{G}_j)_{\sigma(i)}/\mathfrak{M}(\mathcal{G}_j)_{\sigma(i)}$, which demonstates stability by the Frobenius endomorphism. 
Thus, by Theorem \ref{GroupEquiv}, 
$\mathfrak{M}_j\simeq \mathfrak{M}(\mathcal{H}_j)$ for some 
$\OO_{E}$-stable finite subgroup $\mathcal{H}_j$ of $\mathcal{G}_j$. The rest of 
the proposition now follows from the multiplicativity of the determinant.
\end{proof}

\begin{cor}
\label{HodgeUnramType}
Let $E/\QQ_p$ be a ($p$-adic) CM field and $\mathcal{G}/\OO_K$ a 
CM $p$-divisible group with ($p$-adic) CM by $(\OO_E,\Phi)$, and we 
assume that $E/\QQ_p$ does not ramify along the CM type $\Phi$. 
We moreover define $\rho = [E:E^{\mathrm{ur}}]$ and let $\pi_E$ be a 
choice of uniformizer for $E$. Then 
$\mathcal{G}\simeq \mathcal{G}^{\mathrm{ur}}\otimes_{\OO_{E^{\mathrm{ur}}}}\OO_E$, 
and for any saturated submodule $\mathfrak{N}\subseteq\mathfrak{M}(\mathcal{G}[p^n])$ corresponding to a group scheme $\mathcal{H}\subseteq \mathcal{G}$ 
such that $\mathcal{H}\cap \mathcal{G}[\pi_E^r]$ has the same height for all 
$1\le r<\rho$, 
there exists a saturated submodule $\mathfrak{N}_0\subseteq \mathfrak{M}(\mathcal{G}^{\mathrm{ur}})$ such that
$$v_u(\det\phi_{\mathfrak{N}}) = \rho v_u(\det\phi_{\mathfrak{N}_0}) = \rho\left(\frac{p-1}{p^{f}-1}\right)\sum_{i\in I}v_u\left(\phi_{\mathfrak{N}_{0,i}^{(q)}}\right)$$
where $\bigoplus_{i\in I}\mathfrak{N}_{0,i}^{(q)}$ is the quasi-Kisin decomposition 
of $\mathfrak{N}_0$.
\end{cor}
\begin{proof}
By the structure theorem for local fields and the hypothesis that 
$E/\QQ_p$ does not ramify along the CM type $\Phi$, 
$\mathcal{G}$ is the Serre tensor 
construction of a $p$-divisible group $\mathcal{G}^{\mathrm{ur}}$ 
with ($p$-adic) CM by $(\OO_{E^\mathrm{ur}},\Phi^{\mathrm{ur}} = \Phi|_{E^{\mathrm{ur}}})$. 
By Proposition \ref{SerreTens} and Proposition \ref{saturation}, 
we then obtain the existence of $\mathfrak{N}_0$ and 
are able to compute the constant $\rho$ appearing in 
the formula. Finally, the decomposition 
of $\mathfrak{N}_0$ is an immediate application of Lemma \ref{Fontaine}.
\end{proof}

\begin{prop}
\label{GaloisClos}
Let $\widetilde{E}/E$ denote the Galois closure of a ($p$-adic) CM field $E$. 
Let $\mathcal{G}$ be an $\OO_E$-linear CM $p$-divisible group and 
$\widetilde{\mathcal{G}}:=\mathcal{G}\otimes_{\OO_E}\OO_{\widetilde{E}}$ 
be the Serre tensor construction. Then for every subgroup scheme 
$\mathcal{H}\subseteq\mathcal{G}$, there is a subgroup scheme 
$\widetilde{\mathcal{H}}\subseteq \widetilde{\mathcal{G}}$ which pulls back 
to $\mathcal{H}$ under the embedding 
$\mathcal{G}\hookrightarrow \widetilde{\mathcal{G}}$ and 
$$v_u(\det\phi_{\mathfrak{M}(\widetilde{\mathcal{H})}}) = \rho v_u(\det\phi_{\mathfrak{M}(\mathcal{H})})$$
where $\rho = [\widetilde{E}:E]$.
\end{prop}
\begin{proof}
Since $E/\QQ_p$ is local, and every unramified local field is Galois, 
$\widetilde{E}/E$ is totally ramified. Moreover, $\OO_{\widetilde{E}}$ is a 
free $\OO_E$-module of rank $\rho$, and thus as a $p$-divisible group, 
$\widetilde{\mathcal{G}}$ is isomorphic to $\mathcal{G}^\rho$. We can 
moreover assume that the closed immersion 
$\mathcal{G}\hookrightarrow\widetilde{\mathcal{G}}$ is just the inclusion 
of the first factor. Then $\widetilde{\mathcal{H}} = \mathcal{H}^\rho$ is a 
subgroup scheme of $\widetilde{\mathcal{G}}$ which pulls back to 
$\mathcal{H}$ with the desired property.
\end{proof}

\begin{cor}
\label{ram22}
Let $E/\QQ_p$ be a ($p$-adic) CM field and $\widetilde{E}/\QQ_p$ denote 
its Galois closure. Let $\mathcal{G}/\OO_K$ be a CM $p$-divisible group 
with ($p$-adic) CM by $(\OO_E,\Phi)$ and assume that $E$ ramifies along 
$\Phi$. We further denote by $\widetilde{\mathcal{G}}$ 
the Serre tensor construction $\mathcal{G}\otimes_{\OO_E}\OO_{\widetilde{E}}$. Then for any saturated submodule 
$\mathfrak{N}\subseteq \mathfrak{M}(\mathcal{G})$ there exists a 
saturated submodule $\widetilde{\mathfrak{N}}\subseteq\mathfrak{M}(\mathcal{G})$ 
such that
$$v_u(\mathrm{det}\phi_{\mathfrak{N}}) = \frac{1}{\rho}v_u(\mathrm{det}\phi_{\widetilde{\mathfrak{N}}}) = \frac{1}{\rho}\sum_{i\in I}v_u\left(\mathrm{det}\phi_{\widetilde{\mathfrak{N}}_i}\right).$$
\end{cor}
\begin{proof}
By Proposition \ref{GaloisClos} and Proposition \ref{saturation}, 
it immediately follows that there exists a saturated submodule 
$\widetilde{\mathfrak{N}}\subseteq\mathfrak{M}(\mathcal{G})$ 
such that 
$v_u(\mathrm{det}\phi_{\widetilde{\mathfrak{N}}}) = \rho(\mathrm{det}\phi_{\mathfrak{N}})$. 
The decomposition $\widetilde{\mathfrak{N}} = \bigoplus_{i\in I} \widetilde{\mathfrak{N}}_i$ 
is likewise an immediate application of Lemma \ref{Fontaine}.
\end{proof}

\begin{rem}
By Corollary \ref{ram11} and Corollary \ref{ram22}, whenever $E/\QQ_p$ 
is ramified along the CM type we may assume that $E/\QQ_p$ 
is totally ramified and Galois to compute the relative Hodge bundle.
\end{rem}

\subsection{HN Theory of CM Torsion Modules}
Harder-Narasimhan theory (reviewed in Section \ref{HNSect}) allows us to 
identify subgroups of $\mathcal{G}[p^n]$ where $\mathcal{G}/\OO_K$ 
is an $\OO_E$-linear CM $p$-divisible group. 
The results here are very specific to the case when $\mathcal{G}$ is a 
CM $p$-divisible group, 
as its isogeny class, and therefore its Harder-Narasimhan theory, 
is determined by the rational 
endomorphism algebra.

We first demonstrate that 
$\mathcal{G}[p^n]$ is semistable for all $n$.

\begin{lem}
\label{semistable}
Let $\mathcal{G}/\OO_K$ be an $\OO_E$-linear CM $p$-divisible group of 
dimension $d$ and height $h$. 
Then $\mathfrak{M}(\mathcal{G}[p^n])$ is 
semistable for all $n\ge 1$. Moreover,
$$\mu(\mathfrak{M}(\mathcal{G}[p^n])) = \frac{h-d}{h}.$$
\end{lem}
\begin{proof}
By Proposition \ref{CMMods}, $\mathfrak{M}(\mathcal{G}[p^n])$ is 
$\mathfrak{S}_n\otimes\mathcal{O}_E$-projective of rank $1$. Then 
since $\OO_E$ preserves the Harder-Narasimhan filtration, any 
submodule of $\mathfrak{M}(\mathcal{G}[p^n])$ in the filtration 
must be $\mathfrak{S}_n\otimes\mathcal{O}_E$-projective of rank $1$ 
as well. The inclusion of such a submodule is an isomorphism under 
the functor $\mathrm{GF}$, and therefore itself an isomorphism. 
This shows $\mathfrak{M}(\mathcal{G}[p^n])$ is semistable.

For the latter claim, we note that $\#\mathcal{G}[p^n]=p^{nh}$, 
so $\mathrm{rk}(\mathfrak{M}(\mathcal{G}[p^n]) = nh$. Moreover, 
since $\mathcal{G}$ has dimension $d$, 
$v_u(\det\phi_{\mathfrak{M}(\mathcal{G}[p^n])}) = (h-d)e$, 
and hence by Proposition \ref{deg}, 
$\mathrm{deg}(\mathfrak{M}(\mathcal{G}[p^n])) = (h-d)ne$.
\end{proof}

Using the fact now that $\mathcal{G}[p^n]$ is semistable and has an 
$\OO_E$-linear structure, we can characterize subgroups in the 
following special case.

\begin{lem}
\label{quentin}
Suppose that $E/\QQ_p$ is totally ramified, and let $\mathcal{G}/\OO_K$ 
be an $\OO_E$-linear CM $p$-divisible group of dimension $d$ and height $h$. 
Let $\pi = \pi_E$ denote a uniformizer of $\OO_E$ and let 
$\mathfrak{N}\subseteq \mathfrak{M}(\mathcal{G}[\pi^r])$ 
be a saturated submodule. Then 
$\mu(\mathfrak{N}) = \mu(\mathcal{G}[\pi^r])$ if and only if 
$\mathfrak{M}/\mathfrak{N} = \mathfrak{M}(\mathcal{G}[\pi^k])$ 
for some integer $0\le k\le r$.
\end{lem}
\begin{proof}
We proceed by induction on the rank of $\mathfrak{N}$. If 
$\mathfrak{N} = 0$, then take $k=r$. Otherwise there is a 
unique integer $k<r$ such that $\mathfrak{N}$ is contained in 
$\pi^k\mathfrak{M}/\pi^r\mathfrak{M}$ but not in 
$\pi^{k+1}\mathfrak{M}/\pi^r\mathfrak{M}$. This induces a 
projection $\mathfrak{N}\to \pi^k\mathfrak{M}/\pi^{k+1}\mathfrak{M}$ 
with kernel $\mathfrak{N}_1$ and image $\mathfrak{N}_2$, each of which 
are objects in $\mathbf{Mod}_{/\mathfrak{S}}^{\phi,r}$. 
If we show that
$$\mathfrak{N}_1 = \pi^{k+1}\mathfrak{M}/\pi^r\mathfrak{M}$$
$$\mathfrak{N}_2 = \pi^k\mathfrak{M}/\pi^{k+1}\mathfrak{M}$$
this will imply $\mathfrak{N} = \pi^k\mathfrak{M}/\pi^r\mathfrak{M}$, 
and thus the statement.

Note that $\mathfrak{N}_2$ is a nonzero Kisin submodule of $\pi^k\mathfrak{M}/\pi^{k+1}\mathfrak{M} = \mathfrak{M}/\pi\mathfrak{M}$. The latter is of rank $1$ 
and degree $\frac{h-d}{h}$, hence semistable. Thus 
$\mathfrak{N}_2$ has rank $1$ and its slope is at least 
$\frac{h-d}{h}$. However, $\mathfrak{N}_2$ is a quotient of $\mathfrak{N}$, 
which is semistable of slope $\frac{h-d}{h}$, and thus the slope 
of $\mathfrak{N}_2$ is at most and so equal to $\frac{h-d}{h}$. 
It follows $\mathfrak{N}_2 = \pi^k\mathfrak{M}/\pi^{k+1}\mathfrak{M}$ 
as a submodule of the same degree and the same rank.

Now by induction $\mathfrak{N}_1 = \pi^\ell\mathfrak{M}/\pi^r\mathfrak{M}$ for 
some integer $k+1\le \ell\le r$. Since $\mathfrak{N}_2$ is $\pi$-torsion, 
$\pi\mathfrak{N}\subseteq \mathfrak{N}_1$. On the other hand, since 
$\mathfrak{N}_1$ is $\pi^{r-\ell}$-torsion, $\mathfrak{N}$ is 
$\pi^{r-\ell+1}$-torsion. Therefore 
$\mathfrak{N}\subseteq \pi^{\ell-1}\mathfrak{M}/\pi^r\mathfrak{M}$. 
Now by the definition of $k$, $\ell-1\le k$, but since we also have $k+1\le \ell$, 
$\ell = k+1$. Hence $\mathfrak{N}_1 = \pi^{k+1}\mathfrak{M}/\pi^r\mathfrak{M}$, 
which demonstrates the statement.
\end{proof}

Using these lemmas, we identify the torsion quotients $\mathfrak{M}$ of 
$\mathfrak{M}(\mathcal{G})$ corresponding to the exact sequence
\begin{equation}
\label{oprojtors}
0\to \mathfrak{M}(\mathcal{G}')\to \mathfrak{M}(\mathcal{G})\to \mathfrak{M}\to 0
\end{equation}
where $\mathcal{G}'$ is a CM $p$-divisible group with CM by 
$(\OO,\Phi)$ for some non-maximal order $\OO\subseteq E$ 
as those saturated Kisin modules with strictly smaller slope 
than that of $\mathfrak{M}(\mathcal{G})$.
Since it is very important, we spell it out explicitly below.

\begin{lem}
\label{correctlem}
Let $\mathcal{G}/\OO_K$ be an $\OO_E$-linear CM $p$-divisible 
group and $\mathfrak{N}\subseteq\mathfrak{M}(\mathcal{G}[p^n])$ 
be a saturated submodule such that $\mathfrak{M}(\mathcal{G}[p^n])/\mathfrak{N}$ 
is a quotient of the form (\ref{oprojtors}). Then $\mu(\mathfrak{N})<\mu(\mathfrak{M})$.
\end{lem}
\begin{proof}
By Corollary \ref{HodgeUnramType} and Corollary \ref{ram22}, we may assume 
that $E/\QQ_p$ is either unramified or totally ramified (and Galois). 
The latter case follows by Lemma \ref{quentin} and the semistability 
of $\mathcal{G}[p^n]$ proved in Lemma \ref{semistable}. 
In the former case, the argument follows similarly: since $\mathfrak{N}$ 
is not $\OO_E$-stable the proof of Lemma \ref{semistable} shows it 
cannot lie in the Harder-Narasimhan filtration of $\mathfrak{M}(\mathcal{G}[p^n])$, 
and a modification of the argument in Lemma \ref{quentin} 
demonstrates that only the $\OO_E$-stable submodules preserve the 
HN slope. Then since $\mathcal{G}[p^n]$ is semistable by 
Lemma \ref{semistable}, the statement follows.
\end{proof}

The corollary combined with the preceding two lemmas show \emph{a priori} that 
the change in Faltings height in the formula given by Lemma \ref{FaltingsIsog} 
must be \emph{positive} when $A_1$ is taken to have CM by a maximal order. 
Positivity of the Faltings height variation is not strong enough to prove 
a Northcott property, since the contribution by the relative Hodge bundle 
to the variation 
must be shown to not compete with the degree term. 
However, it identifies a ``minimal'' element within 
an isogeny class of CM abelian varieties, namely the CM abelian variety 
corresponding to the 
maximal order (or the CM abelian variety over the minimal 
field of definition), which is a result of intrinsic interest.

\section{Differential Computations}
\label{ramcompus}
\subsection{Unramified Case}
\label{inert}
Fix $E$ to be a ($p$-adic) CM field such that $E/\QQ_p$ is unramified, and 
let $\OO_E\subseteq E$ be its ring of integers with choice of uniformizer $\pi_E$. 
We bound slopes of submodules of 
$\mathfrak{M}(\mathcal{G}[p^n])$, 
where $\mathcal{G}/\OO_K$ is an $\OO_E$-linear CM $p$-divisible group with ($p$-adic) CM 
type $(\OO_E,\Phi)$. We denote by $h = [E:\QQ_p]$ the height of 
$\mathcal{G}$ and by $d$ its dimension.

We first compute saturated submodules of 
$\mathfrak{M}_{\mathfrak{S}_1}(\mathcal{G}[p])$. By Proposition \ref{saturation}, 
these correspond to finite flat subgroup schemes $\mathcal{H}\subseteq \mathcal{G}[p]$ 
of order $p^{k}$ where $k<h$. 
Since $p\in \OO_E$ is prime, we have \emph{a priori} 
that all such saturated submodules correspond to torsion quotients 
of the form in (\ref{oprojtors}). It will 
be evident from 
the computations that this is the minimal degree for the isogeny to such a CM 
$p$-divisible group $\mathcal{G}'$.

Recall the presentation in Example \ref{inertcase} for 
$\mathfrak{M}(\mathcal{G})$. Here we assumed that $K_0$ contains 
all embeddings of $E$ into $\overline{\QQ}_p$ so we can choose an 
$\mathfrak{S}$-basis $\{e_1,\ldots,e_h\}$ such that 
$\phi_{\mathfrak{M}}$ has the presentation
\begin{equation}
\label{unrampresent}
\phi_{\mathfrak{M}}:e_i\mapsto\left\{\begin{array}{ll}
\frac{1}{c}\textsf{Eis}(u)e_{i+1}&\text{if }i\in \Phi\\
e_{i+1}&\text{if }i\notin \Phi\end{array}\right.
\end{equation}
where addition on the indexing set is performed modulo $h$. 
Here we consider the set $\{1,\ldots,h\}$ to be a torsor 
under the action of $\mathrm{Gal}(E/\QQ_p)$, where the 
action is by addition modulo $h$ and induced from 
the corresponding cyclic permutaton of the basis $(e_1,\ldots,e_h)$.

\begin{lem}
\label{quasiquasi}
Let $\mathfrak{M} = \mathfrak{M}(\mathcal{G})$ 
where $\mathcal{G}$ has CM type $\Phi = \{\alpha_1,\ldots,\alpha_{h-d}\}\subseteq \{1,\ldots,h\}$ as an ordered set. If 
we let $\mathfrak{M}^{(q)}\simeq \mathfrak{M}$ as modules with a 
$\phi^{(q)}$-Frobenius $\phi_{\mathfrak{M}^{(q)}}^{(q)} = \phi_{\mathfrak{M}}^q$, then
$$\mathfrak{M}^{(q)}\simeq \bigoplus_{\tau\in \Hom(E^{\mathrm{ur}},K_0)}\mathfrak{M}^{(q)}_\tau$$
where each $\mathfrak{M}^{(q)}_\tau$ is a $\mathfrak{S}^{(q)}$-module 
of rank $1$ with $\phi^{(q)}$-Frobenius
$$\phi_{\mathfrak{M}^{(q)}_\tau}^{(q)}(x) = \left(u^{e\sum_{s=1}^{h-d} p^{h-\tau^{-1}(\alpha_s)}} + pg_\tau(u)\right)\cdot x$$
for some polynomial $g_\tau(u)\in \mathfrak{S}^{(q)}$ with degree 
strictly less than $e\sum_{s=1}^{h-d} p^{h-\tau^{-1}(\alpha_s)}$.
\end{lem}
\begin{proof}
We apply the quasi-Kisin decomposition in Lemma \ref{Fontaine}. The 
precise form follows from the presentation (\ref{unrampresent}) and the 
theory of Eisenstein polynomials.
\end{proof}

Let $I = \mathrm{Hom}(E^{\mathrm{ur}},K_0)$ and $S_r\subseteq I$ denote a subset of 
size $r$.

\begin{prop}
\label{unramp}
Let $\mathcal{G}/\OO_K$ be an $\OO_E$-linear $p$-divisible group 
of type $(\OO_E,\Phi)$ and assume that $E/\QQ_p$ is unramified. Then 
for any subgroup $\mathcal{H}\subseteq \mathcal{G}[p]$ of height $k<h$,
$$v_u\left(\mathrm{det}(\phi_{\mathfrak{M}(\mathcal{H})})\right) = \min_{S_{k}\subseteq I}\left\{\sum_{\tau\in S_{k}}\left(\sum_{s=1}^{h-d} p^{h-\tau^{-1}(\alpha_s)}\right)\right\}\cdot \left(\frac{p-1}{p^h-1}\right)e.$$
\end{prop}
\begin{proof}
Let $\mathfrak{M}^{(q)}$ denote the quasi-Kisin module 
corresponding to $\mathfrak{M}(\mathcal{G}[p])$ by Lemma 
\ref{Fontaine}. 
Then by Proposition \ref{saturation} and Corollary \ref{Fontainecor}, there 
exists a saturated quasi-Kisin submodule $\mathfrak{N}^{(q)}\subseteq \mathfrak{M}^{(q)}$ 
of rank $h-k$ corresponding to the subgroup $\mathcal{H}\subseteq \mathcal{G}[p]$ 
and by which one can compute 
$v_u\left(\mathrm{det}(\phi_{\mathfrak{M}(\mathcal{H})})\right)$. 
Then using the presentation in Lemma \ref{quasiquasi} 
and again Corollary \ref{Fontainecor}, 
a submodule of rank $h-k$ is saturated whenever the iterated sum
$$\sum_{\tau\in S_{h-k}}\left(\sum_{s=1}^{h-d} p^{h-\tau^{-1}(\alpha_s)}\right)$$
is maximized over different sets $S_{h-k}\subseteq I$.
\end{proof}

We now compute the relative Hodge bundle for saturated submodules of 
$\mathfrak{M}(\mathcal{G}[p^n])$. Any $p^n$-torsion 
finite flat subgroup scheme $\mathcal{H}\subseteq \mathcal{G}[p^n]$ 
is characterized by the non-increasing tuple 
$(\lambda_1,\ldots,\lambda_h)$ where $0\le \lambda_i\le n$ are all 
integers, $\lambda_1 = n$, and $\lambda_{k+1},\ldots,\lambda_h = 0$ for some 
$k<h$ determined by $\mathcal{H}$ (this corresponds to 
the assumption that $\mathcal{H}\cap \mathcal{G}[p] = 0$, 
for otherwise we get the same quotient $p$-divisible group). 
Let $T\subseteq\{0,\ldots,h-1\}$ 
consist of those elements $i$ such that $\lambda_{k-i}-\lambda_{k-i+1}>0$ 
(in particular, $0\in T$).

\begin{prop}
\label{formulaunram}
Let $\mathcal{G}/\OO_K$ be an $\OO_E$-linear CM $p$-divisible group 
with ($p$-adic) CM type $(\OO_E,\Phi)$, and assume that $E/\QQ_p$ is unramified. 
Let $\mathcal{H}\subseteq\mathcal{G}[p^n]$ be of type $(\lambda_1,\ldots,\lambda_h)$ 
and have $p$-height $k<h$. Then letting
$$v_u\left(\mathrm{det}(\phi_{\mathfrak{M}(\mathcal{H})})\right) = \sum_{i\in T}\sum_{j=\lambda_{k-i+1}+1}^{\lambda_{k-i}}\left(\min_{S_{d(i)}\subseteq I}\left\{\sum_{\tau\in S_{d(i)}}\left(\sum_{s=1}^{h-d} p^{h-\tau^{-1}(\alpha_s)}\right)\right\}\cdot\left(\frac{p-1}{p^j(p^h-1)}\right)\right)e$$
where $d(i) =\max\{j:\lambda_{j}=\lambda_{k-i}\}$ (in particular, $d(0) = k$).
\end{prop}
\begin{proof}
We assume at first that $\lambda_1 = \lambda_k = n$ so that 
$T$ consists of a single element and $d(i) = k$. Then 
the formula simplifies to the inner sum and we iterate on $n$. 

When $n=1$ this is precisely 
Proposition \ref{unramp}. When $n>1$, for any 
$0\le r\le n$ the finite flat group scheme $\mathcal{H}$ 
factors by a short exact 
sequence of finite flat group schemes
\begin{equation}
\label{SESr}
0\to \mathcal{H}_r\to \mathcal{H}\to \mathcal{H}_{n-r}\to 0
\end{equation}
where $\mathcal{H}_r\subseteq \mathcal{G}[p^r]$. Geometrically, 
means the isogeny $\mathcal{G}\to \mathcal{G}/\mathcal{H}$ 
decomposes into a sequence of isogenies 
$\mathcal{G}\to \mathcal{G}/\mathcal{H}_r\to \mathcal{G}/\mathcal{H}$. 
Iterating on $r$ gives a sequence of length $n$ where each isogeny in 
the sequence has $p$-torsion kernel.

Let $r = 1$. Then (\ref{SESr}) transforms under $\mathfrak{M}$ to 
the short exact sequence of Kisin modules
$$0\to \mathfrak{M}(\mathcal{G}_1)\to \mathfrak{M}(\mathcal{G})\to \mathfrak{M}(\mathcal{H}_1)\to 0$$
which by Lemma \ref{Fontaine} corresponds to a short exact sequence 
of quasi-Kisin modules
$$0\to \mathfrak{M}^{(q)}(\mathcal{G}_1)\to \mathfrak{M}^{(q)}(\mathcal{G})\to\mathfrak{M}^{(q)}(\mathcal{H}_1)\to 0.$$
By Lemma \ref{quasiquasi} and the proof of Proposition \ref{unramp}, we can 
write 
$$\phi_{\mathfrak{M}^{(q)}(\mathcal{H}_1)}^{(q)} = \bigoplus_{i\in \mathcal{S}_k}\phi_{\mathfrak{M}_i^{(q)}}^{(q)},$$ 
where $\mathcal{S}_k\subseteq I$ is a 
subset of size $k$ which minimizes the sum 
$\sum_{\tau\in \mathcal{S}_{k}}\left(\sum_{s=1}^{h-d} p^{h-\tau^{-1}(\alpha_s)}\right)$. 
Then since $\mathfrak{M}(\mathcal{H}_1)$ is the Kisin module of 
a $p$-torsion group of height $k$, 
$\mathfrak{M}^{(q)}(\mathcal{G}_1)$ must have index $p$ in 
$\mathfrak{M}^{(q)}(\mathcal{G})$ on $k$ generators. In other words, 
there is a presentation of $\mathfrak{M}^{(q)}(\mathcal{G}_1)$ such that
$$\phi_{\mathfrak{M}^{(q)}(\mathcal{G}_1)}^{(q)} \equiv \bigoplus_{\tau\in \mathcal{S}_k} u^{\frac{1}{p}\sum_{s=1}^{h-d} p^{h-\tau^{-1}(\alpha_s)}}\oplus \bigoplus_{\tau\notin\mathcal{S}_k}u^{\sum_{s=1}^{h-d} p^{h-\tau^{-1}(\alpha_s)}}\mod p.$$
Iterating on $r$, this provides a presentation of 
$\mathfrak{M}^{(q)}(\mathcal{G}_r)$ such that
$$\phi_{\mathfrak{M}^{(q)}(\mathcal{G}_r)}^{(q)} \equiv \bigoplus_{\tau\in \mathcal{S}_k} u^{\frac{1}{p^r}\sum_{s=1}^{h-d} p^{h-\tau^{-1}(\alpha_s)}}\oplus \bigoplus_{\tau\notin\mathcal{S}_k}u^{\sum_{s=1}^{h-d} p^{h-\tau^{-1}(\alpha_s)}}\mod p.$$
Replacing $\mathcal{G}$ by $\mathcal{G}_r$ in Proposition \ref{unramp} 
allows us to compute the relative Hodge 
bundle on $\mathcal{H}/\mathcal{H}_r\cap \mathcal{G}_r[p]$. Then the formula 
follows from the additivity of the relative Hodge bundle on exact sequences 
of finite flat group schemes.

When $\lambda_1> \lambda_k$ we argue in the same way, only at each step of 
the iteration on $r$ the height of $\mathcal{H}/\mathcal{H}_r\cap \mathcal{G}_r[p]$ 
can change. We account for it by constructing the appropriate set $T$ in 
the statement.
\end{proof}

\begin{thm}
\label{unrambound}
Let $\mathcal{G}$ be an $\OO_E$-linear CM $p$-divisible group of 
type $(\OO_E,\Phi)$ and assume that $E/\QQ_p$ is unramified. Then 
for any subgroup $\mathcal{H}\subseteq \mathcal{G}[p^n]$ of $p$-height 
$k<h$,
\begin{align*}\# \frac{1}{[K:\QQ_p]}\log s^*\Omega_{\mathcal{H}/\OO_K}^k &\le \left(\frac{p-1}{p^{\delta}-1}\right)\left(\frac{1-p^{-k}}{1-p^{-1}}\right)\left(\frac{1-p^{-n}}{1-p^{-1}}\right)\log p\end{align*}
where $\delta = \frac{h}{h-d}>1$.
\end{thm}
\begin{proof}
By Proposition \ref{deg}, the inequality is obtained by 
computing upper bounds on the formula for 
$v_u(\mathrm{det}\phi_{\mathfrak{M}(\mathcal{H})})$ 
in Proposition \ref{formulaunram}. 
First, the quantity
$$\min_{S_{d(i)}\subseteq I}\left\{\sum_{\tau\in S_{d(i)}}\left(\sum_{s=1}^{h-d} p^{h-\tau^{-1}(\alpha_s)}\right)\right\}$$
is maximal when the CM type is $\{1,1+\delta,\ldots,1+\delta(h-d-1)\}$, 
where $\delta = \frac{h}{h-d}$. 
For a given $d(i)$, this maximal value is
$$M(d(i)) = \sum_{\ell = 0}^{d(i)-1}\left(\sum_{s=1}^{h-d} p^{h-\delta s -\ell}\right) = \left(\frac{p^h-1}{p^{\delta}-1}\right)\left(\frac{1-p^{-d(i)}}{1-p^{-1}}\right).$$
The inner sum iterating this quantity in Proposition \ref{formulaunram} 
is maximal when $d(i) = k$ for all $i$. 
Then
\begin{align*}
v_u\left(\mathrm{det}(\phi_{\mathfrak{M}(\mathcal{H})})\right) &\le M(k)\left(\frac{p-1}{p^h-1}\right)\left(\frac{1-p^{-n}}{1-p^{-1}}\right)e\\
& = \left(\frac{p-1}{p^{\delta}-1}\right)\left(\frac{1-p^{-k}}{1-p^{-1}}\right)\left(\frac{1-p^{-n}}{1-p^{-1}}\right)e.
\end{align*}
\end{proof}

\subsection{Totally Ramified Case}
\label{ram}
Fix $E$ to be a ($p$-adic) CM field such that $E/\QQ_p$ is 
totally ramified, and let $\OO_E\subseteq E$ be its ring of 
integers with choice of uniformizer $\pi_E$. As in Section \ref{inert}, 
we bound slopes of saturated Kisin submodules of 
$\mathfrak{M} = \mathfrak{M}_{\mathfrak{S}_1}(\mathcal{G}[p])$, 
where $\mathcal{G}/\OO_K$ is an $\OO_E$-linear CM $p$-divsible group 
with ($p$-adic) CM type $(\OO_E,\Phi)$. We denote by $h = [E:\QQ_p]$ 
the height of $\mathcal{G}$ and by $d$ its dimension.

We first make the computation explicit in the case $h = 2$. 

\begin{prop}
\label{lubintate}
Assume that $[E:\QQ_p]=2$ and $p$ ramifies in $E$, and let 
$\mathcal{G}/\OO_K$ be an $\OO_E$-linear CM $p$-divisible group. 
Denote by $\mathfrak{M}:=\mathfrak{M}_{\mathfrak{S}_1}(\mathcal{G}[p])$. 
Then a saturated $\mathfrak{S}_1$-line $\mathfrak{L}\subseteq \mathfrak{M}$ 
has either
$$v_u(\mathrm{det}(\phi_{\mathfrak{L}})) =\left\{\begin{array}{cl}
\frac{e}{2}&\text{if }\mathfrak{M}/\mathfrak{L}\simeq\mathfrak{M}(\mathcal{G}[\pi_E]),\\
 \left(\frac{2p-1}{p}\right)\frac{e}{2}&\text{otherwise.}\end{array}\right.$$
\end{prop}
\begin{proof}
We let $\pi = \pi_E$ and use Lemma \ref{smodel} and 
Remark \ref{LT} to construct the 
level-$2$ model $\mathfrak{M}^2(\mathcal{G})$ using 
the Eisenstein polynomial
$$\text{2-}\mathsf{Eis}(u) = (\pi + (\pi u + u^{p})^{p-1})(\overline{\pi} + (\overline{\pi} u + u^{p})^{p-1}) = u^e + p h(u)$$
so that
$$P_{K,\pi_K,K_0\otimes E}(u) = \pi + (\pi u + u^{p})^{p-1} = u^{\frac{e}{2}} + \pi h_2(u).$$
Then on the $\mathfrak{S}$-basis $\{1,\pi\}$, 
the Frobenius of $\mathfrak{M}_{\mathfrak{S}_1}^2(\mathcal{G}[p])$ 
has the explicit presentation
$$\phi_{\mathfrak{M}_{\mathfrak{S}_1}^2(\mathcal{G}[p])} = \left(\begin{array}{cc}
u^{\frac{e}{2}}&-u^{\left(\frac{p-1}{p}\right)\frac{e}{2}}+1\\
0&u^{\frac{e}{2}}\end{array}\right).$$

Since $\mathfrak{L}$ has dimension $1$, 
$\phi_{\mathfrak{L}}$ acts as semi-linear 
multiplication by a monomial $u^\mu$ on a 
choice of generating element $e$ of $\mathfrak{L}$. By the 
commutation of the diagram

\bigskip
\centerline{\begin{xy}
(0,15)*+{\mathfrak{L}}="a";
(20,15)*+{\mathfrak{M}_{\mathfrak{S}_1}^2(\mathcal{G}[p])}="b";
(0,0)*+{\mathfrak{L}}="c";
(20,0)*+{\mathfrak{M}_{\mathfrak{S}_1}^2(\mathcal{G}[p])}="d";
{\ar^{\phi_{\mathfrak{M}_{\mathfrak{S}_1}^2(\mathcal{G}[p])}} "b";"d"};{\ar@{^{(}->} "c";"d"};
{\ar@{^{(}->} "a";"b"};{\ar_{u^{\mu} = \phi_\mathfrak{L}} "a";"c"};
\end{xy}}
\bigskip

\noindent an element 
$v = (f_1,f_2)\in \mathfrak{M}_{\mathfrak{S}_1}^2(\mathcal{G}[p])$ lies in $\mathfrak{L}$ 
if and only if
\begin{align*}
u^\mu f_1 &= u^{\frac{e}{2}} f_1^p\\
u^\mu f_2 &= -u^{\left(\frac{p-1}{p}\right)\frac{e}{2}}f_1^p + f_1^p + u^{\frac{e}{2}} f_2^p.\end{align*}
There are exactly two solutions to this system of equations, which 
may be found and verified empirically:

\begin{enumerate}
\item $f_1 = 0$ and $f_2 = 1$: this satisfies $\mu = \frac{e}{2}$ 
and corresponds to the group scheme $\mathcal{G}[\pi]$ 
(the latter conclusion is also seen from Lemma \ref{quentin}).

\item $f_1 = -u^{\frac{e}{2p}}$ and $f_2 = 1$: this satisfies 
$\mu = \left(\frac{2p-1}{p}\right)\frac{e}{2}$ and by 
Lemma \ref{quentin} is a solution to (\ref{oprojtors}).
\end{enumerate}
\end{proof}

\begin{thm}
\label{EllRamComp}
Assume that $[E:\QQ_p] = 2$ and that $p$ ramifies in $E$, and let 
$\mathcal{G}/\OO_K$ be an $\OO_E$-linear CM $p$-divisible group. 
Then for any subgroup $\mathcal{H}\subseteq \mathcal{G}[p^n]$ 
such that $\mathcal{H}\cap \mathcal{G}[\pi_E^r] = 1$ for all 
$r\ge 0$,
$$\#\frac{1}{[K:\QQ_p]}\log s^*\Omega_{\mathcal{H}/\OO_K}^2 = \frac{1}{2p}\left(\frac{1-p^{-n}}{1-p^{-1}}\right)\log p.$$
\end{thm}
\begin{proof}
When $n=1$, by our hypothesis on 
$\mathcal{H}$, the equality is a translation of Proposition \ref{lubintate} 
using Proposition \ref{deg}.

When $n>1$, for any $0\le r\le n$, $\mathcal{H}$ factors  
into the short exact sequence of finite flat group schemes
$$0\to \mathcal{H}_r\to \mathcal{H}\to \mathcal{H}_{n-r}\to 0$$
where $\mathcal{H}_r\subseteq \mathcal{G}[p^r]$. Iterating this 
factorization, $\mathcal{H}$ is constructed as an extension of 
$n$ groups each of $p$-torsion.

Let $r=1$. Then the following resolution on $\mathcal{H}_1$
$$0\to \mathcal{H}_1\to \mathcal{G}\to \mathcal{G}_1\to 0$$
induces the exact sequence of Kisin modules
$$\mathfrak{M}(\mathcal{G}_1)\to \mathfrak{M}(\mathcal{G})\to \mathfrak{M}(\mathcal{H}_1)\to 0.$$
By Proposition \ref{saturation}, and 
Lemma \ref{smodel} and Remark \ref{LT} to construct the 
level-$2$ model $\mathfrak{M}^2(\mathcal{G})$, there exists a saturated line  
$\mathfrak{L}\subseteq \mathfrak{M}^2_{\mathfrak{S}_1}(\mathcal{G})$ 
so that 
$\mathfrak{M}^2(\mathcal{H}_1) = \mathfrak{M}^2_{\mathfrak{S}_1}(\mathcal{G}[p])/\mathfrak{L}$ and 
which satisfies 
\begin{align*}
\phi_{\mathfrak{L}}&\equiv u^{\left(\frac{2p-1}{p}\right)\frac{e}{2}}\mod p\\
\phi_{\mathfrak{M}(\mathcal{H}_1)}&\equiv u^{\frac{e}{2p}}\mod p.\end{align*}
Choosing the $\mathfrak{S}$-basis $\{1,-u^{\frac{e}{2p}} + \pi\}$, 
we may then write
$$\phi_{\mathfrak{M}_{\mathfrak{S}_1}^2(\mathcal{G})}\equiv \left(\begin{array}{cc}
u^{\frac{e}{2p}} &-u^{\left(\frac{p-1}{p}\right)\frac{e}{2} }+1\\
0&u^{\frac{2p-1}{2p}e}\end{array}\right)\mod p.$$
Since $\mathcal{H}_1$ has $p$-torsion, the image of 
$\mathfrak{M}(\mathcal{G}_1)$ in $\mathfrak{M}^2(\mathcal{G})$ 
has index $p$, so that we may write
$$\phi_{\mathfrak{M}_{\mathfrak{S}_1}(\mathcal{G}_1)}\equiv 
\left(\begin{array}{cc}
u^{\frac{e}{2p}} &-u^{\left(\frac{p-1}{p^2}\right)\frac{e}{2} }+1\\
0&u^{\frac{2p-1}{2p^2}e}\end{array}\right)\mod p.$$
Iterating on $r$,
$$\phi_{\mathfrak{M}_{\mathfrak{S}_1}(\mathcal{G}_r)}\equiv 
\left(\begin{array}{cc}
u^{\frac{e}{2p}} &-u^{\left(\frac{p-1}{p^{r+1}}\right)\frac{e}{2} }+1\\
0&u^{\frac{2p-1}{2p^{r+1}}e}\end{array}\right)\mod p.$$
Replacing $\mathfrak{M}^2(\mathcal{G})$ by 
$\mathfrak{M}(\mathcal{G}_r)$ and $\mathcal{H}_1$ by 
$\mathcal{H}_r/\mathcal{H}_{r-1}$ in Proposition \ref{lubintate}, 
the result follows by the additivity of the relative Hodge bundle on short 
exact sequences of finite flat group schemes.
\end{proof}

Consider now $E/\QQ_p$ totally ramified and Galois, 
and assume 
that $K$ contains all 
embeddings of $E$ into $\overline{\QQ}_p$. Then there exists a factorization
$$P_{K,\pi_K,K_0\otimes E}(u)e = \prod_{i\in \Phi^c}(u^{\frac{e}{h}}-c_i\pi h_i(u))e$$
such that $c_i\in \OO_E^\times$, $\pi = \pi_E$, 
and $h_i(u)\in W_E[[u]]$ is a polynomial of degree strictly smaller 
than $e/h$, where $W_E = W\otimes_{W\cap E}E$. 
Then as in Example \ref{ramex} we can choose the 
$\mathfrak{S}$-basis $\{e_1,\ldots,e_h\}$ with 
$e_i = \pi_E^{i-1}$
so that 
$\phi_{\mathfrak{M}_{\mathfrak{S}_1(\mathcal{G}[p])}}$ has the presentation
$$\phi_{\mathfrak{M}_{\mathfrak{S}_1(\mathcal{G}[p])}}(e_{i}) = \sum_{j=0}^{h-1-i}\left[P_j(u)\right]e_{j+i}$$
where the polynomials $P_j(u)$ are defined in Example \ref{ramex}. 
Note that, by the condition on the degree of $h_i(u)$, 
the degrees of $P_j(u)$ are strictly dominated by the degree of  
$P_1(u)$ for $j>1$.

For a general saturated 
submodule, we compute only the following (weak) bound 
which will be sufficient for establishing the Northcott property, as 
within any isogeny class there will always only be finitely many primes 
exhibiting this ramification behavior. Precise formulas are computed 
in special cases of totally ramified, Galois fields $E/\QQ_p$ in 
Section \ref{lastsect}.

\begin{thm}
\label{rambound}
Let $\mathcal{G}/\OO_K$ be an $\OO_E$-linear CM 
$p$-divisible group and assume that $E/\QQ_p$ is totally 
ramified. Then for any subgroup $\mathcal{H}\subseteq \mathcal{G}[p^n]$ 
of $p$-height $k$ such that $\mathcal{H}\cap \mathcal{G}[\pi_E^r] = 1$ for 
all $r\ge 0$,
$$\#\frac{1}{[K:\QQ_p]}\log s^*\Omega^k_{\mathcal{H}/\OO_K} < k\left(\frac{h-d}{h}\right)\left(\frac{1-p^{-n}}{1-p^{-1}}\right)\log p.$$
\end{thm}
\begin{proof}
By Lemma \ref{correctlem} \emph{a priori} one has
$v_u(\mathrm{det}(\phi_{\mathfrak{M}_{\mathfrak{S}_1}(\mathcal{H})}))<ek\left(\frac{h-d}{h}\right)$ 
by our hypothesis on $\mathcal{H}$. 
To deduce the stronger bound, we lose nothing 
to assume that 
$\mathcal{H}$ has constant $p^r$-height $k$ 
for all $1\le r\le n$. 
Saturated rank $h-k$ submodules of $\mathfrak{M}_{\mathfrak{S}_1}(\mathcal{G})$ 
have a corresponding saturated rank $1$ submodule 
$\mathfrak{L}\subseteq\bigwedge^{h-k}\mathfrak{M}_{\mathfrak{S}_1}(\mathcal{G})$ 
with Frobenius 
$$\phi_{\bigwedge^{h-k}\mathfrak{M}_{\mathfrak{S}_1}(\mathcal{G})} = \bigwedge^{h-k}\phi_{\mathfrak{M}_{\mathfrak{S}_1}(\mathcal{G})}.$$ 
Since $\phi_{\mathfrak{M}_{\mathfrak{S}_1}(\mathcal{G})}$ 
has a representation by an 
upper triangular matrix as in Example \ref{ramex}, 
$\phi_{\bigwedge^{h-k}\mathfrak{M}_{\mathfrak{S}_1}(\mathcal{G})}$ 
also has a representation 
by an upper triangular matrix by taking the appropriate tensor power. 
Moreover, the degree of the polynomial entries in the resulting matrix 
are strictly bounded by the degree of the entries on the diagonal, a condition 
which is imposed from Example \ref{ramex}.

Let 
$v = (f_1,\ldots,f_{\binom{h}{h-k}})\in \mathfrak{L}\subseteq \bigwedge^{h-k}\mathfrak{M}_{\mathfrak{S}_1}(\mathcal{G})$ 
generate a saturated $\mathfrak{S}_1$-line. 
Then we may assume $f_\lambda =1$ for some 
$1\le \lambda\le \binom{h}{h-k}$.
By Proposition \ref{saturation}, $\mathfrak{L}$ corresponds to a subgroup 
$\mathcal{H}\subseteq\mathcal{G}[p^n]$ of rank 
$k$, and by the degree of the entries of 
$\phi_{\bigwedge^{h-k}\mathfrak{M}_{\mathfrak{S}_1}(\mathcal{G})}$, 
we are guaranteed it satisfies our hypothesis as 
long as $\mathrm{deg}_u(f_i)>0$ for some $1\le i \le \binom{h}{h-k}$.

Replace 
the $\lambda$th basis element representing $\phi_{\bigwedge^{h-k}\mathfrak{M}_{\mathfrak{S}_1}(\mathcal{G})}$ by $v$ for the above 
choice of $\lambda$. Then we may perform 
the same devissage argument in Theorem \ref{EllRamComp} on the module  
$\bigwedge^{h-k}\mathfrak{M}_{\mathfrak{S}_1}(\mathcal{G})$ 
using this new basis under the assumption that 
$\mathcal{H}$ has constant $p^r$-height $k$. On applying 
Proposition \ref{deg}, this gives the bound.
\end{proof}

\subsection{Subgroups of Products}
\label{prod}
Fix $P = \prod_{i=1}^N E_i$ to be a ($p$-adic) CM algebra, and let 
$\OO_{E_i}\subseteq E_i$ be the ring of integers of the CM field 
$E_i$ with choice of uniformizer $\pi_{E_i}$. Define 
$\mathcal{G}/\OO_K$ to be a CM $p$-divisible group with 
($p$-adic) CM by $(P=\prod_i \OO_{E_i},\Phi_P =\coprod_i \Phi_i)$. 
Then $\mathcal{G}\simeq \prod_i \mathcal{G}_i$, where 
each $\mathcal{G}_i$ is a 
($p$-adic) CM $p$-divisible group of dimension $d_i$ with ($p$-adic) CM by 
$(\OO_{E_i}.\Phi_i)$ corresponding to precisely those pairs 
found in the decomposition of $(P,\Phi_P)$. 
We let $h_i = [E_i:\QQ_p]$ denote the height 
of each group and $h_P = \sum_i h_i$ denote the height of $\mathcal{G}$. 
The goal is to extend 
the differential computations in Section \ref{inert} and 
Section \ref{ram} to 
subgroups of $\mathcal{G}/\OO_K$. Note that those sections 
concern the case when $\mathcal{G} = \mathcal{G}_i$, and 
classification of their subgroups is not sufficient as not all subgroups 
of $\mathcal{G}[p^n]$ decompose to 
a product of the subgroups of $\mathcal{G}_i$.

We first demonstrate that this problem can be formulated 
up to isotypicity, and separate the cases when $E/\QQ_p$ is 
unramified and when $E/\QQ_p$ is totally ramified and Galois to do so.

\begin{prop}
\label{prodpropinert}
Let $\mathcal{G} = \prod_i\mathcal{G}_i$ be a CM $p$-divisible group 
having ($p$-adic) CM by $(\prod_{i=1}^N\OO_{E_i},\coprod_i\Phi_i)$ 
such that $(E_i,\Phi_i) = (E,\Phi)$ for all $1\le i\le N$, where $E/\QQ_p$ 
is unramified and 
$\Phi = \{\alpha_1,\ldots,\alpha_{h-d}\}\subseteq\mathrm{Hom}(E,\overline{\QQ}_p)$ is a ($p$-adic) CM type on $E$. 
Then if $\mathfrak{N}\subseteq \mathfrak{M}_{\mathfrak{S}_1}(\mathcal{G})$ 
is a simple, saturated submodule of rank $h-k$, 
there exists a saturated submodule 
$\mathfrak{N}_1\subseteq\mathfrak{M}_{\mathfrak{S}_1}(\mathcal{G}_1)$ 
of rank $h-k$ such that
$$v_u(\mathrm{det}(\phi_\mathfrak{N})) = v_u(\mathrm{det}(\phi_{\mathfrak{N}_1})).$$
\end{prop}
\begin{proof}
By Theorem \ref{GroupEquiv}, 
$\mathfrak{M}_{\mathfrak{S}_1}(\mathcal{G}) = \prod_{i=1}^N\mathfrak{M}(\mathcal{G}_i)$ 
so that $\phi_{\mathfrak{M}(\mathcal{G})} = \prod_{i=1}^N\phi_{\mathfrak{M}(\mathcal{G}_i)}$. Then 
$\mathfrak{M}(\mathcal{G})$ has a 
quasi-Kisin decomposition induced by the quasi-Kisin decomposition 
of each $\mathfrak{M}(\mathcal{G}_i)$ from Lemma \ref{quasiquasi} as
$$\mathfrak{M}^{(q)}(\mathcal{G}) = \bigoplus_{\tau\in \mathrm{Hom}(E^{\mathrm{ur}},K_0)}\mathfrak{M}_{N,\tau}^{(q)}$$
where each $\mathfrak{M}_{N,\tau}^{(q)}$ is a 
$\mathfrak{S}^{(q)}$-module of rank $N$ with $\phi^{(q)}$-Frobenius
$$\phi_{\mathfrak{M}^{(q)}_{N,\tau}}^{(q)}(f_1,\ldots,f_N) = (f_1^q,\ldots,f_N^q)\left(u^{e\sum_{s=1}^{h-d}p^{h-\tau^{-1}(\alpha_s)}} + pg_\tau(u)\right)\mathrm{Id}_N$$
for some polynomials $g_\tau(u)\in \mathfrak{S}^{(q)}$ with 
degree strictly less than $e\sum_{s=1}^{h-d}p^{h-\tau^{-1}(\alpha_s)}$ and $\mathrm{Id}_N$ 
denoting the identity matrix of rank $N$.

Let $\mathfrak{M}_{\lambda,\iota,\tau}^{(q)}\subseteq \mathfrak{M}_{N,\tau}^{(q)}$ 
be the quasi-Kisin submodule of rank $\lambda$ defined so that 
$$\phi_{\mathfrak{M}^{(q)}_{\lambda,\iota,\tau}}^{(q)}(f_1,\ldots,f_\lambda) = (f_1^q,\ldots,f_\lambda^q)\left(u^{e\sum_{s=1}^{h-d}p^{h-\tau^{-1}(\alpha_s)}} + pg_\tau(u)\right)\mathrm{Id}_{\lambda,\iota}$$
where $\mathrm{Id}_{\lambda,\iota}$ is the identity matrix of rank $\lambda$ 
together with an embedding (i.e., a specified selection of basis elements) 
$\iota:\mathrm{Id}_\lambda\hookrightarrow \mathrm{Id}_N$. Then 
by a slight modification of Corollary \ref{Fontainecor} and Proposition 
\ref{unramp}, 
every simple, saturated quasi-Kisin submodule corresponding to 
a saturated Kisin submodule 
$\mathfrak{N}\subseteq \mathfrak{M}_{\mathfrak{S}_1}(\mathcal{G})$ 
of rank $h-k$ takes the form 
$$\bigoplus_{\tau\in S_{h-k,\mathrm{max}}} \mathfrak{M}_\tau^{(q)}\xhookrightarrow{\prod_{\tau\in S_{h-k,\mathrm{max}}}\Delta_{\lambda,\iota,\tau}} \bigoplus_{\tau\in S_{h-k,\mathrm{max}}}\mathfrak{M}^{(q)}_{\lambda,\iota,\tau}\subseteq\mathfrak{M}^{(q)}(\mathcal{G})$$
for some collection of pairs $(\lambda,\iota)$, 
where $\Delta_{\lambda,\iota,\tau}:\mathfrak{M}_\tau^{(q)}\hookrightarrow \mathfrak{M}^{(q)}_{\lambda,\iota,\tau}\subseteq \mathfrak{M}^{(q)}(\mathcal{G})$ denotes the diagonal embedding and 
$S_{h-k,\mathrm{max}}\subseteq \mathrm{Hom}(E,\overline{\QQ}_p)$ 
is the subset of size $h-k$ which maximizes the sum 
$$\sum_{\tau\in S_{h-k}}\left(\sum_{s=1}^{h-d} p^{h-\tau^{-1}(\alpha_s)}\right)$$
over all subsets $S_{h-k}\subseteq \mathrm{Hom}(E,\overline{\QQ}_p)$ 
of size $h-k$.  
The statement now follows as 
the determinant for every collection of pairs $(\lambda,\iota)$ is computed by 
$\bigoplus_{\tau\in S_{h-k,\mathrm{max}}} \mathfrak{M}_\tau^{(q)}$.
\end{proof}

\begin{thm}
Let $\mathcal{G} = \mathcal{G}_1^N$ be a 
CM $p$-divisible group having ($p$-adic) CM 
by $(\OO_{E}^N,\Phi^N)$, 
where $E/\QQ_p$ is unramified of degree $h$ and 
$\Phi\subseteq\mathrm{Hom}(E,\overline{\QQ}_p)$ is a ($p$-adic) CM type 
on $E$ with $h-d$ elements. Let $\mathcal{H}\subseteq \mathcal{G}[p^n]$ 
be an irreducible subgroup with $p$-height $k<h$. Then
$$\#\frac{1}{[K:\QQ_p]}\log s^*\Omega^k_{\mathcal{H}/\OO_K} \le \left(\frac{p-1}{p^{\delta}-1}\right)\left(\frac{1-p^{-{k}}}{1-p^{-1}}\right)\left(\frac{1-p^{-n}}{1-p^{-1}}\right)\log p$$
where $\delta = \frac{h}{h - d}$.
\end{thm}
\begin{proof}
By Theorem \ref{GroupEquiv} and Proposition \ref{deg}, 
we translate the problem to computing 
the Harder-Narasimhan slope of irreducible, saturated submodules 
$\mathfrak{N}\subseteq \mathfrak{M}_{\mathfrak{S}_n}(\mathcal{G}[p^n])$. 
Performing a devissage on $\mathfrak{N}$ to write it 
as a composition series of $\mathfrak{S}_1$-modules, 
we may apply Proposition \ref{prodpropinert} 
and subsequently Theorem \ref{unrambound} for the 
relevant computation.
\end{proof}

\begin{prop}
\label{prodpropram}
Let $\mathcal{G} = \prod_i\mathcal{G}_i$ be a CM $p$-divisible group 
having ($p$-adic) CM by $(\prod_{i=1}^N\OO_{E_i},\coprod_i\Phi_i)$ 
such that $(E_i,\Phi_i) = (E,\Phi)$ for all $1\le i\le N$, where $E/\QQ_p$ 
is totally ramified and Galois and 
$\Phi \subseteq \mathrm{Hom}(E,\overline{\QQ}_p)$ 
is a ($p$-adic) CM type on $E$. 
Then if $\mathfrak{N}\subseteq \mathfrak{M}_{\mathfrak{S}_1}(\mathcal{G})$ 
is a simple, saturated submodule of rank $h-k$, 
there exists a saturated submodule 
$\mathfrak{N}_1\subseteq\mathfrak{M}_{\mathfrak{S}_1}(\mathcal{G}_1)$ 
of rank $h-k$ such that
$$v_u(\mathrm{det}(\phi_\mathfrak{N})) = v_u(\mathrm{det}(\phi_{\mathfrak{N}_1})).$$
\end{prop}
\begin{proof}
By Theorem \ref{GroupEquiv}, $\mathfrak{M}_{\mathfrak{S}_1}(\mathcal{G}) = \prod_{i=1}^N\mathfrak{M}(\mathcal{G}_i)$ 
so that $\phi_{\mathfrak{M}(\mathcal{G})} = \prod_{i=1}^N\phi_{\mathfrak{M}(\mathcal{G}_i)}$. 
Then $\phi_\mathfrak{M}(\mathcal{G})$ has a representation by a 
block upper triangular matrix where each of the blocks correspond to 
$\phi_\mathfrak{M}(\mathcal{G}_i)$ using the matrix representation 
of Example \ref{ramex}.

Saturated rank 
$h-k$ submodules of $\mathfrak{M}_{\mathfrak{S}_1}(\mathcal{G})$ 
correspond to saturated $\mathfrak{S}_1$-lines $\mathfrak{L}\subseteq \bigwedge^{h-k}\mathfrak{M}_{\mathfrak{S}_1}(\mathcal{G})$ with Frobenius
$$\phi_{\bigwedge^{h-k}\mathfrak{M}_{\mathfrak{S}_1}(\mathcal{G})} = \bigwedge^{h-k}\phi_{\mathfrak{M}_{\mathfrak{S}_1}(\mathcal{G})}.$$
Thus $\phi_{\bigwedge^{h-k}\mathfrak{M}_{\mathfrak{S}_1}(\mathcal{G})}$ 
has a matrix representation by a block upper triangular matrix where 
each block representing $\phi_{\bigwedge^{h-k}\mathfrak{M}_{\mathfrak{S}_1}(\mathcal{G}_i)}$ occurs with $N^2$ multiplicity (and there are additional blocks).
Define $\mathfrak{M}_{N^2,k}\subseteq\bigwedge^{h-k}\mathfrak{M}_{\mathfrak{S}_1}(\mathcal{G})$ 
to be the saturated $\mathfrak{S}_1$-submodule of rank $N^2$ with Frobenius
$$\phi_{\mathfrak{M}_{N^2,k}}= \prod_{i=1}^{N^2}\phi_{\bigwedge^{h-k}\mathfrak{M}(\mathcal{G}_1)}.$$

Let $\mathfrak{M}_{\lambda,\iota,k}\subseteq \mathfrak{M}_{N^2,k}$ 
be the $\mathfrak{S}_1$-submodule of rank $\lambda$ with Frobenius
$$\phi_{\mathfrak{M}_{\lambda,\iota,k}}= \prod_{i=\iota(1)}^{\iota(\lambda)}\phi_{\bigwedge^{h-k}\mathfrak{M}(\mathcal{G}_1)}$$
where $\iota:\{1,\ldots,\lambda\}\hookrightarrow \{1,\ldots,N^2\}$ 
represents the embedding 
$\mathfrak{M}_{\lambda,\iota,k}\hookrightarrow \mathfrak{M}_{N^2,k}$. 
Then by the proof of Theorem \ref{rambound}, every simple, 
saturated Kisin submodule of 
$\mathfrak{M}(\mathcal{G})$ of rank $h-k$ corresponds to a module 
of the form
$$\mathfrak{L}\subseteq \bigwedge^{h-k}\mathfrak{M}_{\mathfrak{S}_1}(\mathcal{G}_1)\xhookrightarrow{\Delta_{\lambda,\iota,k}}\mathfrak{M}_{\lambda,\iota,k}\subseteq \bigwedge^{h-k}\mathfrak{M}(\mathcal{G})$$
for some pair $(\lambda,\iota)$, where 
$\Delta_{\lambda,\iota,k}:\mathfrak{M}_{\mathfrak{S}_1}(\mathcal{G}_1)\hookrightarrow \mathfrak{M}_{\lambda,\iota,k}$ 
denotes the diagonal embedding and 
$\mathfrak{L}$ is saturated of rank $1$. The statement now follows as 
the determinant for every pair $(\lambda,\iota)$ is computed by 
$\mathfrak{L}$.
\end{proof}

\begin{thm}
Let $\mathcal{G} = \mathcal{G}_1^N$ be a 
CM $p$-divisible group having ($p$-adic) CM 
by $(\OO_{E}^N,\Phi^N)$, 
where $E/\QQ_p$ is  
totally ramified and Galois of degree $h$ and 
$\Phi\subseteq\mathrm{Hom}(E,\overline{\QQ}_p)$ is a ($p$-adic) CM type 
on $E$ with $h-d$ elements. Let $\mathcal{H}\subseteq \mathcal{G}[p^n]$ 
be an irreducible subgroup with $p$-height $k<h$ 
such that 
$\mathcal{H}\cap \prod'_i\mathcal{G}_1[\pi_{E}^{r_i}] = 1$ for all 
$r_i\ge 0$ and subproducts 
$\prod'\mathcal{G}_1[\pi_{E}^{r_i}]\subseteq \mathcal{G}_1[\pi_{E}^{ne}]^N$. Then
$$\#\frac{1}{[K:\QQ_p]}\log s^*\Omega^k_{\mathcal{H}/\OO_K} <k\left(\frac{h-d}{h}\right)\left(\frac{1-p^{-n}}{1-p^{-1}}\right)\log p.$$
\end{thm}
\begin{proof}
By Theorem \ref{GroupEquiv} and Proposition \ref{deg}, we translate the 
problem to computing the Harder-Narasimhan slope of irreducible, saturated 
submodules $\mathfrak{N}\subseteq \mathfrak{M}_{\mathfrak{S}_n}(\mathcal{G}[p^n])$. 
Performing a devissage on $\mathfrak{N}$ to write it as a 
composition series of $\mathfrak{S}_1$-modules, we may apply 
Proposition \ref{prodpropram} and subsequently Theorem \ref{rambound} 
for the relevant computation.
\end{proof}

Now we consider more generally when the $E_i$ in the 
decomposition of $P$ are allowed to either be 
unramified or totally ramified and Galois. We further let $I_{\mathrm{ur},k}$, 
resp. $I_{\mathrm{ram},k}$, denote the set of indices $i$ for which 
$E_i/\QQ_p$ are unramified, resp. ramified, and 
$[E_i:\QQ_p] > h-k$.

\begin{prop}
\label{prodpropmix}
Let $\mathcal{G} = \prod_i\mathcal{G}_i$ be a CM $p$-divisible group 
having ($p$-adic) CM by $(\prod_{i=1}^N\OO_{E_i},\coprod_i\Phi_i)$ 
such that each pair $(E_i,\Phi_i)$ is distinct and $E_i/\QQ_p$ is either 
unramified or totally ramified and Galois. 
Then if $\mathfrak{N}\subseteq \mathfrak{M}_{\mathfrak{S}_1}(\mathcal{G})$ 
is a simple, saturated submodule of rank $h-k$, 
there exists a saturated submodule 
$\mathfrak{N}_i\subseteq\mathfrak{M}_{\mathfrak{S}_1}(\mathcal{G}_i)$ 
of rank $h-k$ for some $i$ such that
$$v_u(\mathrm{det}(\phi_\mathfrak{N})) = v_u(\mathrm{det}(\phi_{\mathfrak{N}_i})).$$
\end{prop}
\begin{proof}
By Theorem \ref{GroupEquiv}, 
$\mathfrak{M}_{\mathfrak{S}_1}(\mathcal{G}) = \prod_{i=1}^N\mathfrak{M}(\mathcal{G}_i)$ 
so that $\phi_{\mathfrak{M}(\mathcal{G})} = \prod_{i=1}^N\phi_{\mathfrak{M}(\mathcal{G}_i)}$. 
Then $\phi_{\mathfrak{M}(\mathcal{G})}$ has a representation of a block 
diagonal matrix. We may assume that $[E_i:\QQ_p]>h-k$ for 
every index $i$ and write by $I_{\mathrm{ur}}$ and $I_{\mathrm{ram}}$ 
the set of indices for which 

Suppose first that 
$I_{\mathrm{ram}}$ is empty. Then 
for $q = \mathrm{lcm}_{i\in I_{\mathrm{ur}}}\{p^{h_i}\}$, 
$\mathfrak{M}(\mathcal{G})$ has a 
quasi-Kisin decomposition induced by the quasi-Kisin decomposition of 
each $\mathfrak{M}(\mathcal{G}_i)$ for $i\in I_{\mathrm{ur}}$ from 
Lemma \ref{quasiquasi} as
$$\mathfrak{M}^{(q)}(\mathcal{G}) = \bigoplus_{i\in I_{\mathrm{ur}}}\bigoplus_{\tau\in \mathrm{Hom}(E_i,K_0)}\mathfrak{M}_{i,\tau}^{(q)}(\mathcal{G}_i).$$
Whenever $\mathcal{G}_i$ is not a Serre 
tensor construction of $\mathcal{G}_j$ for any distinct pair of indices $i$ and $j$, 
the slopes in the decomposition are all unique, and a modification of 
Corollary \ref{Fontainecor} gives $\mathfrak{N} = \mathfrak{N}_i\subseteq \mathfrak{M}(\mathcal{G}_i)$ for some saturated submodule. 
Otherwise, the method of considering 
diagonal submodules for Serre tensor constructions works 
identically to isotypicity in Proposition \ref{prodpropinert}. 
Then $\mathfrak{N}_i$ is the projection to a factor in a diagonal 
embedding if necessary.

If $I_{\mathrm{ur}}$ is empty, then $\mathfrak{N}$ corresponds to 
a saturated line in $\bigwedge^{h-k}\mathfrak{M}_{\mathfrak{S}_1}(\mathcal{G})$ 
with Frobenius
$$\phi_{\bigwedge^{h-k}\mathfrak{M}_{\mathfrak{S}_1(\mathcal{G})}} = \bigwedge^{h-k}\phi_{\mathfrak{M}_{\mathfrak{S}_1}(\mathcal{G})}.$$
Thus $\phi_{\bigwedge^{h-k}\mathfrak{M}_{\mathfrak{S}_1}(\mathcal{G})}$
has a matrix representation by a block upper triangular matrix. 
Again whenever $\mathcal{G}_i$ is not a Serre 
tensor construction of $\mathcal{G}_j$ for any distinct pair of indices $i$ and $j$, 
all potential are unique and $\mathfrak{N} = \mathfrak{N}_i$ 
for some $i$. Otherwise, the method of considering diagonal 
submodules for Serre tensor constructions works identically to 
isotypicity in Proposition \ref{prodpropram}. Then $\mathfrak{N}_i$ 
is the projection to a factor in a diagonal embedding if necessary.

In general, the valuation of $\mathfrak{N}$ can be 
computed by one of these two instances.
\end{proof}

\begin{thm}
\label{projection}
Let $\mathcal{G}\simeq \prod_{i=1}^N\mathcal{G}_i$ where 
$\mathcal{G}$ has ($p$-adic) CM by $(E,\Phi)$ and 
$\mathcal{G}_i$ are each $\OO_{E_i}$-linear CM $p$-divisible groups 
of type $(\OO_{E_i},\Phi_i)$, where each $E_i$ is either 
unramified or totally ramified and Galois. Then for any simple subgroup 
$\mathcal{H}\subseteq \mathcal{G}[p^n]$ of $p$-height $k<h$, 
such that $\mathcal{H}\cap \prod'\mathcal{G}_{i}[\pi_{E_i}^{r_i}] = 0$ for all $1\le r_i\le n$ and subproducts 
$\prod'\mathcal{G}_{i}[\pi_{E_i}^{r_i}]\subseteq 
\prod_i\mathcal{G}_{i}[\pi_{E_i}^{r_i}]$, 
$$\#\frac{1}{[K:\QQ_p]}\log s^*\Omega^k_{\mathcal{H}/\OO_K} = \#\frac{1}{[K:\QQ_p]}\log s^*\Omega^k_{\mathcal{H}_i/\OO_K}.$$
\end{thm}
\begin{proof}
We may assume that each 
each pair $(E_i,\Phi)$ in the decomposition 
of the ($p$-adic) CM pair $(P,\Phi_P)$ is distinct 
(so the fields $E_i$ can repeat, but then their CM types must differ) 
by Proposition \ref{prodpropinert} and Proposition \ref{prodpropram}. 
Then Proposition \ref{prodpropmix}, there is a saturated submodule 
$\mathfrak{N}_i\subseteq \mathfrak{M}_{\mathfrak{S}_1}(\mathcal{G}_i)$ 
of rank $h_i-k$ by which we can compute the determinant. 
This module corresponds to a subgroup $\mathcal{H}_i\subseteq \mathcal{G}_i$ 
for some $i$ by Theorem \ref{GroupEquiv}, by which we can conclude.
\end{proof}

\section{Main Theorems}

\subsection{CM Northcott Property}
\label{general}

We prove the two versions of the CM Northcott property stated as 
Theorem \ref{CMIsog} and Theorem \ref{CMNorth} in the introduction.

Fix a CM field $E$ with $[E:\QQ] = 2g$ and 
maximal totally real subfield $F\subseteq E$. Denote by 
$\OO_E\subseteq E$ the corresponding ring of integers and 
$\OO_n\subseteq E$ an order of index $n$ with 
no $\OO_E$-stable submodule. 
Then denote by $A$ and $A_n$ abelian varieties defined 
over a number field $K$ 
of dimension $g$ having, respectively, 
CM by $(\OO_E,\Phi)$ and $(\OO_n,\Phi)$ and by 
$\mathcal{A}$ and $\mathcal{A}_n$ their corresponding 
N\'eron models over $\OO_K$ 
Then one may construct an isogeny $\phi:\mathcal{A}\to \mathcal{A}_n$ 
over $\OO_K$ whose kernel is a finite flat group scheme $\mathcal{G}_n$ 
of size $n$.

\begin{lem}
\label{combinatorics}
Let $A$ and $A_{p^r}$ be abelian varieties over a 
number field $K$ as above, and let $\phi:A\to A_{p^{r}}$ be 
an isogeny of degree $p^r$ over $K$ whose kernel is $p^n$-torsion
(so $r\ge n$). Assume moreover that $p$ factors in $E$ as 
$p = \prod_i \mathfrak{p}^{\nu_i}\subseteq E$ and 
$\mathrm{Nm}(\mathfrak{p}_i) = p^{f_i}$.

\begin{enumerate}
\item If $p$ splits completely in $E$, i.e., $\nu_i = 1$ and $f_i=1$ 
for each $i$, then
$$h_{\mathrm{Fal}}(A_{p^r}) - h_{\mathrm{Fal}}(A) = 
\frac{r}{2}\log p.$$

\item\label{2} If all primes above $p$ are unramified along the CM extension 
$E/F$ and otherwise we allow $\nu_i\ge 1$ and $f_i\ge 1$ for each $i$,
$$h_{\mathrm{Fal}}(A_{p^r}) - h_{\mathrm{Fal}}(A) \ge 
\left[\frac{r}{2} -\sum_{i: f_i>1} \nu_i\left(\frac{p-1}{p^{f_i}-1}\right)\left(\frac{1-p^{-k_i}}{1-p^{-1}}\right)\left(\frac{1-p^{-n_i}}{1-p^{-1}}\right)\right]\log p$$
where $n = \sum_i n_i$ and $k = \sum_i k_i$ are determined by 
the decomposition $A[p^n] \simeq \prod_i A[\mathfrak{p}_i^{\nu_in}]$.

\item\label{3} If all primes above $p$ are ramified along 
the CM extension $E/F$ and otherwise we allow $\nu_i\ge 1$ and $f_i\ge 1$, 
and we define the Galois closure 
$\widetilde{E}_{\mathfrak{p}_i}$ of $E_{\mathfrak{p}_i}$ to have 
relative degree $\rho_i = [\widetilde{E}_{\mathfrak{p}_i}:E_{\mathfrak{p}_i}]$, 
then
$$h_{\mathrm{Fal}}(A_{p^r}) - h_{\mathrm{Fal}}(A) > 
\left[\frac{r}{2} -\sum_i \frac{1}{\rho_i}\left(\frac{p-1}{p^{f_i}-1}\right)\left(\frac{h_i-d_i}{h_i}\right)\left(\frac{1-p^{-n_i}}{1-p^{-1}}\right)\right]\log p$$
where $n = \sum_i n_i$, $g = \sum_i d_i$, and $h = \sum_i h_i$ 
are determined by 
the decomposition $A[p^n] \simeq \prod_i A[\mathfrak{p}_i^{\nu_in}]$.

\item For a general prime $p$, let $I_2$ (resp., $I_3$) denote 
the set of indices $i$ such that $\mathfrak{p}_i$ lies over a 
prime that is unramified 
(resp., ramified) along the CM extension $E/F$. Then, retaining 
the notation from cases (2) and (3), 
\begin{align*}
h_{\mathrm{Fal}}(A_{p^r}) - h_{\mathrm{Fal}}(A) &> \frac{m}{2} -
\left[\sum_{i\in I_2: f_i>1} \nu_i\left(\frac{p-1}{p^{f_i}-1}\right)\left(\frac{1-p^{-k_i}}{1-p^{-1}}\right)\left(\frac{1-p^{-n_i}}{1-p^{-1}}\right)\right.\\
&\quad\quad\quad\quad\quad\quad\left.+\sum_{i\in I_3} \frac{1}{\rho_i}\left(\frac{p-1}{p^{f_i}-1}\right)\left(\frac{h_i-d_i}{h_i}\right)\left(\frac{1-p^{-n_i}}{1-p^{-1}}\right)\right]\log p\end{align*}
where the terms are determined by the decomposition $A[p^n] \simeq \prod_i A[\mathfrak{p}_i^{\nu_in}]$ as in statements (2) and (3).

\end{enumerate}

\end{lem}
\begin{proof}
After a reduction to a local problem, the proof is an 
exercise in piecing together the decomposition 
theorems in Section 6 and the 
computations in Section 7. First, $\phi$ decomposes along the 
factorization $p=\prod_i\mathfrak{p}_i^{\nu_i}$ into 
a composition of isogenies $\phi_i$ each of degree $p^{f_i\nu_i}$. Since the 
$\mathfrak{p}_i$ in this decomposition are relatively prime, there 
is no canonical sequence for this decomposition, so the order may be 
permuted arbitrarily and we get a composition series for the finite flat group 
scheme kernel of $\phi$. We lose nothing in this computation to 
assume that the kernel of $\phi$ splits along the decomposition of 
$p$, otherwise apply Theorem \ref{projection}, which allows us to project 
along this factorization and compute based on an 
appropriate projection. 
Then as Theorem \ref{FaltingsIsog} 
is additive when evaluated on exact sequences of finite flat group schemes, 
we may compute everything for finite flat subgroup schemes of 
each of the $A[\mathfrak{p}_i^{\nu_in}]$.

Suppose first that $p$ splits completely as in statement $(1)$. 
Then by Lemma \ref{correctlem}, the isogeny $\phi$ must be 
\'etale for otherwise we contradict our hypothesis on $A_{p^n}$. 
Thus, the relative Hodge bundle is trivial and has no contribution to 
the variation of the Faltings height.

Assume now that no prime above $p$ ramifies 
in the CM extension $E/F$ as in 
statement $(2)$. Then the relative Hodge bundle 
at each $\phi_i$ may be computed by the general 
formula in Corollary \ref{HodgeUnramType}, which demonstrates how to
incorporate the ramification degree, and the bound in 
Theorem \ref{unrambound}.

Now assume that all primes above $p$ do ramify 
in the CM extension $E/F$ as in statement $(3)$. Then the 
relative Hodge bundle at each $\phi_i$ may be computed by 
the general formula in Corollary \ref{ram22}, which demonstrates how 
to incorporate the degree of the Galois extension, Corollary \ref{ram11}, 
which incorporates the inertia degree, and the bound in 
Theorem \ref{rambound}.

For a general $p$ as in statement $(4)$ we combine preceding cases.
\end{proof}

We prove first Theorem \ref{CMIsog}.

\begin{thm}[CM Northcott Property for Isogeny Classes]
\label{NorthcottIsog}
Let $C$ be a fixed positive constant and $g\ge 1$ be a fixed integer. Then 
the number of isomorphism classes of abelian varieties $A$ of dimension $g$ 
with $h_{\mathrm{Fal}}(A)<C$ within an isogeny class is finite.
\end{thm}
\begin{proof}
By the Poincar\'e reducibility theorem, every CM abelian variety is isogenous 
to one that is a product of simple abelian varieties each with CM by 
$(\OO_{E_i},\Phi_i)$ for some CM field $E_i/\QQ$. Therefore we 
fix $A\simeq \prod_i A_i$ where each factor $A_i$ is simple with 
CM by $(\OO_{E_i},\Phi_i)$. By Lemma \ref{quentin}, the Faltings 
height variation is zero precisely when the group scheme 
kernel of the isogeny is $\prod_i \OO_{E_i}$-stable, i.e., when the isogeny 
is between two CM abelian varieties both with CM by 
$(\prod_i \OO_{E_i},\prod_i \Phi_i)$. Since there are finitely 
many such isomorphism classes of CM abelian varieties by a theorem of 
Brauer-Siegel, the theorem reduces to counting pairs $(A,\mathcal{G}_n)$ where 
$\mathcal{G}_n\subseteq \mathcal{A}$ is a finite flat subgroup scheme 
which is not $\OO_E$-stable and has prescribed order $n$.

The computations in Sections \ref{inert}, \ref{ram}, and \ref{prod} demonstrate 
the theorem. Indeed, since the Hodge bundle is additive on short exact sequences 
and the isogeny decomposes along the primary decomposition of $n$ in 
$\prod_i E_i$, we may always assume that $n = p^s$ for some prime $p$ 
and $s\ge 1$. Lemma \ref{combinatorics}, which may be strengthened 
to all subgroups products of simple abelian varieties by Theorem \ref{projection}, 
summarizes the computations within each $p$-isogeny class and 
allows us to conclude.
\end{proof}

We now prove the (conditional) Theorem \ref{CMNorth}.

\begin{thm}[CM Northcott Property]
\label{Northcott}
Let $C$ be a fixed positive constant and $g\ge 1$ be a fixed integer. Then 
assuming Conjecture \ref{colmez} and the Artin Conjecture, the number of 
CM abelian varieties of dimension $g$ with $h_{\mathrm{Fal}}(A)<C$ is finite.
\end{thm}

\begin{proof}
Theorem \ref{NorthcottIsog} demonstrates finiteness within an isogeny class. 
By the Poincar\'e reducibility theorem, every isogeny class contains an abelian 
variety $A\simeq \prod_i A_i$ unique up to isomorphism 
where each factor $A_i$ is simple with CM by 
$(\OO_{E_i},\Phi_i)$. Since the Faltings height is additive on products of 
abelian varieties, we conclude by Corollary \ref{primitive}, which relies 
on the conjectures of Colmez and Artin.
\end{proof}

\begin{rem}
When we consider only the category of \emph{simple} CM abelian 
varieties, we can describe this growth in terms of the \emph{conductor} 
of the order. We recall that the conductor is the smallest ideal that is 
an ideal of both $\OO_E$ and $\OO$. In particular, the conductor always 
divides the degree, so to count along the conductor, we can count 
along those finite flat group schemes whose order divides the conductor with 
the smallest HN slope.
\end{rem}

\subsection{Colmez-Type Formulae}
\label{lastsect}
\label{RM}
We make the computations precise to give a few intrinsic cases of 
Colmez-type formulae for abelian varieties of low dimension. Our 
first result is to recover the formula of Nakkajima-Taguchi \cite{japanese} in the case of 
elliptic curves. Here we again are in a global setting, and will let 
$E$ denote a CM field with ring of integers $\OO_E$ and order 
$\OO_n\subseteq \OO_E$ of index $n$ such that $\OO_E/\OO_n$ 
has no $\OO_E$-stable fixed piece.

\begin{thm}[Colmez Formula for CM Elliptic Curves]
\label{ColmezEll}
Let $A_n$ be an elliptic curve which has CM by $(\OO_n,\Phi)$ 
and let $A$ be an elliptic curve which has CM by $(\OO_E,\Phi)$. Then
$$h_{\mathrm{Fal}}(A_n) = h_\mathrm{Fal}(A) + \frac{1}{2}\left(\sum_{p|n}\log p\left(r_p - \left(\frac{1-\left(\frac{p}{d}\right)}{p-\left(\frac{p}{d}\right)}\right)\left(\frac{1-p^{-r_p}}{1-p^{-1}}\right)\right)\right)$$
where $r_p$ is the largest exponent such that $p^{r_p}|n$. 

Combining this with Conjecture \ref{colmez} (a theorem 
in the case of elliptic curves), we obtain the 
formula
\begin{align*}
h_{\mathrm{Fal}}(A_n) &=  -\frac{L'(\chi_{-d},0)}{L(\chi_{-d},0)} - \frac{1}{2}\log d + \frac{1}{2}\log 2\pi  \\
&\quad\,+\frac{1}{2}\left(\sum_{p|n}\log p\left(r_p - \left(\frac{1-\left(\frac{p}{d}\right)}{p-\left(\frac{p}{d}\right)}\right)\left(\frac{1-p^{-r_p}}{1-p^{-1}}\right)\right)\right)\end{align*}
where $d$ is the discriminant of $E$ and $\chi_{-d}$ the 
corresponding quadratic Artin character.
\end{thm}
\begin{proof}
By the same argument as in Theorem \ref{Northcott}, we may 
assume that $n = p^{r_p}$ and study the problem locally. 
Since $A$ is an elliptic curve, the trichotomy of the splitting behavior 
of $p$ in the CM field $E$ (encoded by the quadratic residue character) 
determines the structure of the finite flat group scheme kernel 
for the minimal degree isogeny between $A$ and $A_n$, and 
therefore the Faltings height. 

When $p$ splits, Deuring's theorem shows that the isogeny is 
necessarily \'etale and the 
only contribution to the Faltings height variation is the degree of the 
isogeny. When $p$ is inert, the differential in this variation calculation 
attains the maximum along the bound in Theorem \ref{unrambound} with 
$k = 1$ and $\delta = 2$, since $A$ is an elliptic curve and, up to 
complex conjugation, admits a unique CM type. When $p$ is 
ramified the computation has already been done in 
Theorem \ref{EllRamComp}.
\end{proof}

We extend this formula to simple abelian surfaces for which $E/\QQ$ is 
Galois. We note that by this hypothesis 
$E/\QQ$ must be cyclic, and 
there can even only be one CM type (up to complex conjugation) on these abelian 
surfaces, given by $\Phi = \{1,2\}$. These assertions 
are checked explicitly in \cite{HabPaz}. The local computations 
for this theorem are found in Section \ref{ramcompus} and Section \ref{9}.

\begin{thm}[Colmez Formula for Cyclic CM Abelian Surfaces]
\label{LASTTHING}
Let $A$ be a simple CM abelian surface over a number 
field $K$ with CM by $(\OO_E,\Phi)$ 
where $E/\QQ$ is a cyclic, Galois CM field of degree $4$ 
and $F\subseteq E$ its totally real subfield. Let 
$A_{p^n}/K$ be a simple CM abelian surface with CM by $(\OO,\Phi)$ 
where $\OO\subseteq \OO_E$ is a non-maximal order with conductor 
divisible by $p$, and an isogeny 
$\phi:A\to A_{p^n}$ has minimal degree $p^n$. There are nine cases:

\begin{enumerate}
\item If $p = \mathfrak{p}_1\mathfrak{p}_2\mathfrak{p}_3\mathfrak{p}_4$ 
is totally split, then
$$h_{\mathrm{Fal}}(A_{p^n}) = h_{\mathrm{Fal}}(A) + \frac{n}{2}\log p.$$
\item If $p = \mathfrak{p}_1\mathfrak{p}_2$ such that $p$ remains inert in $F/\QQ$ 
and the prime above $p$ splits in $E/F$, 
then
$$h_{\mathrm{Fal}}(A_{p^n}) = h_{\mathrm{Fal}}(A) + \frac{n}{2}\log p.$$
\item $p = \mathfrak{p}_1^{2}\mathfrak{p}_2^{2}$ such that $p$ ramifies in $F/\QQ$ and the prime above $p$ splits in $E/F$, then
$$h_{\mathrm{Fal}}(A_{p^n}) = h_{\mathrm{Fal}}(A) + \frac{n}{2}\log p.$$
\item If $p = \mathfrak{p}_1\mathfrak{p}_2$ such that $p$ splits in $F/\QQ$ and 
each $\mathfrak{p}_i$ remains inert in $E/F$, 
$$h_{\mathrm{Fal}}(A_{p^n}) = h_{\mathrm{Fal}}(A) + \left[\frac{n}{2} - \frac{1}{p+1}\left(\frac{1-p^{-n}}{1-p^{-1}}\right)\right]\log p.$$
\item If $p = \mathfrak{p}_1^2\mathfrak{p}_2^{2}$ such that $p$ splits in $F/\QQ$ 
and each $\mathfrak{p}_i$ ramifies in $E/F$, 
then
$$h_{\mathrm{Fal}}(A_{p^n}) = h_{\mathrm{Fal}}(A) + \left[\frac{n}{2} -  \frac{1}{2p}\left(\frac{1-p^{-n}}{1-p^{-1}}\right)\right]\log p.$$
\item If $p = \mathfrak{p}^2$ such that $p$ ramifies in $F/\QQ$ 
and $\mathfrak{p}$ remains inert in $E/F$,
$$h_{\mathrm{Fal}}(A_{p^n}) = h_{\mathrm{Fal}}(A) + \left[\frac{n}{2} - \frac{2}{p+1}\left(\frac{1-p^{-n}}{1-p^{-1}}\right)\right]\log p.$$
\item If $p = \mathfrak{p}^2$ such that $p$ remains inert in $F/\QQ$ and 
the prime above $p$ ramifies in $E/F$,
$$h_{\mathrm{Fal}}(A_{p^n}) = h_{\mathrm{Fal}}(A) + \left[\frac{n}{2} -  \frac{1}{p(p+1)}\left(\frac{1-p^{-n}}{1-p^{-1}}\right)\right]\log p.$$
\item If $p = \mathfrak{p}$ is inert, let 
$(\lambda_1,\lambda_2,\lambda_3)$ be a non-increasing tuple 
characterizing $\mathcal{G}$ satisfying
$\lambda_1 + \lambda_2 + \lambda_3 = n$, 
$\mathcal{G}$ is $p^{\lambda_1}$-torsion, and the $p$-height 
$\iota\in \{1,2,3\}$ of $\mathcal{G}$ is the largest integer 
such that $\lambda_{\iota}\neq 0$. 
Then
$$h_{\mathrm{Fal}}(A_{p^n}) = h_{\mathrm{Fal}}(A) + \left[\frac{n}{2} - \frac{p-1}{p^4-1}\left(\sum_{i=\lambda_2+1}^{\lambda_1}\frac{p + 1}{p^i}\right)\left(\sum_{j=\lambda_3 +1}^{\lambda_2}\frac{(p+1)^2}{p^j}\right)\left(\sum_{k = 1}^{\lambda_3}\frac{(p^2 + 2)(p+1)}{p^k}\right)\right]\log p$$
where the sums are evaluated if and only if the difference in bounds is at 
least zero.
\item If $p = \mathfrak{p}^4$ is totally ramified, 
let 
$(\lambda_1,\lambda_2,\lambda_3)$ be a non-increasing tuple 
characterizing $\mathcal{G}$ satisfying
$\lambda_1 + \lambda_2 + \lambda_3 = n$, 
$\mathcal{G}$ is $p^{\lambda_1}$-torsion, and the $p$-height 
$\iota\in \{1,2,3\}$ of $\mathcal{G}$ is the largest integer 
such that $\lambda_{\iota}\neq 0$. 
Then
$$h_{\mathrm{Fal}}(A_{p^n}) = h_{\mathrm{Fal}}(A) + \left[\frac{n}{2} - \frac{1}{4(p^4-p^3)}\left(\sum_{i=\lambda_2 + 1}^{\lambda_1}\frac{R_1(p)}{p^i}\right)\left(\sum_{j=\lambda_3+1}^{\lambda_2}\frac{R_2(p)}{p^j}\right)\left(\sum_{k=1}^{\lambda_3}\frac{R_3(p)}{p^k}\right)\right]\log p$$
where the sums are evaluated if and only if the difference in bounds 
is at least zero, and $R_i(p)$ are defined in Theorem \ref{ramcyc4fin}.
\end{enumerate}
\end{thm}
\begin{proof}
We will denote by $\mathcal{A}/\OO_K$ the N\'eron model of $A/K$, and, 
for any prime $\mathfrak{p}\subseteq \mathcal{O}_E$, 
$\mathcal{G}_\mathfrak{p} = \mathcal{A}[\mathfrak{p}^\infty]$ 
the corresponding CM $p$-divisible group with ($p$-adic) CM type  
$(E_{\mathfrak{p}},\Phi_{\mathfrak{p}})$. We will let 
$\mathcal{G}$ denote the finite flat group scheme kernel of 
$\phi : \mathcal{A}\to \mathcal{A}_{p^n}$ and 
$\omega_\mathcal{G} = s^{*}\Omega^1_{\mathcal{G}/\OO_K}$ denote 
the relative Hodge bundle. By Theorem \ref{FaltingsIsog}, 
we then obtain the change in Faltings height as a 
computation of $\omega_{\mathcal{G}}$.

In (1), $\mathcal{A}[p^\infty]\simeq \mathcal{G}_{\mathfrak{p}_1}\times \mathcal{G}_{\mathfrak{p}_2}\times \mathcal{G}_{\mathfrak{p}_3}\times \mathcal{G}_{\mathfrak{p}_4}$ and $\mathcal{G}_{\mathfrak{p}_1}\simeq \mathcal{G}_{\mathfrak{p}_2}^\vee$ and $\mathcal{G}_{\mathfrak{p}_3}\simeq \mathcal{G}_{\mathfrak{p}_4}^\vee$. Since $p$ splits in $E/F$, we 
may suppose $\mathcal{G}_{\mathfrak{p}_1}$ and $\mathcal{G}_{\mathfrak{p}_3}$ are CM $p$-divisible groups with $d = 0$ and $h = 1$, and 
$\mathcal{G}_{\mathfrak{p}_2}$ and $\mathcal{G}_{\mathfrak{p}_4}$ 
are CM $p$-divisible groups with $d = 1$ and $h = 1$. 
Then $\mathcal{G}$ must be a subgroup of 
$\mathcal{G}_{\mathfrak{p}_1}[p^n]\times \mathcal{G}_{\mathfrak{p}_3}[p^n]$, 
otherwise by Lemma \ref{correctlem} $A_{p^n}$ 
cannot have CM by an order whose conductor 
divides $p$. Since $\mathcal{G}_{\mathfrak{p}_1}$ 
and $\mathcal{G}_{\mathfrak{p}_3}$ are \'etale, $\omega_{\mathcal{G}} = 0$.

In (2), $\mathcal{A}[p^\infty] \simeq \mathcal{G}_{\mathfrak{p}_1}\times\mathcal{G}_{\mathfrak{p}_2}$ and 
$\mathcal{G}_{\mathfrak{p}_1}\simeq \mathcal{G}_{\mathfrak{p}_2}^\vee$. 
Since $p$ is inert in $F/\QQ$, $E_{\mathfrak{p}_i}/\QQ_p$ is unramified 
for each $\mathfrak{p}_i$. 
However, since the prime above $p$ splits in $E/F$, we may suppose 
$\mathcal{G}_{\mathfrak{p}_1}$ is a CM $p$-divisible group with 
$d = 0$ and $h = 2$, and $\mathcal{G}_{\mathfrak{p}_2}$ is a 
CM $p$-divisible group with $d = 2$ and $h = 2$. 
Then $\mathcal{G}$ must be a subgroup of 
$\mathcal{G}_{\mathfrak{p}_1}[p^n]$, 
otherwise by Lemma \ref{correctlem} $A_{p^n}$ 
cannot have CM by an order whose conductor 
divides $p$. Since $\mathcal{G}_{\mathfrak{p}_1}$ 
is \'etale, $\omega_{\mathcal{G}} = 0$.

In (3), $\mathcal{A}[p^\infty] \simeq \mathcal{G}_{\mathfrak{p}_1}\times\mathcal{G}_{\mathfrak{p}_2}$ and 
$\mathcal{G}_{\mathfrak{p}_1}\simeq \mathcal{G}_{\mathfrak{p}_2}^\vee$. 
Since $p$ is ramified in 
$F/\QQ$, $E_{\mathfrak{p}_i}/\QQ_p$ is totally ramified and Galois 
for each $\mathfrak{p}_i$. 
However, since the prime above $p$ splits in $E/F$, we may suppose 
$\mathcal{G}_{\mathfrak{p}_1}$ is a CM $p$-divisible group with 
$d = 0$ and $h = 2$, and $\mathcal{G}_{\mathfrak{p}_2}$ is a 
CM $p$-divisible group with $d = 2$ and $h = 2$. 
Then $\mathcal{G}$ must be a subgroup of 
$\mathcal{G}_{\mathfrak{p}_1}[p^n]$, 
otherwise by Lemma \ref{correctlem} $A_{p^n}$ 
cannot have CM by an order whose conductor 
divides $p$. Since $\mathcal{G}_{\mathfrak{p}_1}$ 
is \'etale, $\omega_{\mathcal{G}} = 0$.

In (4), $\mathcal{A}[p^\infty] \simeq \mathcal{G}_{\mathfrak{p}_1}\times\mathcal{G}_{\mathfrak{p}_2}$ 
and $\mathcal{G}_{\mathfrak{p}_1}\simeq \mathcal{G}_{\mathfrak{p}_2}$. 
Since the primes above $p$ are inert in $E/F$, 
$E_{\mathfrak{p}_i}/\QQ_p$ is unramified for each $\mathfrak{p}_i$, 
and $\mathcal{G}_{\mathfrak{p}_i}$ are each CM 
$p$-divisible groups 
with $d=1$ and $h=2$. By Proposition \ref{prodpropinert}, 
we lose nothing to assume  
$\mathcal{G}\subseteq \mathcal{G}_{\mathfrak{p}_1}[p^n]$ 
in order to compute $\omega_{\mathcal{G}}$, 
and thus may directly apply Theorem \ref{ColmezEll}.

In (5), $\mathcal{A}[p^\infty] \simeq \mathcal{G}_{\mathfrak{p}_1}\times\mathcal{G}_{\mathfrak{p}_2}$ and 
$\mathcal{G}_{\mathfrak{p}_1}\simeq \mathcal{G}_{\mathfrak{p}_2}$. 
Since the primes above $p$ are ramified in 
$E/F$, $E_{\mathfrak{p}_i}/\QQ_p$ is totally ramified and Galois 
for each $\mathfrak{p}_i$, and 
$\mathcal{G}_{\mathfrak{p}_i}$ are each CM 
$p$-divisible groups with $d=1$ and $h=2$. By Proposition \ref{prodpropram}, 
we lose nothing to assume 
$\mathcal{G}\subseteq \mathcal{G}_{\mathfrak{p}_1}[p^n]$, 
and thus may directly apply Theorem \ref{ColmezEll}.

In (6), $\mathcal{A}[p^\infty]\simeq \mathcal{G}_{\mathfrak{p}}$. 
Since the prime above $p$ is inert in $E/F$ and 
$p$ is ramified in $F/\QQ$, there exists 
a degree $2$ unramified subfield $E^{\mathrm{ur}}\subseteq E_{\mathfrak{p}}$  
such that $\mathcal{G}_{\mathfrak{p}}\simeq \mathcal{G}^{\mathrm{ur}}_{\mathfrak{p}}\otimes_{\OO_{E^{\mathrm{ur}}}}\OO_{E_{\mathfrak{p}}}$ and $\mathcal{G}^{\mathrm{ur}}_{\mathfrak{p}}$ is a CM $p$-divisible 
group with $d = 1$ and $h = 2$. Then by Corollary \ref{HodgeUnramType}, 
we lose nothing to study subgroups of $\mathcal{G}^{\mathrm{ur}}_{\mathfrak{p}}$ 
in order to compute $\omega_\mathcal{G}$, and thus may apply 
Theorem \ref{ColmezEll} in conjunction with Corollary \ref{HodgeUnramType}.

In (7), $\mathcal{A}[p^\infty]\simeq \mathcal{G}_{\mathfrak{p}}$. 
Since $p$ is ramified in $E/F$ and inert in $F/\QQ$, by Corollary \ref{ram11} 
we lose nothing to study subgroups of a CM $p$-divisible group 
$\mathcal{G}'_\mathfrak{p}$ with $d = 1$ and $h=2$ having ($p$-adic) CM 
by a totally ramified degree $2$ extension of $\QQ_p$, as by 
Corollary \ref{Fontainecor} we may recover this information 
to compute $\omega_\mathcal{G}$. We thus apply Theorem \ref{ColmezEll} 
to study the subgroups of $\mathcal{G}'_p[p^n]$. 

In (8), $\mathcal{A}[p^\infty]\simeq \mathcal{G}_{\mathfrak{p}}$. 
Since $p$ remains inert in $E/\QQ$, $E_{\mathfrak{p}}/\QQ_p$ is unramified 
and $\mathcal{G}_{\mathfrak{p}}$ is a CM $p$-divisible group with 
$d = 2$ and $h = 4$. By the assumption that $E/\QQ$ is cyclic, 
we may further deduce that $\Phi = \{1,2\}$ as the decomposition 
group embeds into the Galois group, and the unique involutive element 
of $\mathrm{Gal}(E/\QQ)$ must conjugate $\Phi$ to the opposite 
set of embeddings in $\mathrm{Hom}(E,\overline{\QQ})$. Then we can directly 
apply Proposition \ref{formulaunram} 
using $(E_\mathfrak{p},\Phi_\mathfrak{p} = \{1,2\})$.

In (9), $\mathcal{A}[p^\infty]\simeq \mathcal{G}_{\mathfrak{p}}$. 
Since $p$ is totally ramified in $E$, $E_{\mathfrak{p}}/\QQ_p$ is totally 
ramified and Galois, 
and $\mathcal{G}_{\mathfrak{p}}$ is a CM $p$-divisible group with 
$d = 2$ and $h = 4$. By the assumption that $E/\QQ$ is cyclic, 
we may further deduce that $\Phi_{\mathfrak{p}} = \{1,2\}$. 
Then this computation is done in Theorem \ref{ramcyc4fin}.
\end{proof}


\appendix

\section{Ramification Computations for Cyclic Abelian Surfaces}
\label{9}

The computations presented here provide the 
analogue of Theorem \ref{EllRamComp} for $\OO_E$-linear 
Kisin modules which correspond to abelian surfaces. 

Here we let $k$ be a finite field of characteristic $p>0$ 
and $W = W(k)$ be its ring of Witt vectors. We construct 
the ring $\mathfrak{S} = W[[u]]$ which has a natural 
Frobenius homomorphism $\phi$ extending the Frobenius 
on $W$ by sending $u\mapsto u^p$. We also let 
$\mathfrak{S}_n = W_n[[u]]$, where $W_n$ is the 
$n$th truncation of the Witt construction of $W$; 
$\mathfrak{S}_n$ is likewise equipped with a Frobenius 
which we also denote by $\phi$ extending the Frobenius 
on $W_n$ by $u\mapsto u^p$. We refer the reader to 
Definition \ref{BTDef} for the 
definition of an $\mathfrak{S}$-module and to Proposition \ref{fingroup} 
for their reduction to $\mathfrak{S}_n$-modules.

Let $(E,\Phi)$ be a ($p$-adic) CM pair with 
$E/\QQ_p$ a totally ramified, Galois field of degree $h$ 
and $\Phi$ a ($p$-adic) CM type. In most of this section we 
will let $h = 4$ and $E/\QQ_p$ be cyclic with ($p$-adic) CM type 
$\{1,2\}$. In Definition \ref{OEDef} we introduced the notion of 
an $\OO_E$-linear CM Kisin module, where $\OO_E\subseteq E$ 
is the ring of integers, a definition which is fundamental to 
the computations in this appendix. We invite the reader to 
first look at the 
fundamental structure theorems regarding these modules, 
most namely Proposition \ref{CMMods} and Lemma \ref{smodel}.

The following lemma is fundamental.

\begin{lem}
\label{cyc4}
$\QQ_p$ admits a totally ramified cyclic extension of degree $4$ 
only if $p \equiv 1 \mod 4$ or $p=2$.
\end{lem}
\begin{proof}
The Galois group of a finite, ramified, cyclic extension of $\QQ_p$ 
is isomorphic to a finite quotient of $\ZZ_p^\times$. 
If $p\neq 2$, $\ZZ_p^\times \simeq \mu_{p-1}\times 1+p\ZZ_p$, 
and as $(4,p) = 1$, it follows that this quotient must be 
a quotient of $\mu_{p-1}$.
\end{proof}

\begin{lem}
\label{computation}
Let $\mathfrak{M}$ be an $\mathfrak{S}_1$-module of rank $h$ and 
$\{e_1,\ldots,e_h\}$ be a choice of basis such that 
$$\phi_{\mathfrak{M}}(f_1,\ldots,f_h) = (f_1^p,\ldots,f_h^p)
\left(\begin{array}{ccccc}
a_1&a_2&\cdots&a_{h-1}&a_h\\
0&a_1&\cdots&a_{h-2}&a_{h-1}\\
&&\ddots&\vdots&\vdots\\
\vdots&&&a_1&a_2\\
\\
0&&\cdots&0&a_1\end{array}\right)$$
where the entries satisfy the inequalities
\begin{align*}
\mathrm{deg}_u(a_i)&>\mathrm{deg}_u(a_j),\\
\mathrm{deg}_u (a_i^p a_{k-i+1})  &> \mathrm{deg}_u\left(a_j^pa_{k-j+1}\right),\end{align*}
for all $1\le i<j\le k\le h$. Then, up to $\mathfrak{S}_1^\times$-multiples, 
there are at most $h$ distinct 
saturated lines $\mathfrak{L}\subseteq \mathfrak{M}$ with 
$$v_u(\mathrm{det}(\phi_{\mathfrak{L}}))\in\left\{\mathrm{deg}_u(a_1) + (p-1)\mathrm{deg}_u\left(\frac{a_1}{a_k}\right)\right\}_{1\le k\le h}.$$
\end{lem}
\begin{proof}
Let $\mathfrak{L}\subseteq \mathfrak{M}$ be a 
saturated line with Frobenius $\phi_\mathfrak{L} = u^\mu$. 
By the commutation of the diagram

\bigskip
\centerline{\begin{xy}
(0,15)*+{\mathfrak{L}}="a";
(15,15)*+{\mathfrak{M}}="b";
(0,0)*+{\mathfrak{L}}="c";
(15,0)*+{\mathfrak{M}}="d";
{\ar^{\phi_{\mathfrak{M}}} "b";"d"};{\ar@{^{(}->} "c";"d"};
{\ar@{^{(}->} "a";"b"};{\ar_{u^{\mu} = \phi_\mathfrak{L}} "a";"c"};
\end{xy}}
\bigskip

\noindent an element $v = (f_1,\ldots,f_h)\in \mathfrak{M}$ is 
also in $\mathfrak{L}$ if and only if it satisfies the system of 
equations
\begin{equation}
\label{inline1}
a_kf_1^p + \cdots + a_1f_k^p = u^\mu f_k\;\text{  for  }\; 1\le k\le h.
\end{equation}
We may assume $f_1\neq 0$ (otherwise let $1\le s\le h$ be the 
smallest integer such that $f_s \neq 0$ and rephrase the problem 
with $h$ replaced by $h-s+1$) so that $u^{\mu} = a_1f_1^{p-1}$.

We proceed by induction and show at each step that the following 
three conditions are met:

\begin{enumerate}
\item If $v = (f_1,\ldots,f_h)\in\mathfrak{L}$ then if 
$f_j = 0$ for some $1\le j\le h$, then $f_i = 0$ for all 
$1\le i<j$.
\item If $v = (f_1,\ldots,f_h)\in\mathfrak{L}$ 
then $\mathrm{deg}_u(f_i)\ge \mathrm{deg}_u(f_j)$ for $1\le i< j\le h$. 
In particular, since $\mathfrak{L}\subseteq \mathfrak{M}$ is 
a saturated submodule, we may assume that $f_h = 1$.
\item Compute $v_u(\mathrm{det}(\phi_\mathfrak{L}))$ for all potential 
solutions to the problems above.
\end{enumerate}

For the base case $h=2$, 
(1)--(3) have already been verified explicitly in Proposition 
\ref{lubintate}.

To show the induction step for (1), we assume that $f_h = 0$ and 
by the induction hypothesis 
that $f_{h-1} \neq 0$, otherwise $f_k = 0$ for all $1\le k \le h-1$. 
A solution to the equation
$$a_hf_{1}^p + \cdots + a_2f_{h-1}^p = 0,$$
which is satisfied since $f_{h} = 0$, then exists if and only 
if 
$$\mathrm{deg}_u(a_if_{h-i+1}^p) = \mathrm{deg}_u(a_jf_{h-j+1}^p)$$
for some pairs $i>j$.
Since $f_{h-1}\neq 0$ by hypothesis, $f_k\neq 0$ for some 
$1\le k<h-1$, and hence for all $f_i$ for $k\le i\le h-1$ by 
the induction hypothesis. We may also safely assume $k = 1$, 
for otherwise replace the $h$ equations by $h-k$ equations in 
(\ref{inline}). But then by the induction hypothesis, the equality of 
degrees above becomes either 
$$\mathrm{deg}_u(a_ia_{h-i+1}^p) = \mathrm{deg}_u(a_ja_{h-j+1}^p)$$
for some pairs $i>j$. This contradicts our inequality hypotheses, 
and hence the only solution exists when $f_i = 0$ for all 
$1\le i\le h$.

Similarly to show the induction step for (2), 
suppose the statement holds for $h-1$ and yet fails at $h$, so 
that, in particular, $\mathrm{deg}_u(f_h)>\mathrm{deg}_u(f_{h-1})$. 
By the induction hypothesis, 
we may thus assume that $f_{h-1} = 1$, so that we obtain the equation
$$a_1f_h^p + \cdots + a_hf_1^p = \left(a_{h-1}f_1^{p} + \cdots + a_2f_{h-2}^p + a_1\right)f_h.$$
We may moreover assume by the induction hypothesis that 
$$\mathrm{deg}_u(f_1) - \mathrm{deg}_u(f_k) = \mathrm{deg}_u(a_1) - \mathrm{deg}_u(a_k)$$ 
for all $1\le k\le h-1$ so that there exists a solution to the equation 
if and only if 
$$\mathrm{deg}_u\left(a_1f_h^p\right) = \mathrm{deg}_u(a_{h-1}f_1^pf_h).$$
This last assertion follows from the inequalities of our hypothesis, 
which show that $\mathrm{deg}_u(a_{h-1}f_1^pf_h)$ 
dominates all the excluded terms. Then by our induction hypothesis, 
we may conclude from this that $\mathrm{deg}_u(f_h) = \mathrm{deg}_u(f_1)$. 
But then $a_hf_1^p + \cdots + a_2f_{h-1}^p = 0$, which by the 
induction hypothesis has a solution if and only if $f_i = 0$ for all 
$1\le i\le h-1$.

Finally, for (3) we may now \emph{a priori} assume by (1) and (2) 
that $f_k\neq 0$ for all $1\le k\le h$ and 
$\mathrm{deg}_u(f_i)\ge \mathrm{deg}_u(f_j)$ for 
$i<j$. We can then rewrite the system (\ref{inline1}) as
\begin{equation}
\label{inline}
\left(\frac{f_1}{f_k}\right)^p + \frac{a_{k-1}}{a_k}\left(\frac{f_2}{f_k}\right)^p+\cdots + \frac{a_1}{a_k} = \frac{a_1}{a_k}\left(\frac{f_1}{f_k}\right)^{p-1}\;\text{  for  }\; 1\le k\le h
\end{equation}
and note that the solutions for the first $k$ set of equations 
for $1\le k\le h-1$ is given by the induction hypothesis. 
Thus, for all $1\le k\le h-1$,
$$\mathrm{deg}_u\left(f_k\right) = \mathrm{deg}_u\left(\frac{f_k}{f_{h-1}}\right) + \mathrm{deg}_u\left(f_{h-1}\right)$$
and 
$v = \left(\frac{f_1}{f_{h-1}},\ldots,\frac{f_{h-2}}{f_{h-1}}\right)$ is 
a solution for the system (\ref{inline}) on 
$h-1$ equations when for all $1\le k\le h-1$,
$$\mathrm{deg}_u\left(\frac{f_k}{f_{h-1}}\right) = \mathrm{deg}_u\left(\frac{a_k}{a_{h-1}}\right).$$
This has a solution if and only if
$$\mathrm{max}_{1\le k\le h}\left\{\mathrm{deg}_u\left(\frac{a_{h-k+1}}{a_h}\right)\left(\frac{a_k}{a_{h-1}}\right)^p + \mathrm{deg}_u\left(\frac{f_{h-1}}{f_{h}}\right)^p\right\} = \mathrm{deg}_u\left[\frac{a_1}{a_h}\left(\frac{f_1}{f_h}\right)^{p-1}\right].$$
By our hypothesis that $\mathrm{deg}_u(a_i^pa_{h-i+1})>\mathrm{deg}_u(a_j^pa_{h-j+1})$ for $1\le i<j\le h$, the unique solution (if it exists) is 
given precisely when the maximum on the left is attained by $k = 1$. 
Since $\mathfrak{L}$ is saturated, we may assume by (2) that 
$f_h = 1$. This provides solutions 
for $f_k$ and $\mu$ uniquely, for all $1\le k\le h$.
\end{proof}

We recall here Example \ref{ramex}, which provides
for $E/\QQ_p$ any Galois and totally ramified field 
of degree $h$ the representation of 
the (reduced) Kisin module associated to a CM $p$-divisible group $\mathcal{G}$ 
with ($p$-adic) CM by $(\OO_E,\Phi)$. For this, we choose 
the $\mathfrak{S}_1$-basis $\{\pi^0,\ldots,\pi^{h-1}\}$, where 
$\pi$ is a uniformizer for $E$, so that
$$\phi_{\mathfrak{M}_{\mathfrak{S}_1(\mathcal{G}[p])}}(e_i) = \sum_{j=0}^{h-1-i}[P_j(u)]e_{j+i}$$
where the polynomials $P_j(u)$ are precisely defined in Example \ref{ramex}. 
We note the important property that the degree of $P_j(u)$ for $j>1$ 
is strictly less than the degree of $P_1(u)$.

\begin{prop}
\label{ramlines}
Assume that $p\ge h$ and $E/\QQ_p$ is totally ramified and 
cyclic of degree $4$, and let 
$\mathcal{G}/\OO_K$ be an $\OO_E$-linear CM 
$p$-divisible group of type $(\OO_E,\Phi)$. Denote by 
$\mathfrak{M}:=\mathfrak{M}_{\mathfrak{S}_1}(\mathcal{G})$ 
and assume that $\sum_{i\in \Phi^c} c_i\neq 0$ where 
$c_i$ are the units in Example \ref{ramex}. 
Then for a saturated $\mathfrak{S}_1$-line 
$\mathfrak{L}\subseteq \mathfrak{M}$, 
$$v_u(\det(\phi_{\mathfrak{L}}))\in \left\{[(h-d+p-1)(p^h-p^{h-1}) - (p-1)(p^h - 2p^{h-1} + p^{h-k})]e
\right\}.$$
\end{prop}
\begin{proof}
This follows if we can apply Lemma \ref{computation} to 
the presentation of $\mathfrak{M}_{\mathfrak{S}_1}(\mathcal{G})$ 
of Example \ref{ramex}. We therefore check the conditions of 
the hypotheses, namely the conditions on the 
degree of the polynomials which 
appear in the entries of $\phi_{\mathfrak{M}_{\mathfrak{S}_1}(\mathcal{G})}$.

Recall that by Lemma \ref{smodel} and Remark \ref{LT} we can express the 
polynomials $u^{\frac{e}{h}} + \pi h_i(u)$ (here we 
absorb the unit $c_i$ into our choice of uniformizer $\pi$) 
by using the $h$th iteration of the Lubin-Tate polynomial. By the 
hypothesis that $p>h$, we may therefore write
\begin{align*}
u^{\frac{e}{h}}+\pi h_i(u)&\equiv\pi + (\pi^{h-1}u +\pi^{h-2}u^p + \cdots  + u^{p^{h-1}})^{p-1}\mod p\\
&\equiv\pi + \sum_{i_1 = 0}^{p-1}\left(\pi^{i_1}(u^{p^{h-1}})^{p-1-i_1}\cdot\sum_{i_2 = 0}^{i_1}\left(\pi^{i_2}(u^{p^{h-2}})^{i_1-i_2}\cdot\left(\sum_{i_3 = 0}^{i_2} \pi^{i_3}(u^{p^{h-3}})^{i_2-i_3}\cdot\right.\right.\right.\\
&\quad\quad\quad\quad\quad\quad\quad\quad\quad\left.\left.\left.\cdots\cdot\sum_{i_{h-1}=0}^{i_{h}}\pi^{i_{h-1}}(u^p)^{i_{h-2}-i_{h-1}}u^{i_{h-1}}\right)\right)\right)\mod \pi^h.\end{align*}
Each monomial 
comprising $u^{\frac{e}{h}}+\pi h_i(u)$ is thus identified by a 
tuple $(i_1,\ldots,i_{h-1})$ such that $i_1\ge i_2\ge \cdots \ge i_{h-1}$. 
Moreover, if $\sum_{j=1}^{j-1}i_j = k$, the coefficient 
of the corresponding monomial (up to unit multiplication) is $\pi^k$. 
One readily checks that 
the maximal degree a monomial with coefficient $\pi^k$ takes is 
$p^h - 2p^{h-1} + p^{h-k}$, which 
occurs precisely when $i_1 = i_k=1$ and $i_{k+1} = i_{h-1} = 0$.

By the hypothesis that $\sum_{i\in \Phi^c}c_i\neq 0$, the 
previous computation gives
$$\mathrm{deg}_u(a_k) = (p^h - 2p^{h-1} + p^{h-k}) + (h-d-1)(p^{h} - p^{h-1}).$$
Then both hypotheses on the degree of the entries $a_k$ in 
Lemma \ref{computation} are satisfied, 
and we compute the potential degrees from these 
values for $\mathrm{deg}_u(a_k)$.
\end{proof}

When $h=4$, the case of interest for an abelian surface, 
this proposition gives for all $p>4$,
$$v_u(\mathrm{det}(\phi_{\mathfrak{L}}))\in \{2p^4 - 2p^3,\;3p^4 - 4p^3 + p^2,\;3p^4 - 3p^3 - p^2 + p,\;3p^4 - 3p^3 - p+1\}.$$
By Lemma \ref{cyc4} this covers solutions for all potentially 
ramified primes except $p = 2$. We 
consider this separately below.

\begin{lem}
\label{ramline2}
Let $E/\QQ_2$ be totally ramified and cyclic of degree $h=4$, and let 
$\mathcal{G}/\OO_K$ be an $\OO_E$-linear CM $p$-divisible group 
of type $(\OO_E,\Phi)$ with $\Phi = \{1,2\}$. Denote by 
$\mathfrak{M}:=\mathfrak{M}_{\mathfrak{S}_1}(\mathcal{G})$. 
Then for a saturated $\mathfrak{S}_1$-line 
$\mathfrak{L}\subseteq \mathfrak{M}$,
$$v_u(\mathrm{det}(\phi_{\mathfrak{L}}))\in \{2(2^4-2^3),2(2^4-2^3) + 2^3\}.$$
\end{lem}
\begin{proof}
By Example \ref{ramex}, we can choose in each case the 
basis $(1,\pi,\pi^2,\pi^3)$ for a choice of uniformizer 
$\pi$ of $E$ such that
$$\phi_{\mathfrak{M}}(f_1,f_2,f_3,f_4) = (f_1^p,f_2^p,f_3^p,f_4^p)\left(\begin{array}{cccc}
a_1&a_2&a_3&a_4\\
&a_1&a_2&a_3\\
&&a_1&a_2\\
&&&a_1\end{array}\right).$$

By Lemma \ref{smodel} and 
Remark \ref{LT}, we can express the polynomials 
$u^{\frac{e}{4}} + c_i\pi h_i(u)$ as
$$u^{\frac{e}{4}} + c_i\pi h_i(u)\equiv u^8 + c_i\pi(u^4 + 1) + c_i^2\pi^2(u^4 + u^2) + c_i^3\pi^3(u^2 + u)\mod \pi^4$$
so that
$$a_1 = u^{16},\quad a_2 = 0,\quad a_3= u^8+1, \quad a_4 = 0.$$
We then compute that the only solutions to this system 
of equations are
\begin{align*}
(0,0,x,y)&\quad\quad \phi_{\mathfrak{L}} = u^{16},\\
(u^8x,u^8y,x,y)&\quad\quad \phi_{\mathfrak{L}} = u^{24},
\end{align*}
where $x,y\in \mathbb{F}_2$ are not both $0$.
\end{proof}

Now we consider the situation of submodules of higher rank. 
Saturated submodules of rank $h-k$ 
in a Kisin module $\mathfrak{M}$ of rank $h$ correspond to 
saturated lines in $\bigwedge^{h-k}\mathfrak{M}$. Moreover, since 
$\phi_{\bigwedge^{h-k}\mathfrak{M}} = \bigwedge^{h-k}\phi_{\mathfrak{M}}$, 
we can compute these modules by taking partial determinants.

Specializing to $h = 4$, we invoke Example \ref{ramex} 
to write a representation of the module 
$\mathfrak{M} := \mathfrak{M}_{\mathfrak{S}_1}(\mathcal{G})$ on 
the explicit basis $\{1,\pi,\pi^2,\pi^3\}$ for a choice of 
uniformizer $\pi$ of $E$. We make the assumption that 
$E/\QQ_p$ is cyclic and the ($p$-adic) CM type is $\Phi = \{1,2\}$. Then

$$\phi_{\bigwedge^2\mathfrak{M}}=\left(\begin{array}{cccccc}
a_1^2&a_1a_2&a_2^2-a_1a_3&a_1a_3&a_2a_3-a_1a_4&a_3^2-a_2a_4\\
&a_1^2&a_1a_2&a_1a_2&a_2^2&a_2a_3-a_1a_4\\
&&a_1^2&0&a_1a_2&a_2^2-a_1a_3\\
&&&a_1^2&a_1a_2&a_1a_3\\
&&&&a_1^2&a_1a_2\\
&&&&&a_1^2
\end{array}\right)$$
and

$$\phi_{\bigwedge^3\mathfrak{M}}=\left(\begin{array}{cccc}
a_1^3&a_1^2a_2&a_1a_2^2 - a_1^2a_3&a_2^3 - 2 a_1a_2a_3+a_1^2a_4\\
&a_1^3&a_1^2a_2&a_1a_2^2 - a_1^2a_3\\
&&a_1^3&a_1^2a_2\\
&&&a_1^3
\end{array}\right),$$
where $a_i = P_i(u)$ (defined in Example \ref{ramex}) 
and we recall from Proposition \ref{ramlines} that
$$\mathrm{deg}_u(P_i(u)) = 2p^4 - 3p^3 + p^{4-i}$$
when $p>4$.

\begin{prop}
\label{rammods}
Assume that $p >4$ and $E/\QQ_p$ is totally ramified and cyclic 
of degree $4$, and let 
$\mathcal{G}/\OO_K$ be an $\OO_E$-linear CM 
$p$-divisible group of type $(\OO_E,\Phi)$. Denote by 
$\mathfrak{M}:=\mathfrak{M}_{\mathfrak{S}_1}(\mathcal{G})$ 
and assume that $\sum_{i\in \Phi^c} c_i\neq 0$ where 
$c_i$ are the units in Example \ref{ramex}. 
Then:
\begin{enumerate}
\item For a saturated $\mathfrak{S}_1$-line 
$\mathfrak{L}\subseteq \bigwedge^2\mathfrak{M}$, 
\begin{align*}
v_u(\det(\phi_{\mathfrak{L}}))\in &\left\{4p^4-4p^3,\;5p^4-6p^3+ p^2,\;5p^4-5p^3 - p^2 + p,\;5p^4-5p^3-p+1,\right.\\
&\left.\:\:6p^4-8p^3+2p^2,\;6p^4-7p^3 + p^2-p  + 1
\right\}.\end{align*}
\item For a saturated $\mathfrak{S}_1$-line 
$\mathfrak{L}\subseteq \bigwedge^3\mathfrak{M}$, 
$$v_u(\det(\phi_{\mathfrak{L}}))\in \left\{6p^4 - 6p^3, \;
7p^4 - 8p^3 +p^2,\;7p^4 - 7p^3 - p^2 +p,\; 7p^4 - 7p^3 - p + 1\right\}.$$
\end{enumerate}
\end{prop}
\begin{proof}
For (1), let $v = (f_1,\ldots,f_6)$ generate $\mathfrak{L}$. 
Then it satisfies the system of equations
\begin{align}
\label{rank2sys}
\begin{split}
a_1^2f_1^p &= u^\mu f_1,\\
a_1a_2f_1^p + a_1^2f_2^p &= u^\mu f_2,\\
(a_2^2 - a_1a_3)f_1^p + a_1a_2f_2^p + a_1^2f_3^p &=u^\mu f_3,\\
a_1a_3 f_1^p+ a_1a_2f_2^p + a_1^2f_4^p &= u^\mu f_4,\\
(a_2 a_3 - a_1a_4)f_1^p + a_2^2f_2^p + a_1a_2 f_3^p + a_1a_2f_4^p + a_1^2f_5^p &=u^\mu f_5,\\
(a_3^2 - a_2a_4)f_1^p + (a_2a_3 - a_1a_4)f_2^p + (a_2^2 - a_1a_3)f_3^p + a_1a_3f_4^p + a_1a_2f_5^p + a_1^2 f_6^p&=u^\mu f_6.
\end{split}
\end{align}

If $f_1 = f_2 = f_3 = 0$ or $f_1 = f_2 = f_4 = 0$, the argument 
from Proposition \ref{ramlines} shows that the resulting system 
of equations satisfies the conditions of Lemma \ref{computation} 
and we obtain a corresponding set of solutions
$$v_u(\mathrm{det}(\phi_{\mathfrak{L}}))\in \left\{2\mathrm{deg}_u(a_1) + (p-1)\left(\mathrm{deg}_u(a_1)-\mathrm{deg}_u(a_k)\right)\right\}_{1\le k\le 3}.$$ 

If $f_1 = f_2 = 0$ and $f_3,f_4\neq 0$, then by the third and fourth 
equations in the system (\ref{rank2sys}) we 
have $f_3 = f_4$, and making this substitution the remaining 
system satisfies the conditions of Lemma 
\ref{computation}. This gives the corresponding set of solutions
\begin{align*}
v_u(\mathrm{det}(\phi_{\mathfrak{L}}))\in &\left\{2\mathrm{deg}_u(a_1),\;2\mathrm{deg}_u(a_1) + (p-1)\left(\mathrm{deg}_u(a_1)-\mathrm{deg}_u(a_2)\right),\;\right.\\
&\:\:\left.2\mathrm{deg}_u(a_1) + 2(p-1)\left(\mathrm{deg}_u(a_1)-\mathrm{deg}_u(a_2)\right)\right\}.
\end{align*}

If $f_1 = 0$ and $f_2\neq 0$, then again $f_3 = f_4$ by the third and 
fourth equations 
in the system (\ref{rank2sys}). However, 
the remaining system does not satisfy Lemma \ref{computation}, 
as $\mathrm{deg}_u(P_1(u)P_4(u))>\mathrm{deg}_u(P_2^2(u))$. 
We note that if $f_2\neq 0$ then $f_i\neq 0$ for $i>2$. 
Moreover, using the second through fifth equations of the system 
(\ref{rank2sys}), the proof of Lemma \ref{computation} applies to 
show that $\mathrm{deg}_u(f_2)>\mathrm{deg}_u(f_3)>\mathrm{deg}_u(f_4)>\mathrm{deg}_u(f_5)$. Hence, 
up to multiplication by a unit, either $f_5 = 1$ or 
$f_6 = 1$. These correspond to the solutions
\begin{align*}v_u(\mathrm{det}(\phi_{\mathfrak{L}}))\in &\left\{2\mathrm{deg}_u(a_1) + 2(p-1)\left(\mathrm{deg}_u(a_1)-\mathrm{deg}_u(a_2)\right),\right.\\
&\left.\:\:2\mathrm{deg}_u(a_1) + (p-1)\left(\mathrm{deg}_u(a_1)-\mathrm{deg}_u(a_4)\right)\right\}.\end{align*}

Finally, if $f_1\neq 0$, then it still follows that 
$\mathrm{deg}_u(f_3) = \mathrm{deg}_u(f_4)$ by the third 
equation in the system (\ref{rank2sys}). 
Moreover, $f_i\neq 0$ for all $1\le i\le 6$. While again 
(\ref{rank2sys}) does not satisfy Lemma \ref{computation}, 
the proof of \emph{loc. cit.} demonstrates that 
$\mathrm{deg}_u(f_1)>\mathrm{deg}_u(f_2)>\mathrm{deg}_u(f_3)>\mathrm{deg}_u(f_4)>\mathrm{deg}_u(f_5)$ by 
analyzing the first five equations in (\ref{rank2sys}). Hence either 
$f_5 = 1$ or $f_6 = 1$, which correspond to the solutions
\begin{align*}v_u(\mathrm{det}(\phi_{\mathfrak{L}}))\in &\left\{2\mathrm{deg}_u(a_1) + (p-1)\left(\mathrm{deg}_u(a_1)-\mathrm{deg}_u(a_4)\right),\right.\\
&\left.\:\:2\mathrm{deg}_u(a_1) + (p-1)\left(2\mathrm{deg}_u(a_1)-\mathrm{deg}_u(a_2)-\mathrm{deg}_u(a_4)\right)\right\}.\end{align*}

For (2), the argument from 
Proposition \ref{ramlines} on the degrees of $a_i = P_i(u)$ 
shows that the entries of 
$\phi_{\bigwedge^3 \mathfrak{M}}$ satisfy Lemma \ref{computation}. 
Thus, after a simplification using the known degrees 
of the entries $a_i = P_i(u)$, we obtain
$$v_u(\det(\phi_{\mathfrak{L}}))\in \left\{3\mathrm{deg}_u(a_1)+(p-1)\left(\mathrm{deg}_u(a_1)-\mathrm{deg}_u(a_k)\right)\right\}_{1\le k\le 4}.$$
\end{proof}

By Lemma \ref{cyc4}, we again cover by this proposition all 
potentially ramified primes except $p=2$. We again consider this 
separately below.

\begin{lem}
\label{rammod2}
Let $E/\QQ_2$ be totally ramified and cyclic of degree $h=4$, 
and let $\mathcal{G}/\OO_K$ be an $\OO_E$-linear CM 
$p$-divisible group of type $(\OO_E,\Phi)$. Denote by 
$\mathfrak{M}:=\mathfrak{M}_{\mathfrak{S}_1}(\mathcal{G})$. 
Then:
\begin{enumerate}
\item For a saturated $\mathfrak{S}_1$-line $\mathfrak{L}\subseteq\bigwedge^2\mathfrak{M}$,
$$v_u(\mathrm{det}(\phi_{\mathfrak{L}})\in\{4(2^4-2^3),\;4(2^4-2^3)+2^3\}.$$
\item For a saturated $\mathfrak{S}_1$-line $\mathfrak{L}\subseteq \bigwedge^3\mathfrak{M}$,
$$v_u(\mathrm{det}(\phi_{\mathfrak{L}})\in\{6(2^4-2^3),\;6(2^4-2^3) + 2^3\}.$$
\end{enumerate}
\end{lem}
\begin{proof}
By the computation of $P_i(u)$ in the proof of Lemma \ref{ramline2},
$$\phi_{\bigwedge^2\mathfrak{M}}=\left(\begin{array}{cccccc}
u^{32}&0&u^{16}(u^8+1)&u^{16}(u^8+1)&0&u^{16}+1\\
&u^{32}&0&0&0&0\\
&&u^{32}&0&0&u^{16}(u^8+1)\\
&&&u^{32}&0&u^{16}(u^8+1)\\
&&&&u^{32}&0\\
&&&&&u^{32}
\end{array}\right)$$
and

$$\phi_{\bigwedge^3\mathfrak{M}}=\left(\begin{array}{cccc}
u^{48}&0& u^{32}(u^8+1)&0\\
&u^{48}&0& u^{32}(u^8+1)\\
&&u^{48}&0\\
&&&u^{48}
\end{array}\right).$$

For (1), let $v = (f_1,\ldots,f_6)\in \mathfrak{M}$ generate 
$\mathfrak{L}$. Then we easily compute 
that all saturated lines are of the form
\begin{align*}
(0,0,u^8,0,0,1),\;(0,0,0,u^8,0,1) &\quad\quad \phi_{\mathfrak{L}} = u^{40}\\
(0,x,w,w,y,z)&\quad\quad \phi_{\mathfrak{L}} = u^{32},
\end{align*}
where $x,y,z,w\in \mathbb{F}_2$ are not all $0$.

For (2), let $v = (f_1,\ldots,f_4)$ generate 
$\mathfrak{L}$. Then again we easily compute that 
all saturated lines are of the form
\begin{align*}
(0,0,x,y)&\quad\quad \phi_{\mathfrak{L}} = u^{48}\\
(u^8x,u^8y,x,y)&\quad\quad\phi_{\mathfrak{L}} = u^{56},\end{align*}
where $x,y\in \mathbb{F}_2$ are not both $0$.
\end{proof}

\begin{thm}
\label{ramcyc4fin}
Assume that $[E:\QQ_p] = 4$ and that $p$ is totally 
ramified in $E$, and let $\mathcal{G}/\OO_K$ be 
an $\OO_E$-linear CM $p$-divisible group. Let 
$\mathcal{H}\subseteq \mathcal{G}[p^n]$ be a subgroup such 
that $\mathcal{H}\cap\mathcal{G}[\pi_E^r] = 1$ for all 
$r\ge 0$, and characterized by the non-increasing tuple 
$(\lambda_1,\lambda_2,\lambda_3)$ satisfying 
$\lambda_1 + \lambda_2 + \lambda_3 = \log \#\mathcal{H}$, 
$\mathcal{H}$ is $p^{\lambda_1}$-torsion, and the $p$-height 
$\iota\in \{1,2,3\}$ of $\mathcal{H}$ is the largest integer 
such that $\lambda_{\iota}\neq 0$. 
Then
$$\#\frac{1}{[K:\QQ_p]}\log s^*\Omega^4_{\mathcal{H}/\OO_K} = \frac{1}{4p^4-4p^3}\left(\sum_{i=\lambda_2 + 1}^{\lambda_1}\frac{R_1(p)}{p^i}\right)\left(\sum_{j=\lambda_3+1}^{\lambda_2}\frac{R_2(p)}{p^j}\right)\left(\sum_{k=1}^{\lambda_3}\frac{R_3(p)}{p^k}\right)\log p,$$
where the sums are evaluated if and only if the difference in bounds is at 
least zero, and for $p\equiv 1\mod 4$,
$$R_1(p) = \left\{\begin{array}{l}
p^4  -p^2\\
p^4 - p^3 + p^2 - p\\
p^4 - p^3 + p-1\end{array}\right.$$
$$R_2(p) = \left\{\begin{array}{l}
3p^4 -2p^3 -p^2\\
3p^4 -3p^3 + p^2 - p\\
3p^4 - 3p^3 + p-1\\
2p^4-2p^2\\
2p^4 - p^3 - p^2 + p - 1
\end{array}\right.$$
$$R_3(p) = \left\{\begin{array}{l}
5p^4 - 4p^3 -p^2\\
5p^4 - 5p^3 + p^2 - p\\
5p^4 - 5p^3 + p-1\end{array}\right.$$
and for $p=2$,
$$R_1(2) = 2(2^4-2^3) - 2^3$$
$$R_2(2) = 4(2^4-2^3) - 2^3$$
$$R_3(2) = 6(2^4-2^3) - 2^3.$$
\end{thm}
\begin{proof}
This compiles the computations in Proposition \ref{ramlines}, 
Lemma \ref{ramline2}, Proposition \ref{rammods}, and 
Lemma \ref{rammod2} using Proposition \ref{deg}.
\end{proof}

\pagebreak

\end{document}